\documentclass[11pt]{article}
\usepackage[a4paper,left=1in,right=1in,top=1in,bottom=1in]{geometry}

\usepackage{graphicx}
\graphicspath{{figures/} {figures/petal/}}
\usepackage{caption}
\usepackage[english]{babel}
\usepackage[utf8]{inputenc}

\usepackage{amssymb}
\usepackage{amsmath}
\usepackage{amsthm}
\usepackage{amsfonts}
\usepackage{hyperref}
\usepackage{cases}

\usepackage[nameinlink]{cleveref}
\hypersetup{
	pdftitle={Persistent Homology of the Signed Distance and Morse Theory},
	colorlinks=true,
	linkcolor={Maroon},
	filecolor={Maroon},
	citecolor={Maroon},
	urlcolor={blue}
}

\PassOptionsToPackage{dvipsnames,svgnames,x11names}{xcolor}
\usepackage{xcolor}

\usepackage[textsize=footnotesize]{todonotes}
\usepackage{float}
\usepackage{multirow}
\usepackage{parskip}
\usepackage{enumitem}
\usepackage{comment}
\usepackage{thmtools, thm-restate}

\usepackage{blkarray, bigstrut}

\usepackage{tabu}
\usepackage{diagbox}
\newcommand{\specialcell}[2][c]{%
	\begin{tabular}[#1]{@{}c@{}}#2\end{tabular}}

\usepackage{tikz-cd} 
\usepackage{amscd} 

\usepackage{algorithm}
\usepackage[noend]{algpseudocode}

\usepackage[]{natbib}
\usepackage{authblk}
\usepackage{mathtools}
\newcommand{\R}{\mathbb{R}}

\newcommand{\Sbb}{\mathbb{S}}

\newcommand{\Nbb}{\mathbb{N}}

\newcommand{\io}{\iota}

\newcommand{\nbf}{\mathbf{n}}

\newcommand{\PHf}[2]{\mathrm{PH}_{#1}\left( {#2} \right)}
\newcommand{\Barc}[2]{\mathrm{Bar}_{#1}\left( {#2} \right)}

\renewcommand{\Bar}{\mathrm{Bar}}
\newcommand{\Dgm}{\mathrm{Dgm}}
\newcommand{\PH}{\mathrm{PH}}
\newcommand{\birth}{\mathrm{birth} \,}
\newcommand{\death}{\mathrm{death} \,}

\newcommand{\sgn}{\mathrm{sgn}}

\newcommand{\cutlocus}[1]{\overline{\mathcal{M}_{#1}}}

\newcommand{\W}{\Omega}

\newcommand{\eps}{\epsilon}
\newcommand{\kap}{\kappa}
\newcommand{\kapmax}{\kappa_{\text{max}}}
\newcommand{\kapmin}{\kappa_{\text{min}}}

\newcommand{\Kb}{K_\bullet}
\newcommand{\Xb}{X_\bullet}

\newcommand{\bord}{\partial}
\newcommand{\inter}{\cap}
\newcommand{\union}{\cup}

\newcommand{\Emb}{\mathrm{Emb}^k(M,\R^3)} 

\newcommand{\idx}{\mathrm{index}}
\newcommand{\dist}{\mathrm{dist}}
\newcommand{\med}{\mathcal{M}}

\newcommand{\Surf}{\mathcal{S}}

\newcommand{\grad}{\nabla}
\newcommand{\Hess}{\mathrm{Hess}}
\newcommand{\Id}{\mathrm{Id}}

\newcommand{\level}[2]{{#1}^{-1}\left({#2}\right)}
\newcommand{\sublevel}[2]{{#1}^{-1}\left(-\infty, {#2}\right]}

\newcommand{\intlevel}[3]{{#1}^{-1}\left[{#2}, {#3}\right]}

\newcommand{\SW}{\mathrm{SW}}
\newcommand{\NW}{\mathrm{NW}}
\newcommand{\NE}{\mathrm{NE}}

\newcommand{\Ker}{\mathrm{Ker} \,}
\renewcommand{\Im}{\mathrm{Im} \,}

\newcommand{\rank}{\mathrm{rank} \,}
\renewcommand{\H}{\mathrm{H}}

\newcommand\topstrut[1][1.2ex]{\setlength\bigstrutjot{#1}{\bigstrut[t]}}
\newcommand\botstrut[1][1.2ex]{\setlength\bigstrutjot{#1}{\bigstrut[b]}}

\theoremstyle{plain}
\newtheorem{thm}{Theorem}
\newtheorem{lem}{Lemma}
\newtheorem{prop}{Proposition}
\newtheorem{cor}{Corollary}

\theoremstyle{definition}
\newtheorem{defi}{Definition}

\theoremstyle{remark}
\newtheorem{rmk}{Remark}

\providecommand{\keywords}[1]{\textbf{\textit{Keywords---}} #1}

\title{Generalized Morse Theory of Distance Functions to Surfaces\\ for Persistent Homology}

\author[1,2]{Anna Song\thanks{a.song19@imperial.ac.uk}}
\author[3]{Ka Man Yim\thanks{YimKM@cardiff.ac.uk}}
\author[1]{Anthea Monod\thanks{a.monod@imperial.ac.uk}}

\affil[1]{\small Department of Mathematics, Imperial College London, London, UK}
\affil[2]{\small Haematopoietic Stem Cell Laboratory, The Francis Crick Institute, London, UK}
\affil[3]{\small School of Mathematics, Cardiff University, Cardiff, UK}

\begin{document}

	\maketitle
	
	\begin{abstract}
		This paper brings together three distinct theories with the goal of quantifying shape textures with complex morphologies. Distance fields are central objects in shape representation, while topological data analysis uses algebraic topology to characterize geometric and topological patterns in shapes. The most well-known and widely applied tool from this approach is persistent homology, which tracks the evolution of topological features in a dynamic manner as a barcode.  Morse theory is a framework from differential topology that studies critical points of functions on manifolds; it has been used to characterize the birth and death of persistent homology features. However, a significant limitation to Morse theory is that it cannot be readily applied to distance functions because distance functions lack smoothness, which is required in Morse theory. Our contributions to addressing this issue is two fold. First, we generalize Morse theory to Euclidean distance functions of bounded sets with smooth boundaries. We focus in particular on distance fields for shape representation and we study the persistent homology of shape textures using a sublevel set filtration induced by the signed distance function. We use transversality theory to prove that for generic embeddings of a smooth compact surface in $\R^3$, signed distance functions admit finitely many non-degenerate critical points.  This gives rise to our second contribution, which is that shapes and textures can both now be quantified and rigorously characterized in the language of persistent homology: signed distance persistence modules of generic shapes admit a finite barcode decomposition whose birth and death points can be classified and described geometrically. We use this approach to quantify shape textures on both simulated data and real vascular data from biology.
	\end{abstract}
	
	
	\vfill\eject

	\tableofcontents

	\section{Introduction} \label{sec:intro}
	
	Information on the shape and texture of objects captured in images can provide valuable insight to the context of study.  For example, in a study to predict the survival of patients diagnosed with glioblastoma multiforme---an extremely aggressive brain cancer---it has been shown that both the shape and texture of tumor images is even more informative in the prediction task than gene expression \citep{crawford2020predicting}, which is a type of molecular data obtained from surgery following biopsy and sequencing and is a mainstream type of data used to study cancer.  However, the task of feasibly quantifying shape and texture information for computational, statistical, and machine learning tasks in the most informative yet interpretable manner remains challenging.  This is especially true when considering complex morphologies entailing porosity and branching behavior, which are not easily represented by meshes \citep{song_generation_2022}. Quantifying shape and texture in an interpretable manner is the driving motivation and practical task of interest in our work.  To answer this practical question, we contribute novel, fundamental theoretical results in differential topology with respect to distance functions.
	
	Distance functions (or fields) are fundamental objects in geometry and topological data analysis, \citep{osher_signed_2003, lieutier_any_2003, chazal_stability_2004, attali_stability_2009}, PDEs \citep{albano_singular_2013}, as well as non-smooth analysis and singularity theory \citep{cheeger_critical_1991,birbrair_medial_2017}. 
	They have become popular tools to represent geometry in multiple fields of applications, ranging from computer graphics to shape analysis \citep{lee_medial_1982, lindquist_medial_1996, sigg_signed_2003, park_deepsdf_2019}.
	Given a subset $A \subset \R^n$, the unsigned distance function assigns to each point in space its distance to $A$. The \emph{signed} distance function for a set $A$ with non-empty interior is defined by modifying the distance function to $\partial A$, distinguishing points in the interior from those outside by attributing a negative sign or a positive sign respectively. Signed distance fields are useful in shape representation and analysis \citep{osher_signed_2003}; since an object can be implicitly represented as the zero sublevel set of its corresponding signed distance field, an equivalent amount of geometric and topological information is contained in both the shape and the field. Signed distance fields are useful tools to represent complex porous and branching structures, which is a particular application interest in this work. 
	
	In this paper, we focus on a method combining signed distance fields and persistent homology---\emph{signed distance persistent homology} (SDPH) \citep{delgado-friedrichs_morse_2014, delgado-friedrichs_skeletonization_2015, herring_topological_2019, moon_statistical_2019}---to rigorously define and quantify the texture of materials. Persistent homology is the leading algebraic method in topological data analysis that measures multiscale features in data and fundamentally depends on a distance function \citep{edelsbrunner_persistent_2008,ghrist_barcodes_2008,carlsson_topology_2009}.  By measuring the scale at which components, cycles, voids, appear or disappear, persistent homology intrinsically captures heuristic geometric notions such as ``shape'' and ``size'' as topological (specifically, homological) features. The lifetimes and algebraic relations between these features are summarized in a \emph{barcode} or \textit{persistence diagram}, a collection of birth--death intervals where each interval corresponds to a topological feature. Moreover, barcodes can be further processed as features in statistical analysis, using either a variety of vectorization methods available or methodology developed to intake persistence diagrams.
	
	For the case of SDPH barcodes, there is a need for a precise interpretation of birth and death events in terms of the geometry of the shape from which they are derived. This lack of interpretability is a major obstacle to using SDPH in applications where interpretability and geometric understanding are key. To achieve interpretability of birth and death events of the barcode, we appeal to \emph{Morse theory} \citep{milnor_morse_1963}, which is a framework relating the gradient dynamics of a smooth function to the topology of the underlying manifold.  For smooth functions with only non-degenerate critical points---\emph{Morse functions}---Morse theory implies that birth and death values in their persistence diagrams correspond to critical values of the function. Furthermore, critical points of a Morse function are classified by their \emph{Morse index}, $\lambda$: crossing a critical point with index $\lambda$ is topologically equivalent to attaching a $\lambda$-dimensional ``handle,'' which in terms of persistence, corresponds to the  creation or destruction of a homology generator, i.e., endpoints of a birth--death interval. Moreover, smooth functions are generically Morse, in the sense that Morse functions form a dense subset of smooth functions \cite[Chapter 6]{hirsch_differential_1976}.
	
	An immediate impasse to applying Morse theory to achieve interpretability in our setting is that distance functions are not smooth.  This leads to our main theoretical contribution, which is a generalization of Morse theory of smooth functions to the class of general Euclidean distance functions $f$ to smooth compact boundaries in $\mathbb{R}^n$.  In particular, our results also adapt to the case of signed distance functions and we leverage these results to systematically interpret SDPH diagrams, giving rise to an interpretable quantifier of both shape and texture of complex morphologies.

	\textbf{Contributions.}
	We generalize Morse theory to Euclidean distance functions (signed or unsigned) of smooth compact boundaries in $\R^n$. 
	To do this, we first show that the fundamental results of smooth Morse theory---the isotopy and handle attachment lemmas---can also be generalized to topological Morse functions \citep{morse_topologically_1959}. Such functions $f$ admit either a regular form (up to a homeomorphism) at topological regular points,
	\[f \underset{\mathrm{homeo}}{\sim} \mathrm{const.} + x_n,\] 
	or a normal form at topological critical points with index $\lambda$,
	\[ f \underset{\mathrm{homeo}}{\sim} \mathrm{const.} - \sum_{i = 1}^{\lambda} x_i^2 + \sum_{i = \lambda + 1}^{n} x_i^2.\]
	Similar to the smooth case, we obtain the topological counterparts of the two fundamental Morse lemmas, the isotopy lemma (\Cref{thm:topological_isotopy_lemma}) and the handle attachment lemma (\Cref{thm:topological_handle_lemma}). They state that the topology of the sublevel sets changes exactly when the level crosses a critical value, at which a $\lambda$-cell is attached.

	Distance functions are, in fact, special cases of topological Morse functions, under suitable geometric assumptions on the surface that we provide in \Cref{lem:dist_as_min_type}. The essential idea is to express distance functions locally as Min-type functions \citep{gershkovich_morse_1997}, for which the notion of non-degenerate critical points is well-defined and provides a topological normal form to distance functions (\Cref{thm:normal_form_dist}). This sequence of ideas can be figuratively summarized as
	\[\text{distance functions} \xrightarrow[\text{conditions}]{\text{geometric}} \text{Min-type functions} \xrightarrow[\text{form}]{\text{normal}} \text{topological Morse functions}.\]
	
	As a key result, we prove that signed distance functions of generically embedded compact surfaces in $\R^3$ are topological Morse functions with finitely many critical points (\Cref{thm:generic}). We use transversality theory \citep{abraham_transversal_1967} to prove this result which echoes the fact that smooth functions are generically Morse.

	This theoretical investigation has important consequences on the interpretation of SDPH in applications to data, which we present as our applied contribution. Our previous theoretical contributions imply that, for generic shapes in $\R^3$, the signed distance persistence modules are tame and thus admit a finite barcode decomposition. Furthermore, we give a geometric description of the Min-type index of critical points and obtain a classification of critical points as well as of birth--death points in persistence diagrams.
	We demonstrate how the SDPH approach quantifies texture in shapes by applying it to both simulated and real data: in particular, realizations of Gaussian random fields; synthetic porous shapes generated by ``curvatubes" \citep{song_generation_2022}; and confocal images of bone marrow vessels. From these case studies, we derive a general interpretation of SDPH diagrams that is useful for shape texture analysis.

	\textbf{Outline.}
	In \Cref{sec:background}, we provide the necessary background and prove results associated with foundational concepts necesssary for our work; we recall the essential facts linking persistent homology to Morse theory and introduce the Min-type framework. In particular, we focus on describing non-degenerate Min-type critical points of functions. \Cref{sec:morse_theory_signdist} is the theoretical pillar of this work and contains most of our theoretical contributions. In \Cref{sec:topological_morse_theory}, we first show that the fundamental lemmas of smooth Morse theory also hold in the more general setting of topological Morse functions.
	Then, in \Cref{sec:generalization_dist_functions}, we describe in \Cref{def:ndg_crit_dist} the geometric conditions where signed distance functions are topological Morse functions. Concluding with \Cref{sec:genericity}, we prove that signed distance functions induced by \emph{generic} embeddings of surfaces are topological Morse functions in \Cref{thm:generic} by satisfying the conditions of \Cref{def:ndg_crit_dist}.

	\Cref{sec:morse_theory_signdist} lays the foundations for \Cref{sec:SDPH}, where we define the SDPH approach rigorously. In particular, \Cref{thm:generic} allows us to show for generic embeddings of surfaces, the persistence module of the signed distance filtration admits a barcode decomposition into finitely-many interval modules (\Cref{cor:gendistPH_finite}). Moreover, as the topological critical points of the signed distance filtration are shown to be generically non-degenerate, we can classify and geometrically interpret birth--death pairs in the barcode using the index of the critical points. We use this fact in \Cref{sec:quantifying_texture_applis}, where we apply SDPH to analyze synthetic and real data describing textures and derive a general interpretation of the SDPH diagrams for applications.
	Readers solely interested in the practical applications of SDPH may go to Sections \ref{sec:SDPH} and \ref{sec:quantifying_texture_applis} directly.

	\section{Background and Foundational Concepts}
	\label{sec:background}
	
	In this section, we focus on outlining foundational concepts in persistent homology, smooth Morse theory, and Min-type Morse theory. The ideas and results described in this section motivate our study of topological Morse functions in \Cref{sec:topological_morse_theory}, and give us the theoretical framework for analyzing signed distance persistent homology in \Cref{sec:generalization_dist_functions}. 
	
	Morse theory \citep{milnor_morse_1963} is an important tool in differential topology for studying topological properties of manifolds, grounded on a fundamental relationship between  gradient dynamics of smooth functions and the homology of the underlying manifold. 
	A thorough account of applications of Morse theory to topology, differential geometry, and mathematical physics is given by \cite{bott1988morse}. Within the scope of topological data analysis, Morse theory also has foundational consequences; first, for theoretically understanding sublevel sets filtrations; second, for deriving new computational algorithms; and third, for generating new methods for data analysis. We outline below how Morse theory describes the persistence modules of filtrations of smooth functions in terms of relations between critical points (see also Chapter 7 of \cite{ghrist2014elementary}). \emph{Discrete} Morse Theory---an extension of smooth Morse theory to monotone functions on simplicial complexes---underpins efficient algorithms for computing the barcodes of filtered simplical complexes \citep{harker2014discrete} and other topological invariants \citep{curry2016discrete}. Morse-theoretic concepts have also been applied to methods for clustering data \citep{chazal2013persistence}, shape analysis \citep{cazals2003molecular}, image analysis and partitioning \citep{edelsbrunner2001hierarchical}, and topographical function simplification \citep{bauer2012optimal,tierny2012generalized}.
	
	Min-type Morse theory---a variant of Morse theory for \emph{Min-type functions} \citep{gershkovich_morse_1997}---has been instrumental in quantifying and bounding the homological behavior of geometric complexes on point clouds, a natural setting in many applications that topological data analysis is geared towards. A rich vein of research in this topic has been contributed by \cite{bobrowski_distance_2014,bobrowski2017vanishing,bobrowski2019random,de2022random}, among others. In contrast to those approaches, which cast distance functions to point clouds as Min-type functions, we focus on distance functions to surfaces which may not be globally Min-type. Nevertheless, we show in \Cref{sec:generalization_dist_functions} that the Min-type Morse theory framework is a key tool for describing the critical points of distance functions to generic surfaces. We give a summary of the relevant ideas from Min-type theory in this section to build up to that result. However, we note that for a full Morse theory for distance functions, we will require the more general framework of \emph{topological Morse theory} described in \Cref{sec:topological_morse_theory}.

	\subsection{Persistent Homology and Morse Theory}
	\label{sec:PH_and_Morse_theory}
	
	We begin our discussion with a brief summary of basic notions in persistent homology, the leading framework in topological data analysis. In particular, we highlight how classical Morse theory allows us to relate the barcode of a persistence module with critical points of a smooth function, setting the stage for generalizations to persistence modules induced by filtrations of Min-type or topological Morse functions. 
	
	We present how persistent homology relates to Morse theory in the  smooth setting. To do this, we introduce some essential facts and basic terminology related to persistent homology and Morse theory in the smooth setting.
	
	\textbf{Sublevel Set Filtrations and Persistence Diagrams.}
	Given a topological space $X$ and a continuous function $f : X \to \R$, the \textit{sublevel set filtration} of $f$ is the sequence of inclusion of sublevel sets $X_a = \sublevel{f}{a}$, given by
	\[\Xb : \quad X_a \subseteq X_b, \quad \forall\ a \leq b \in \R.\]
	The inclusion maps $\iota_{a,b} : X_a \to X_b$ induce by functoriality a sequence of homology groups
	\begin{equation*}
		\PH_k(f) : \quad \H_k(X_a) \xrightarrow{\H_k(\iota_{a,b})} \H_k(X_b), \quad \forall\ a \leq b \in \R
	\end{equation*}
	that describes the changes in the topology of sublevel sets. Here, $k \geq 0$ refers to the homology dimension.
	If the homology groups are taken with coefficients in a field $\mathbb{F}$, this sequence is called the \emph{persistence module} of $f$ \citep{zomorodian2004computing, edelsbrunner_persistent_2008}. Furthermore, if the vector spaces $\H_k(\sublevel{f}{a})$ are finite-dimensional for all $a \in \R$, the persistence module is said to be \emph{pointwise finite dimensional} (p.f.d.); in this case, the persistence module can be specified up to isomorphism by a \emph{barcode}, which is a multiset of (open--open, closed--closed, closed--open, or open--closed) intervals in the extended real number line $\bar \R = \R \union \{-\infty, +\infty\}$,
	\begin{equation*}
		\Bar_k(f) = \{ \langle b, d \rangle \subset \bar\R \},
	\end{equation*}
	where $\R \union \{-\infty\} \ni b < d \in \R \union \{+\infty\}$.
	Alternatively, the module can be characterized by a \textit{persistence diagram}, which is a multiset of \textit{birth--death points} in the upper-diagonal extended half-plane
	\begin{equation*}
		\Dgm_k(f) = \{ (b,d) \in (\bar \R)^2\}.
	\end{equation*}
	Each interval (or bar, or birth--death point) describes a $k$-dimensional homology class being born at filtration value $b$ and being trivialized at value $d$. This characterization is possible due to the interval decomposition theorem of \cite{crawley2015decomposition} for p.f.d.~persistence modules: algebraically, a persistence module can be uniquely decomposed into a direct sum of \emph{interval modules}, where each interval module is in one-to-one correspondence with an interval in the barcode. Details can be found in \cite{gabriel1972unzerlegbare} and \cite{crawley2015decomposition}.
	
	\textbf{Stability.}
	Persistence diagrams exhibit a property which is fundamental to their use in data analysis, namely, their \textit{stability} with respect to perturbations of the input data, in terms of the bottleneck distance.
	Let $\Dgm^{(1)}$ and $\Dgm^{(2)}$ be two persistence diagrams, and $\Pi$ denote the set of all possible bijections $\pi : \Dgm^{(1)} \to \Dgm^{(2)}$ between them, where it is assumed that any point on the diagonal $\{x = y\}$ in the diagrams has infinite multiplicity. 
	The \emph{bottleneck distance} between two persistence diagrams is
	\[d_B(\Dgm^{(1)}, \Dgm^{(2)}) = \inf \limits_{\pi \in \Pi} \sup \limits_{x \in \Dgm^{(1)}} \|x - \pi(x) \|_\infty.\]
	Here, we use the convention that $\infty - \infty = 0$ to deal with infinite death times.
	
	The stability result holds for a broad class of functions called tame.
	A function $f : X \to \R$ on a topological space $X$ is \emph{tame} if the persistent homology of the sublevel set filtration $\PH(f) = \PH(\Xb)$, $X_t = \{f \leq t\}$, has \textit{finite type}. 
	Equivalently, the persistence module is p.f.d.~and has finitely many \emph{homological critical values}, at which, by definition, there exists $k$ such that for all sufficiently small $\eps > 0$, the map $\H_k(f^{-1}(-\infty,t-\eps]) \to H_k(f^{-1}(-\infty,t+\eps])$) is \textit{not} an isomorphism.
	
	We suppose that $X$ is \emph{triangulable}, i.e., there is a (finite) simplicial complex $K$ whose underlying space $|K|$ is homeomorphic to $X$.
	\begin{thm}[Stability Theorem \citep{cohen-steiner_stability_2007}]
		\label{thm:stability_thm}
		Let $X$ be a triangulable space with continuous tame functions $f, g : X \to \R$.
		Then the persistence diagrams of the sublevel set filtrations they induce on $X$ satisfy
		\[d_B(\Dgm(f), \Dgm(g)) \leq \|f - g\|_\infty.\]
	\end{thm}
	The Stability Theorem states that, if two input functions are geometrically close to each other in $L_\infty$ norm, then necessarily their diagrams are also close to each other in bottleneck distance $d_B$.  
	This is a crucial result for applicability of persistent homology to real data: if the input is perturbed by noise, sampling error, and so on, the output diagram will not be too different from the actual diagram.
	There are many other stability results in different settings; an account of these works can be found in \cite[Section 2.3]{cao_approximating_2022}.

	\textbf{Smooth Morse Theory and Persistent Homology.}
	If $f$ is a Morse function \citep{milnor_morse_1963}, there is a correspondence between the critical values of $f$ (values where $\nabla f = 0$), and the end points of the intervals in the barcode of $f$. A \textit{Morse function} is a smooth, proper function on a smooth manifold $M$, where any critical point $q$ is non-degenerate (i.e., the Hessian evaluated at $q$ has full rank). The correspondence between the end points of intervals in the barcode of $f$ and its critical values stem from the following fundamental theorems of Morse theory.

	Recall that a $\lambda$-cell is a closed ball of dimension $\lambda$,
	$e^\lambda = \{y \in \R^\lambda ~|~ \|y\| \leq 1 \}$,
	and its boundary is denoted $\bord e^\lambda = \{y \in \R^\lambda ~|~ \|y\| = 1 \}$.
	Attaching a $\lambda$-cell to a topological space $Y$ consists of taking first the disjoint union of $Y$ and $e^\lambda$, and then identifying points $y \in \bord e^\lambda$ to $\Phi(y) \in Y$ via a continuous map $\Phi : \bord e^\lambda \to Y$. The resulting space is endowed with the quotient topology and denoted by $Y \union_\Phi e^\lambda$.
	
	\begin{thm}[\cite{milnor_morse_1963}] \label{thm:Morsebasic} Let $f: M \to \R$ be a smooth, proper function on an $n$-dimensional manifold $M$ and let $\mathrm{Crit}(f)$ denote the set of critical points of $f$. 
		\begin{enumerate}[label=(M\arabic*), ref=(M\arabic*), start=0]
			\item If $q$ is a non-degenerate critical point, there is a diffeomorphism $\phi : \R^n \to U$ onto an open neighborhood $U$ of $q$, such that $\phi(0) = q$, and
			\begin{equation} \label{eq:Morsebasic}
				f \circ \phi (x) = f(q) - \sum_{i = 1}^{\lambda(q)} x_i^2 + \sum_{i = \lambda(q) + 1}^{n} x_i^2.
			\end{equation}
			The number of negative coefficients $\lambda(q)$ is the \emph{index} of $q$.
			\item  \label{thm:Morsebasic_1} If $\intlevel{f}{a}{b}$ contains no critical points of $f$, then $\sublevel{f}{a}$ is a deformation retract of $\sublevel{f}{b}$ to ; thus they are homotopy equivalent.
			\item  \label{thm:Morsebasic_2}  Furthermore, if $\mathrm{Crit}(f) \cap \intlevel{f}{a}{b} = \{q_1, \ldots, q_k\}$ is a finite set of non-degenerate critical points, and $\mathrm{Crit}(f) \cap \intlevel{f}{a}{b} \subset \level{f}{c}$,  then $\sublevel{f}{b}$ has the homotopy type of $\sublevel{f}{a}$ with cells $e^{\lambda(q_i)}$ of dimension $\lambda(q_i)$ attached along the boundaries of the cells:
			\begin{equation*}
				\sublevel{f}{b} \simeq \sublevel{f}{a} \cup e^{\lambda(q_1)} \cup \cdots \cup e^{\lambda(q_k)}.
			\end{equation*}
		\end{enumerate}
	\end{thm}
	The two statements \ref{thm:Morsebasic_1} and \ref{thm:Morsebasic_2} are informally known as the ``isotopy lemma" and the ``handle attachment lemma". \Cref{eq:Morsebasic} is known as the normal form of the function at the critical point $q$.
	
	As a consequence of the description of topological changes in sublevel sets of Morse functions in  \Cref{thm:Morsebasic}, we can characterize the persistence modules of Morse functions where the following finiteness conditions hold. Note that the conditions always hold if the underlying manifold is compact. 
	
	\begin{cor} 
		\label{cor:PM_morse_function_tame} Let $f: M \to \R$ be a smooth, proper Morse function on a manifold $M$ with finite dimensional homology groups $\H_\bullet(M)$ with field coefficients. If $f$ only has finitely many critical points, then:
		\begin{enumerate}[label = (\roman*)]
			\item The persistence module of the sublevel set filtration $f$ is pointwise finite dimensional (p.f.d.);
			\item For all $k$, the barcode $\Bar_k(f)$ is a finite multiset of intervals $\{ [b,d) \subset \bar{\R} \}$;
			\item A critical point with index $\lambda$ either corresponds to the birth of an interval in $\Barc{\lambda}{f}$ with homology dimension $\lambda$, or the death of an interval in $\Barc{\lambda-1}{f}$ with homology dimension $\lambda - 1$.
		\end{enumerate} 
	\end{cor}
	\begin{proof}
		Since $f$ only has finitely many critical points, \ref{thm:Morsebasic_1} and \ref{thm:Morsebasic_2} together imply $\H_\bullet(\sublevel{f}{t}) \to \H_\bullet(M)$ only has finite rank and kernel. Therefore $\H_\bullet(\sublevel{f}{t})$ is finite dimensional, and $\PHf{\bullet}{f}$ are p.f.d.~persistence modules. As each critical point $p$ of index $\lambda$ corresponds to attaching a $\lambda$-cell in the sublevel set filtration, the filtration passing $p$ will either correspond to the birth of an interval in $\Barc{\lambda}{f}$ with homology dimension $\lambda$, or the death of an interval in $\Barc{\lambda-1}{f}$ with homology dimension $\lambda - 1$. Finally \ref{thm:Morsebasic_1} and \ref{thm:Morsebasic_2} together with finite $\mathrm{Crit}(f)$ imply that homological features can only be born and destroyed at finitely many critical values; thus any interval in the barcode of $f$ can only be of the form $ [b,d) $. 
	\end{proof}
	
	\subsection{Min-Type Functions and their Critical Points}
	\label{sec:min-type}
	
	Two major components are needed here to generalize Morse theory to Euclidean distance functions: \textit{criticality} and \textit{index}. \cite{grove_generalized_1977} introduced a definition of \textit{critical points} for distance functions $d_p$ w.r.t.~a single point on a Riemannian manifold, precisely to generalize Morse's isotopy lemma (see \citep{cheeger_critical_1991, grove_critical_1993}). 
	However, no explicit notion of index was declared. 
	Later, \cite{gershkovich_singularity_1997} and \cite{gershkovich_morse_1997} 
	generalized Morse theory to the broad class of \emph{Min-type} functions, which allowed for the definition of an appropriate notion of \textit{index} and \textit{non-degeneracy} of critical points.
	
	Essentially, Min-type functions are functions that can be expressed as the minimum of a finite number of smooth functions.
	\cite{gershkovich_morse_1997} studied Min-type distance functions $d_p$ on non-positively curved Riemannian manifolds $M$ and showed that for compact $M$, the $d_p$ are even Morse-Min-type (i.e., have non-degenerate regular and non-degenerate critical points only) for generic non-positively curved Riemannian metrics on $M$.
	Later, \cite{itoh_cut_2007} obtained similar results by rephrasing Min-type properties of the distance function $d_p$ using more direct geometric arguments.
	
	In topological data analysis, the Min-type framework had an important application to distance functions $d_{\mathcal{P}}$ defined w.r.t.~a finite point cloud $\mathcal{P}$ in  $\R^n$ \citep{bobrowski_distance_2014}. In particular,
	the $\eps$-\v{C}ech complexes of $\mathcal{P}$ are homotopy equivalent to the sublevel sets $\{d_{\mathcal{P}} \leq \eps \}$ of the distance function, by the Nerve Theorem \citep{borsuk_imbedding_1948}. Therefore, if $\mathcal{P}$ has a generic configuration, the critical points of $d_{\mathcal{P}}$ are non-degenerate and determine the exact homological changes in the \v{C}ech filtration.
	Yet, to the best of our knowledge, there has been no generalization of Morse theory for Euclidean distance fields w.r.t.~smooth boundaries, especially in the context of persistent homology.

	\subsubsection{Min-Type Functions and their Representations}
	
	We begin with an introduction to Min-type functions. For the sake of clarity and conciseness, we only cover the necessary content from the Min-type function theory for our goals; moreover, we rephrase it in simpler terms for the Euclidean case only. See \citep{gershkovich_morse_1997} for complete details of the Min-type framework in full generality.  In what follows, we borrow from \citep{gershkovich_morse_1997} and provide some proofs with the goal of making this work self-contained.
	
	As mentioned previously, a Min-type function is a function which can be locally written as the minimum of a finite number of smooth functions. As such, a Min-type function itself may not be smooth even though it can be constructed from a collection of smooth functions. 
	
	As Min-type functions are locally constructed, we first review some terminology regarding local equivalences of sets and functions before formally giving the definition of Min-type functions. The properties stated below are only concerned with the behavior of functions and sets in arbitrarily small neighborhoods of a point. Given a topological space $X$, two maps $f, g : X \to \R$ define the same germ at $x \in X$ if there is a neighborhood $U$ of $x$, such that the restrictions of $f$ and $g$ to $U$ are equal: $\forall\ u \in U, f(u) = g(u)$.
	Similarly, if $A$ and $B$ are two subsets of $X$, then they define the same germ at $x$ if there is a neighborhood $V$ of $x$ such that $A \inter V = B \inter V$.
	Defining the same germ at $x$ is an equivalence relation (on maps or sets) and the equivalence classes are called \textit{germs} (of maps or of sets). If $f$ and $g$ define the same germ, then the terminology is that $f = g$ ``as germs of functions"; and, similarly, $A = B$ ``as germs of sets".
	
	Hereafter, we use the notation $\{f = g\}$ in place of $\{x ~|~ f(x) = g(x)\}$, $\{f < g\}$ in place of $\{x ~|~ f(x) < g(x)\}$, and $\{f \leq g\}$ in place of $\{x ~|~ f(x) \leq g(x)\}$.
	
	\begin{defi}[Min-type functions \citep{gershkovich_morse_1997}]
		\label{def:min_type}
		Let $f : \R^n \to \R$ be a continuous function. $f$ is a $C^k$-smooth \emph{Min-type function at $q$} if there exists a finite family of $C^k$-smooth functions $\alpha_1, \ldots, \alpha_m$, called a \emph{representation} of $f$, such that 
		\[f = \min \{\alpha_1,\ldots,\alpha_m\} \]
		as germs of functions.
	\end{defi}
	
	The \emph{active set} of $\alpha_i$ is the germ of set
	\[A_i = \{x ~|~ f(x) = \alpha_i(x)\}.\]
	The germ of the interior of $A_i$, $A_i^\circ$, is called the \emph{open active set} of $\alpha_i$.
	
	The statement that $f$ is a Min-type function at $q$ is equivalent to saying there exists a neighborhood $N$ of $q$ on which $\forall\ x \in N, f(x) = \min_{i} \alpha_i(x)$.
	In particular, $N = \union_i ~ A_i$ (as germs of sets) and $f$ is $C^k$-smooth on $A_i^\circ$.
	
	In general, a Min-type function can have more than one representation $\{ \alpha_i \}$. A ``most concise'' representation that uses the fewest constituent functions $\alpha_i$ is said to be \emph{efficient}.
	
	\begin{defi}[Efficient representation]
		The representation $f = \min \{ \alpha_1, \ldots, \alpha_m\}$ as germs of functions is \emph{efficient at $q$} if $m$ is the minimal number of functions necessary to represent $f$ around $q$. \label{def:efficient_repr}
	\end{defi}
	
	Efficient representations were originally referred to as ``minimal representations'' by \cite{gershkovich_morse_1997}. To avoid confusing terminology, we use the term ``efficient'' instead of minimal in this work. 
	
	While efficient representations may not be unique, it can be shown that if an efficient representation admits \emph{linearly-independent} gradients (LIG), then the efficient-LIG representation is unique as a set of functions \citep{gershkovich_morse_1997}. 
	
	Recall that a set of vectors $v_1,\ldots,v_m \in \R^n$ is said to be in \emph{general position} in $\R^n$ if the dimension of the affine subspace spanned by $\{v_1,\ldots,v_m\}$ is equal to $m-1$.  Equivalently, this means the collection $v_1,\ldots, v_m$ are in general position if $v_2 - v_1, \ldots, v_m - v_1$ are linearly independent.  Notice that the origin may or may not belong to the affine subspace (depending on whether $v_1,\ldots,v_m$ are linearly dependent or not).
	
	\begin{defi}
		\label{def:LIG}
		A representation $f = \min \{\alpha_1,\ldots,\alpha_m\}$ at $x$ is \emph{efficient-LIG} if it is efficient and if the gradients $\{\grad \alpha_1,\ldots,\grad \alpha_m\}$ are in general position at $q$.
	\end{defi}

	\begin{lem}
		\label{lem:efficient_non_empty}
		Let $f = \min \{\alpha_1,\ldots,\alpha_m\}$ be an efficient representation at $q$. Then $A_i^\circ \neq \emptyset$ as germs of sets.
	\end{lem}
	\begin{proof}
		Since the representation is efficient, for any $i$, there exists $y$ such that $\forall\ j \neq i,\, \alpha_i(y) < \alpha_j(y)$ in any neighborhood of $q$. Otherwise, we could write $f = \min_{j \neq i} \{ \alpha_j\}$, which contradicts the efficiency of the representation. Since $\inter_{j \neq i} \{\alpha_i < \alpha_j\} \subset A_i^\circ$, we have $A_i^\circ \neq \emptyset$ as germs of sets.
	\end{proof}

	\begin{lem}
		\label{lem:connected}
		The germs of the open active sets of an efficient-LIG representation are connected.
	\end{lem}
	\begin{proof}
		By subtracting $\alpha_1$ from any $\alpha_i$, we may suppose that $\alpha_1 \equiv 0$, so then $f = \min\{0, \alpha_2, \ldots, \alpha_m\}$ and the LIG condition means that $\grad \alpha_2, \ldots, \grad \alpha_m$ are linearly independent. In particular, $\grad \alpha_i \neq 0$ for $i \geq 2$ and we can apply the implicit function theorem. Locally, for $i \geq 2$, the equation $\alpha_i(x) = 0$ represents a submanifold of dimension $n-1$ at $q$, which partitions a ball at $q$ into two connected open sets, $\{\alpha_i > 0\}$ and $\{\alpha_i < 0\}$. Thus $A_1 = \{\forall\ i \geq 2,\, \alpha_i \geq \alpha_1 \} = \inter_{i \geq 2} \{\alpha_i \geq 0\}$ is connected as a germ of set, as well as $A_1^\circ = \inter_{i \geq 2} \{\alpha_i > 0\}$. We can similarly show  that $A_i^\circ$ is connected for any $i$.
	\end{proof}
	
	\begin{lem}
		\label{lem:active_sets_maximal}
		Let $f = \min \{\alpha_1,\ldots,\alpha_m\}$ be an efficient-LIG representation at $q$. Denote by $\Xi$ the set of germs of connected open subsets on which $f$ is $C^k$-smooth.
		Then $\{A_i^\circ\}$ are the maximal elements of $\Xi$.
	\end{lem}
	\begin{proof}
		The germ of set $A_i^\circ$ is connected (\Cref{lem:connected}) and $f$ is smooth on it, so $A_i^\circ \in \Xi$. There also exists $a \in A_i^\circ \neq \emptyset$ (\Cref{lem:efficient_non_empty}), with $\alpha_i(a) = f(a)$.
		Let $V \in \Xi$ be such that $A_i^\circ \subset V$. We need to show that $V = A_i^\circ$.  Suppose not, then $V \setminus A_i \neq \emptyset$ and there exists $b \in V \setminus A_i$ and $j \neq i$ such that $\alpha_i(b) > f(b) = \alpha_j(b)$.
		Then since $V$ is connected, consider a continuous path $\gamma(t) \in V$ joining $a \in A_i^\circ$ to $b$. $f$ coincides with $\alpha_i$ on $a$ but then $\alpha_i$ becomes strictly greater than $f$ around $b$. By continuity and definition of Min-type functions, we can then find $\ell \neq i$ (which is not necessarily $j$) and a point $y$ on the path such that $\alpha_\ell(y) = \alpha_i(y) = f(y)$, so $y \in A_\ell \inter A_i \inter V$.
		
		The previous arguments apply to any neighborhood of $q$. By extracting an index $\ell_0 \neq i$, we can build a sequence $(y_\nu) \in V \inter A_i \inter A_{\ell_0}$ converging to $q$. But since $f$, $\alpha_i$, and $\alpha_{\ell_0}$ are $C^k$-smooth on $V$ and $\forall\ j,\, A_j = \overline{A_j^\circ}$, we get $\grad f(y_\nu) = \grad \alpha_i(y_\nu) = \grad \alpha_{\ell_0}(y_\nu)$. At the limit, we obtain $\grad \alpha_i(q) = \grad \alpha_{\ell_0}(q)$, which contradicts the LIG property. Therefore $V = A_i^\circ$.
		
		To show that the sets $A_i^\circ$ are the only maximal elements of $\Xi$, consider $V \in \Xi$. There exists $i$ such that $V \inter A_i \neq \emptyset$ as germs of sets. We can re-use the previous argument to show that $V \setminus A_i = \emptyset$, hence $V \subset A_i^\circ$.
	\end{proof}
	
	It then follows that germs of efficient-LIG representations are uniquely defined up to permutation since they induce the same maximal connected open sets on which $f$ is $C^k$-smooth.
	\begin{cor}[Uniqueness]
		\label{cor:uniqueness_minirep}
		Let $f = \min \{\alpha_1,\ldots,\alpha_m\} = \min \{\beta_1,\ldots,\beta_m\}$ be two efficient-LIG representations at $q$. Then there exists a permutation of indices $\sigma$ such that the germs of active sets and functions $A_i = B_{\sigma(i)}$ and $\alpha_i = \beta_{\sigma(i)}$ are equal.
	\end{cor}
	
	In subsequent discussions, we will find it useful to characterize the intersection of active sets $A_i = \{ f(x) = \alpha_i \}$ given an efficient-LIG representation. 
	
	\begin{defi}
		\label{def:G(x)}
		Given an efficient-LIG representation $f = \min \{\alpha_1, \ldots \alpha_m\}$ at $q$,  $G_f(q)$ is defined as the \emph{germ of set}
		$$
		G_f(q) = \{\alpha_1 = \ldots = \alpha_m = f\} = \bigcap_i A_i,
		$$
		where recall that $A_i = \{f = \alpha_i\}$. 
	\end{defi}
	Notice that $q \in G_f(q)$ and that this definition does not depend on the choice of representation due to \Cref{cor:uniqueness_minirep}.
	
	The form of $G_f(q)$ is locally a smooth submanifold of dimension $n-m+1$ due to the LIG property. In particular, if the number of functions $m$ over which the minimum is taken is equal to $n+1$, then $G_f(q)$ may reduce to the point $\{q\}$ itself.
	
	\begin{lem} \label{lem:G(q)_mfd} Given an efficient-LIG representation $f = \min \{\alpha_1, \ldots \alpha_m \}$ at $q$, the germ of set $G_f(q)$ is an $(n-m+1)$-dimensional manifold. 
	\end{lem}
	
	\begin{proof}
		Let $U$ be a neighborhood of $q$ where $f = \min \{\alpha_1, \ldots \alpha_m \}$. As the representation is efficient, $G_f(q) = \cap_{i = 1, \ldots, m} \{\alpha_i = f\} = \cap_{i= 2,\ldots,m}\{\alpha_i - \alpha_1 = 0\}$. Thus, we can express $G_f(q)$ as an intersection of $(m-1)$ level sets of smooth functions $\cap_{i= 2,\ldots,m}\{\alpha_i - \alpha_1 = 0\}$. Since the representation is LIG, the set of $(m-1)$ $\{\nabla (\alpha_i - \alpha_1)\}_{i = 2,\ldots, m}$ conists of linearly independent vectors. Thus, due to the implicit function theorem, there is an open neighborhood $V \subset U$ on which $G_f(q)$ is a $(n-(m-1))$-dimensional submanifold.
	\end{proof}
	
	We also remark that the germ of the restriction $f_{|G_f(q)}$ is a smooth function, since it coincides with the restriction of  $\alpha_i$ to $G_f(q)$ for any $i$.
	
	\subsubsection{Non-Degenerate Critical Points and Index}
	We now define the critical points of Min-type functions.
	
	\begin{defi}[Min-type Critical Point]\label{def:q_mintype_crit}  Let $f$ be a $C^k$ Min-type function that admits an efficient-LIG representation $f = \min \{ \alpha_1, \ldots, \alpha_m\}$ at $q$. If $0$ is contained in the convex hull of gradients $\{ \nabla \alpha_1 ,\ldots  \nabla \alpha_m \}$ at $q$, then $q$ is said to be a \emph{Min-type critical point} of $f$.
	\end{defi}
	
	We use the terminology of ``critical points" for points described in \Cref{def:q_mintype_crit} because by restricting to the submanifold on which $f$ is smooth, those points are indeed critical points of the restricted smooth function.
	
	\begin{lem} \label{lem:q_mintype_crit} If $q$ is a Min-type critical point of $f$ then $\nabla f_{\rvert G_f(q)}\rvert_q= 0$ on the submanifold $G_f(q)$.
	\end{lem}
	
	\begin{proof}
		Recall $G_f(q)$ is the germ of the $(n-m+1)$ dimensional submanifold (\Cref{lem:G(q)_mfd}) locally defined by the solution set  $\alpha_i - \alpha_1 = 0$ for $i= 2,\ldots, m$. Since $f_{\rvert G_f(q)} = \alpha_i {\ }_{\rvert G_f(q)}$ and $\alpha_i$ is smooth, the restriction $f_{\rvert G_f(q)}$ of $f$ to $G_f(q)$ is smooth. 
		
		If $\nabla \alpha_i\rvert_q = 0$,  then $\nabla f_{\rvert G_f(q)} = 0$ (notice that since the representation is LIG at most one of $ \nabla \alpha_i$ vanish at $q$). Thus, consider the case where none of $\alpha_i$ are critical at $q$. Then $q$ is a critical point of $f_{\rvert G_f(q)}$ if and only if $\nabla \alpha_i\rvert_{q} \perp T_q G_f(q)$. Since $G_f(q)$ is locally defined by the solution set  $\alpha_i - \alpha_1 = 0$ and the representation is LIG, the set of linearly independent vectors $\{\nabla (\alpha_i -  \alpha_1)\rvert_{q}\}_{i= 2,\ldots, m}$ is a basis of $N_qG_f(q)$. Therefore, $q$ is a critical point of $f_{\rvert G(q)}$ if and only if $\nabla  \alpha_1\rvert_{q}$ is in the span of $\{\nabla (\alpha_i -  \alpha_1)\rvert_{q}\}_{i= 2,\ldots, m}$. 
		
		Now assume $0$ is contained in the convex hull of gradients $\{ \nabla \alpha_1 ,\ldots  \nabla \alpha_m \}$ at $q$: i.e., there is a set of non-negative coefficients $t_i$ that sum to one, such that
		\begin{align*}
			0 &= \sum_{i=1}^m t_i \nabla \alpha_i \rvert_{q} =  \sum_{i= 2}^m t_i \nabla (\alpha_i - \alpha_1) \rvert_{q} + \nabla  \alpha_1\rvert_{q}.
		\end{align*}
		Hence, we have $\nabla  \alpha_1\rvert_{q} \in N_q G(q)$ and thus $q$ is a critical point of the smooth function $f_{\rvert G(q)}$ on $G_f(q)$.
	\end{proof}
	
	We now turn to the notions of non-degeneracy and index in the context of Min-type functions. 
	
	\begin{defi}[Non-degenerate critical point \citep{gershkovich_morse_1997}]
		\label{def:non_degen_crit}
		Let $f$ be a $C^k$-Min-type function. If $q$ is a Min-type critical point with an efficient-LIG representation $f = \min \{ \alpha_1, \ldots, \alpha_m\}$ at $q$, such that 
		\begin{enumerate}
			\item The origin lies strictly in the \emph{interior} of the convex hull of the gradients $\grad \alpha_1, \ldots, \grad \alpha_m$ at $q$; and
			\item The restriction $f_{|G_f(q)}$ is a Morse function.
		\end{enumerate}
		Then $q$ is said to be a \emph{non-degenerate critical point} (in the Min-type sense) of $f$. 
	\end{defi}
	
	Notice that the definition of non-degenerate critical points does not depend on the choice of efficient representation $f = \min \{ \alpha_1, \ldots, \alpha_m\}$, due to \Cref{cor:uniqueness_minirep}.
	Moreover, a non-degenerate critical point $q$ of $f$ is a non-degenerate critical point for the restriction $f_{|G_f(q)}$ in the sense of smooth Morse theory.
	
	\begin{defi}[Index of a Min-type function]
		\label{def:index_min_type}
		Let $f : \R^n \to \R$ be a $C^k$-Min-type function with efficient representation $f = \min\limits_{i = 1,\ldots,m} \alpha_i$ at $q$,
		and suppose that $q$ is a non-degenerate critical point. The \emph{index} of $f$ at $q$ is
		\[ \idx(q;f) := (m-1) + \idx(q;\, f_{|G_f(q)}).\]
	\end{defi}
	The index has a geometric interpretation: it is the maximal dimension of a submanifold in a neighborhood of $q$ such that the restriction of $f$ on it has a strict maximum at $q$ \cite[Section 3.1, Proposition 3]{gershkovich_morse_1997}.
	
	The Min-type framework is interesting for our purpose of extending Morse theory to distance functions because it allows us to write Min-type functions into a normal form around non-degenerate critical points, similar to the setting of smooth Morse theory (see \Cref{thm:Morsebasic}).
	In classical Morse theory \citep{milnor_morse_1963}, smooth functions $f$ have simple forms at regular points and non-degenerate critical points. 
	These forms are on the same orbit as $f$ under the action of the group of local diffeomorphisms of $\R^n$, and we either have $f \sim \mathrm{const.} + x_n$ at regular points, or $f \sim \mathrm{const.} - \sum_{i = 1}^{\lambda} x_i^2 + \sum_{i = \lambda + 1}^{n} x_i^2$ at non-degenerate critical points. However, for Min-type functions, another group of transformations is needed \citep{matov_singularities_1986, gershkovich_singularity_1997}.
	
	\begin{defi}[Group of almost smooth homeomorphisms]
		A homeomorphism $\phi : U \to V$, where $U$ and $V = \phi(U)$ are open neighborhoods of $0$ in $\R^n$ and $\phi(0) = 0$, is said to be \textit{almost smooth} if $\phi$ and $\phi^{-1}$ are smooth except on a submanifold of codimension $\geq 1$.
	\end{defi}
	
	The following normal form theorem for Min-type functions is due to \cite{gershkovich_morse_1997} (Section 3.2, Theorem 1) and \cite{gershkovich_singularity_1997} (Theorem 1.3). 
	\begin{thm}[Normal form \citep{gershkovich_morse_1997}]
		\label{thm:normal_form}
		Let $f : \R^n \to \R$ be a Min-type function with efficient representation $f = \min \{ \alpha_1, \ldots, \alpha_m\}$ at $q$, and suppose that $q$ is a non-degenerate critical point. Then there exists an almost smooth homeomorphism $\phi : U \to V$, where $U$ and $V = \phi(U)$ are open neighborhoods of $0$ and $q$ in $\R^n$, such that $\phi(0) = q$ and for all $x \in U$, 
		\[ f \circ \phi \,(x) = f(q) - \sum_{i = 1}^{\lambda} x_i^2 + \sum_{i = \lambda + 1}^{n} x_i^2,\]
		where $\lambda = \idx(q; f)$. 
	\end{thm}
	This implies that non-degenerate critical points of a Min-type function are \emph{isolated}.
	
	\begin{rmk}
		\cite{gershkovich_morse_1997} also defined a notion of \emph{non-degenerate regular points}. The definition is similar to \Cref{def:non_degen_crit}, where Condition 2 must be replaced by 2': the origin does not belong to the convex hull of the gradients $\grad \alpha_1, \ldots, \grad \alpha_m$ at $q$, and a Condition 4' must be added, which states that: any $m-1$ gradients among $\grad \alpha_1, \ldots, \grad \alpha_m$ are linearly independent at $q$.
		Then they defined \emph{Morse Min-type functions} to be Min-type functions that only admit non-degenerate regular points and non-degenerate critical points. These refinements will not be needed for our purpose, as we are only concerned by non-degenerate critical points. 
	\end{rmk}
	
	\begin{rmk} \label{rmk:clarke1}
		It can be shown that Min-type critical points (\Cref{def:q_mintype_crit}) are also critical points of the \emph{Clarke subgradient} \citep{clarke1990optimization,agrachev1997morse}, a generalization of gradients for functions that need only be locally Lipschitz continuous. For a compact interlevel set $\intlevel{f}{a}{b}$ that contains only regular points of the Clarke subgradient, it can be shown that $\sublevel{f}{a}$ is a strong deformation retract of $\sublevel{f}{b}$ using the theory of \emph{weak slopes} \citep{corvellec2001Second} (another related generalization of gradients).
	\end{rmk}

	\section{Morse Theory for Signed Distance Functions}
	\label{sec:morse_theory_signdist}
	
	In this section, we generalize Morse theory to \emph{signed} distance functions $d : \R^n \to \R$ to smooth compact surfaces bounding an open set in $\R^n$. With the goal of computing SDPH to analyze shapes, on the theoretical front, we investigate conditions under which the persistence module $\PH_k(d)$ is p.f.d.~(see \Cref{sec:PH_and_Morse_theory}), and aim to find a correspondence between homological births or deaths and geometric features of the surface.
	Since $d$ is not continuously differentiable on all of its domain, it is not a Morse function in the classical sense. Therefore we cannot apply \Cref{thm:Morsebasic}. 
	However, we will show that $d$ can be framed as a \emph{topological Morse function} that shares properties of smooth Morse functions described by classical Morse theory. Topological Morse functions either have the form $f \sim \mathrm{const.} + x_n$ at topological regular points, or $f \sim \mathrm{const.} - \sum_{i = 1}^{\lambda} x_i^2 + \sum_{i = \lambda + 1}^{n} x_i^2$ at non-degenerate topological critical points, up to a homeomorphism between open neighborhoods.
	
	We first show in \Cref{sec:topological_morse_theory} that the fundamental Morse lemmas formulated in the smooth setting also hold in the more general setting of topological Morse functions.
	This allows us to generalize the Morse lemmas to signed distance functions later on in \Cref{sec:generalization_dist_functions}: the idea is to successively reframe them as $C^k$-Min-type functions, under some geometric assumptions, and then as topological Morse functions.
	Finally, in \Cref{sec:genericity}, we show that these geometric assumptions are not restrictive and hold for generic signed distance functions.
	
	\subsection{Topological Morse Functions}
	\label{sec:topological_morse_theory}

	The main obstruction to generalizing classical Morse theory to continuous functions is the requirement of continuous gradients and gradient flows that underpin the results of \Cref{thm:Morsebasic}. This motivated the study of continuous functions that carry over generalized notions of critical and regular points from smooth functions. One such family are \emph{topological Morse functions} (originally called topological non-degenerate functions) introduced by \cite{morse_topologically_1959}.
	
	\begin{defi}[Topological regular point \citep{morse_topologically_1959}]
		\label{def:topological_regular_point}
		Let $f : X \to \R$ be a continuous function. A point $y \in X$ is a \emph{topological regular point} of $f$ if there is a homeomorphism $\varphi : M_1 \to M_2$ between open neighborhoods $M_1$ of $0$ in $\R^n$ and $M_2$ of $y$ with $\varphi(0) = y$, such that for all $x = (x_1, \ldots, x_n) \in M_1 \subset \R^n$, 
		\begin{equation}\label{eq:topregpoint}
			f \circ \varphi \ (x) = f(y) + x_n.
		\end{equation}
	\end{defi}
	
	For notational convenience, we write $f \overset{\varphi}{\sim} f(y) + x_n$ where $\varphi : (M_1, 0) \to (M_2, y)$.
	
	\begin{defi}[Topological critical point \citep{morse_topologically_1959}]
		\label{def:topological_critical_point}
		Let $f : \R^n \to \R$ be a continuous function. 
		A point $y \in \R^n$ is said to be a \emph{topological critical point} of $f$ if $y$ is not a topological regular point of $f$. Furthermore, $y$ is a \emph{non-degenerate topological critical point} of $f$ if there exists an integer $0 \leq \lambda \leq n$ and a homeomorphism $\phi : N_1 \to N_2$ where $N_1$ and $N_2 = \phi(N_1)$ are open neighborhoods of $0$ and $y$ in $\R^n$, such that $\phi(0) = y$ and for all $x \in N_1$,
		\[f \circ \phi \,(x) = f(y) - \sum_{i = 1}^{\lambda} x_i^2 + \sum_{i = \lambda + 1}^n x_i^2.\]
	\end{defi}
	For notational convenience, we write $f \overset{\phi}{\sim} f(y) - \sum_{i = 1}^{\lambda} x_i^2 + \sum_{i = \lambda + 1}^n x_i^2$ where $\phi : (N_1, 0) \to (N_2, y)$.
	Notice that a non-degenerate topological critical point cannot be a topological regular point \citep{morse_topologically_1959}. We also note that non-degenerate topological critical points are \textit{isolated}.

	\begin{defi}[Topological Morse function \citep{morse_topologically_1959}]
		\label{def:topological_morse_function}
		A continuous function $f : \R^n \to \R$ is a \emph{topological Morse function} if all the points $y \in \R^n$ are either topological regular points or non-degenerate topological critical points of $f$.
	\end{defi}

	The literature dealing with topological Morse functions appears to be sparse (see for instance \cite{cantwell_topological_1968,morse_f_1973,landis_tractions_1975}); in particular, there does not currently exist full counterparts to the isotopy and handle attachment lemmas in the same form as that given by \cite{milnor_morse_1963}. Essay III by \cite{kirby_foundational_1977} asserts that the handle attachment lemma is still true for topological Morse functions, by mentioning that a proof by \cite{siebenmann_deformation_1972} based on mutually transverse foliations exists. However some details of the proofs appear to be missing to allow their complete reconstruction.
	For completeness, we provide the full proofs based on purely topological arguments; some of them are similar to those found in \cite{cantwell_topological_1968} (see Lemma 3.1) and \cite{corvellec_deformation_1993} (see Theorem 2.8).
	
	\subsubsection{The Isotopy Lemma}
	
	We now show that the isotopy lemma that describes deformation retractions between sublevel sets of smooth functions over intervals of regular values  (\Cref{thm:Morsebasic}\ref{thm:Morsebasic_1}) generalizes to topological Morse functions. 
	
	Unlike the smooth setting, we do not have gradient vector fields that generate a flow to deform sublevel sets onto one another. To create a deformation retraction from $\sublevel{f}{b}$ onto $\sublevel{f}{a}$ without a vector field, we form an open cover of the interlevel set $\intlevel{f}{a}{b}$, and perform successive $f$-decreasing deformations on the individual cover elements, on each of which $f$ is simply the projection onto a Euclidean axis as specified in \cref{eq:topregpoint} of \Cref{def:topological_regular_point}. 
	
	Recall that a continuous function $f : \R^n \to \R$ is \emph{proper} if for any compact subset $K \subset \R$, $f^{-1}(K) \subset \R^n$ is compact.
	\begin{thm}[Isotopy Lemma, topological version]
		\label{thm:topological_isotopy_lemma}
		Let $f : \R^n \to \R$ be a continuous proper function, and $a < b$. Suppose that all points $y \in f^{-1}[a,b]$ are topological regular points for $f$.
		
		Then there exists a strong deformation retraction sending $f^{-1}(-\infty, b]$ onto $f^{-1}(-\infty, a]$. As a consequence, $f^{-1}(-\infty, b]$ and $f^{-1}(-\infty, a]$ are homotopy equivalent.
	\end{thm}
	
	\begin{proof}
		The strategy of the proof is to consider arbitrary level sets $f^{-1}(c)$, small enough $\eps$, and then apply successive deformations of $\R^n$ that only move points inside finitely many subsets that form an open cover of the interlevel set $f^{-1}[c-\eps, c+ \eps]$. The aim of these deformations is to deform the interlevel set $f^{-1}[c-\eps, c+ \eps]$ into the level set $f^{-1}(c-\eps)$ by lowering the values of $f$. To facilitate this construction, we consider the \emph{open solid cylinder} $CL(b,s) \subset \R^n$, which is the open set bounded by the cylinder whose base disk is $D^{n-1}(0, b)$ along the $n-1$ first coordinates and side is $D^1(0, s)$ along the last coordinate: 
		\[ CL(b,s) = \{x = (x_1,\ldots, x_n) \in \R^n ~|~ x_1^2 + \ldots + x_{n-1}^2 < b^2 \text{ and } x_n^2 < s^2 \}.\] 
		
		\textbf{Finite Open and Compact Covers of $f^{-1}[c-\eps, c+\eps]$.}
		Let $c \in [a,b]$ and consider the compact set $f^{-1}(c)$.
		It can be covered by open sets $U_{y}$ that are neighborhoods of points $y \in f^{-1}(c)$ where $f$ can be written in a regular form over $U_y$: 
		$f \overset{\varphi_y}{\sim} c + x_n$ with $\varphi_y : (V_y, 0) \to (U_y, y)$.
		By choosing small enough $\eps_y > 0$ and $b_y > 0$, we may suppose that $V_y = CL(b_y, \eps_y)$ are of open solid cylinder form. In particular, we have $f(U_y) = f(\varphi_y(V_y)) = (c-\eps_y, c+\eps_y)$.
		
		By compactness, $f^{-1}(c)$ may be covered by a finite number of the sets $U_y$:
		\[f^{-1}(c) \subset W = \bigcup\limits_{i = 1, \ldots, I} U_{y_i}.\]
		Then there exists $\eps > 0$ small enough such that $f^{-1}[c-\eps, c+ \eps] \subset W$. Otherwise, there exists sequences $(y_\nu)_\nu \subset \R^n$, $(\eps_\nu)_\nu \to 0$ such that $y_\nu \in f^{-1}[c-\eps_\nu, c+ \eps_\nu] \setminus W$, i.e., $f(y_\nu) \in [c-\eps_\nu, c+ \eps_\nu]$ and $y_\nu \notin W$. This implies that $f(y_\nu) \to c$. Because $f$ is proper, there exists a compact set $K$ such that for all $\nu$, $y_\nu \in K$. Up to extracting a subsequence, $y_\nu$ converges to $y \in K \setminus W$. By continuity $f(y) = \lim f(y_\nu) = c$, so then $y \in f^{-1}(c) \setminus W$, a contradiction. Therefore, such an $\eps$ exists and we have
		\begin{equation}
			\label{eq:open_covering}
			f^{-1}[c-\eps, c+ \eps] \subset W = \bigcup\limits_{i = 1, \ldots, I} U_{y_i}.
		\end{equation}
		
		We can furthermore suppose that $\eps < \min\limits_{i = 1, \ldots, I} \eps_{y_i}$. Hence, for $y = y_i$, $\varphi_{y}^{-1}( f^{-1}[c-\eps, c+ \eps] \inter U_{y})$ is a smaller closed solid cylinder $\overline{CL(b_y, \eps)}$ included in $V_y = CL(b_y, \eps_y)$.
		
		For simplicity, we drop the $y$ in the subscript of $U_{y_i}$ and simply write $U_i$ or $V_i$ instead, where the understanding is that we are referring to $y_i$.
		
		In fact, the open covering in \eqref{eq:open_covering} can be refined into a compact covering. Indeed, there exists a family of closed cylinders $\overline{CL(b_i', \eps_i')}$ with $0 < b_i' < b_i$ and $\eps < \eps_i' < \eps$, with $\mu = b_i - b_i' = \eps_i - \eps_i'$ small enough such that they still cover the interlevel set up to the homeomorphisms $\varphi_i$:
		\begin{equation}
			\label{eq:compact_covering}
			f^{-1}[c-\eps, c+ \eps] \subset \bigcup\limits_{i = 1, \ldots, I} K_i,
		\end{equation}
		where $K_i = \varphi_i(\overline{CL(b_i', \eps_i')}) \subset U_i$ are compact sets.
		Otherwise, there exist sequences $(\mu_\nu)_\nu \to 0$, $(y_\nu)_\nu \in \R^n$ such that we successively obtain the following:
		
		\begin{align*}
			y_\nu &\in f^{-1}[c-\eps, c+ \eps] \setminus \bigcup\limits_{i = 1, \ldots, I} \varphi_i(\overline{CL((b_i')_\nu,\, (\eps_i')_\nu)}) \\
			y_\nu &\in \bigcup\limits_{i = 1, \ldots, I} \varphi_i(CL(b_i, \eps_i)) \setminus \varphi_i(\overline{CL((b_i')_\nu,\,  (\eps_i')_\nu)})\\
			y_\nu &\in \bigcup\limits_{i = 1, \ldots, I} \varphi_i(DL_{i, \nu}),
		\end{align*}
		where $DL_{i, \nu} = CL(b_i, \eps_i) \setminus \overline{CL((b_i')_\nu,\,  (\eps_i')_\nu)}$. This implies that $\dist(y_\nu, (\union_i U_i)^C) \to 0$.
		Because $f$ is proper, up to extracting a converging subsequence, we obtain $y_\nu \to y \in f^{-1}[c-\eps, c+ \eps] \setminus \union_i U_i$, a contradiction. Therefore \eqref{eq:compact_covering} holds for small enough $\mu$.
		
		\textbf{Local Deformations.}
		We now define a family of local deformations corresponding to each of the $U_{i}$.

		By convention, if $x \in \R^n$, we write $x = (\zeta, x_n)$ where $\zeta = (x_1,\ldots,x_{n-1})$.
		
		Consider the following map: For $x \in \R^n$, $t \in [0,1]$, let
		\[ H_i(x,t) =
		\begin{cases}
			x & \text{ if } x \notin V_i \text{ or } x \in V_i \inter \{x_n \leq - \eps\} \\
			(\zeta,\, \max(-\eps, x_n - t \, \delta_i(x) )) & \text{ if } x \in V_i \inter \{x_n > - \eps\}
		\end{cases}
		\]
		where $\delta_i$ is defined similarly to the function introduced by \cite{cantwell_topological_1968}:
		\[\forall\ x \in V_i, \quad \delta_i(x) = \frac{\eps_i}{2} \, \left(1 - \frac{|\zeta|}{b_i} \right) \left(1 - \frac{x_n^2}{\eps_i^2} \right) > 0 .\]
		Note that $H_i(x,0) = x$. The quantity $\delta_i(x)$ can be understood as a maximal potential amount of decrease in the last coordinate applied to $x$. 
		For a given starting point $x$ at $t = 0$, the path formed by the points $H_i(x, \cdot)$ moves down linearly in time along the last dimension; at $t = 1$ the last coordinate has decreased by $\delta_i(x)$, unless the path hits the level $\{x_n = -\eps\} \inter V_i$ at an earlier timepoint $t_{\mathrm{hit}} \leq 1$ (which means that $H_i(x,t_\mathrm{hit}) = (\xi, -\eps$)), in which case it stays there for $t \geq t_{\mathrm{hit}}$.
		
		By compactness, there exists some $\eta > 0$ such that $\forall\ i = 1, \ldots, I$, $\forall\ x \in \varphi_i^{-1}(K_i)$, $\delta_i(x) \geq \eta$.
		
		Let $\Pi_n(x') = x_n'$ denote the projection onto the last coordinate.
		
		$H_i : \R^n \times [0,1] \to \R^n$ is continuous and satisfies the following properties:
		
		\begin{itemize}
			\item $\forall\ x \in \R^n, H_i(x,0) = x$;
			\item $\forall\ x \in V_i^C \union (V_i \inter \{x_n \leq - \eps\}), \forall\ t \in [0,1],\, H_i(x,t) = x$;
			\item $\forall\ x \in V_i \inter \{x_n > - \eps\}$, \,
			$ -\eps \leq \Pi_n ( H_i(x,1) ) < x_n$;
			\item $\forall\ x \in \varphi_i^{-1}(K_i)$,\,
			$ -\eps \leq \Pi_n ( H_i(x,1) ) \leq \max(-\eps,\, x_n - \eta).$
		\end{itemize}
		
		We can carry this deformation $H_i$ into the original domain through the homeomorphism $\varphi_i$. We define:
		\[G_i(y,t) :=
		\begin{cases}
			\varphi_i \circ H_i(\cdot, t) \circ \varphi_i^{-1} &\text{ if } y \in \varphi_i(V_i) \\
			y &\text{ if } y \notin \varphi_i(V_i).
		\end{cases}
		\]
		
		$G_i : \R^n \times [0,1] \to \R^n$ is continuous and satisfies the following:
		\begin{itemize}
			\item $\forall\ y \in \R^n,\, G_i(y,0) = y$;
			\item $\forall\ y \in U_i^C \union (U_i \inter \{f \leq c - \eps\}),\, \forall\ t \in [0,1],\, G_i(y,t) = y$;
			\item $\forall\ y \in U_i \inter \{f > c - \eps\}, c -\eps \leq f( G_i(y,1) ) < f(y)$;
			\item $\forall\ y \in K_i$,\,  
			$ c-\eps \leq f(G_i(y,1)) \leq \max(c - \eps,\, f(y) - \eta)$.
		\end{itemize}
		
		\textbf{Composition of Local Deformations.}
		Now consider the map $G : \R^n \times [0,I] $, which 
		results from the successive concatenation of these finitely many local deformations: if $t \in [0,1]$, we set $G(x,t) = G_1(x,t)$; if $t \in [1,2]$, we set $G(x,t) = G_2(\cdot, t - 1) \circ G_1(x,1)$, and so on.  If $t \in [I-1, I]$, we set $G(x,t) = G_{I}(\cdot,\, t - (I-1)) \circ G_{I-1}(\cdot, 1) \circ \cdots \circ G_1(x, 1)$. 
		This successive concatenation is denoted by $G = G_{I} \square \cdots \square G_1$.
		
		$G$ is continuous, and furthermore, using the compact covering $f^{-1}[c-\eps, c+\eps] \subset \bigcup\limits_{i = 1, \ldots, I} K_i$,
		\begin{itemize}
			\item $\forall\ y \in \R^n$, $G(y, 0) = y$;
			\item $\forall\ y \in f^{-1}(-\infty, c-\eps], \forall\ t \in [0, I]$, $G(y, t) = y$;
			\item $\forall\ y \in f^{-1}[c-\eps, c+\eps]$, 
			\[c - \eps \leq f(G(y,I)) \leq \max(c-\eps,\, c + \eps - \eta)\] because $y$ belongs to some $K_i$ and then the value of $f$ cannot be increased after applying the other deformations.
		\end{itemize}
		
		The last inequalities mean that $G( f^{-1}[c-\eps, c+\eps], I ) \subset f^{-1}[c-\eps, \max(c-\eps,c+\eps - \eta)]$.
		In fact, for the $k$-times self-concatenation of $G$ denoted by $G^{\square k}$, we similarly obtain that $G^{\square k}( f^{-1}[c-\eps, c+\eps], I ) \subset f^{-1}[c-\eps,\max(c-\eps, c+\eps - k\eta)]$. Then, for $k$ large enough, we have
		\[G^{\square k}( f^{-1}[c-\eps, c+\eps], I ) \subset \level{f}{c-\eps}.\]
		On the other hand, since the values of $f$ cannot increase over time through deformation, we have
		\[\forall\ t \in [0, k \, I], \quad G^{\square k}(f^{-1}(-\infty, c+\eps], t) \subset f^{-1}(-\infty, c+\eps].\]
		Therefore, $G^{\square k} : \R^n \times [0, k \, I] \to \R^n$, when restricted to $f^{-1}(-\infty, c+\eps]$ on the source and target spaces, provides a strong deformation retraction of the space sending $f^{-1}(-\infty, c+\eps]$ onto $f^{-1}(-\infty, c-\eps]$:
		\begin{itemize}
			\item $\forall\ y \in \R^n$,\, $G^{\square k}(y, 0) = y$;
			\item $\forall\ y \in f^{-1}(-\infty, c-\eps], \forall\ t \in [0, I]$,\, $G^{\square k}(y, t) = y$;
			\item $\forall\ y \in f^{-1}(-\infty, c+\eps]$,\, $G^{\square k}(y, k\,I) \in f^{-1}(-\infty, c-\eps]$.
		\end{itemize}
		
		\textbf{Homotopy Equivalence of $f^{-1}(-\infty, a]$ and $f^{-1}(-\infty, b]$.}
		We return to the original problem and see that for any value $c \in [a,b]$ and sufficiently small $\epsilon > 0$, the sublevel set $\sublevel{f}{c-\epsilon}$ is a deformation retract of $\sublevel{f}{c+\epsilon}$.
		
		The compact interval $[a,b]$ may be covered by a finite number of such intervals associated to $c_1, \ldots, c_N$. We may suppose that the value $a$ belongs to the last interval $[c_N-\eps_{c_N}, c_N+ \eps_{c_N}]$ only. Then we obtain the desired strong deformation retraction, by successively applying the strong deformation retractions associated to $c_1, \ldots, c_{N-1}$, and then adapting the definition of the local deformations for the last value $c_N$, for instance:
		\[ H_i(x,t) =
		\begin{cases}
			x & \text{ if } x \notin V_i \text{ or } x \in V_i \inter \{x_n \leq -(c_N - a)\} \\
			(\zeta,\, \max(-(c_N - a), x_n - t \, \delta_i(x) )) & \text{ if } x \in V_i \inter \{x_n > -(c_N - a)\}
		\end{cases}.
		\]
		This completes the proof.\end{proof}
	
	\subsubsection{The Handle Attachment Lemma}
	
	We now derive the counterpart of the handle attachment lemma of smooth Morse functions (\ref{thm:Morsebasic_2} of \Cref{thm:Morsebasic}) for topological Morse functions. Having derived the isotopy lemma for topological Morse functions, we can adopt the proof in \cite{milnor_morse_1963} for the smooth case for our purposes here. 
	
	\begin{thm}[Handle Attachment Lemma, topological version]
		\label{thm:topological_handle_lemma}
		Let $f : \R^n \to \R$ be a continuous proper function.
		Suppose that there exists $c \in \R$ and $\eps > 0$ such that 
		all points in $f^{-1}[c-\eps,c+\eps]$ are topological regular points for $f$, except a single point $q$ that is a non-degenerate critical point with index $\lambda$ and value $d(q) = c$. 
		Then $f^{-1}(-\infty, c + \eps]$ has the homotopy type of $f^{-1}(-\infty, c - \eps]$ with a $\lambda$-cell attached: 
		\[f^{-1}(-\infty, c + \eps] \simeq f^{-1}(-\infty, c - \eps] \union e^\lambda.\]
	\end{thm}
	
	\begin{proof}
		Let $\phi : (N_1, 0) \to (N_2, q)$ be a homeomorphism such that	
		\[f \overset{\phi}{\sim} c - \sum_{i = 1}^{\lambda} x_i^2 + \sum_{i = \lambda + 1}^n x_i^2.\] 
		Furthermore, assume that $N_1 = B(0,r)$ is a small open ball around $0$. Denote the quadratic part by $h(x) := - \sum_{i = 1}^{\lambda} x_i^2 + \sum_{i = \lambda + 1}^n x_i^2$.
		
		Notice that by applying the Isotopy Lemma of \Cref{thm:topological_isotopy_lemma} as needed, we can assume that $\eps$ is small enough so that the ellipsoid $\mathcal{E} = \{x = (x_1,\ldots, x_n) \in \R^n ~|~ x_1^2 + \ldots + x_{\lambda}^2 + 2 \, (x_{\lambda + 1}^2 + \ldots + x_n^2) \leq 2\, \eps \}$ is strictly contained inside of $B(0, r)$, while maintaining the statement of the Handle Attachment Lemma in the same form as given above. 
		
		\textbf{Auxiliary Function.}
		We construct an auxiliary function $F$ with respect to the set $U = \phi(B(0,r))$, following the prescription of \cite{milnor_morse_1963} for the smooth case. Let $\mu : \R \to \R$ be a smooth bump function such that
		\begin{align*}
			\mu(0) &> \eps & \\
			\mu(s) &= 0 & \text{ for } s \geq 2\,\eps, \\
			\mu'(s) &\in (0,1) & \forall\ s \in \R.
		\end{align*}
		Let $F : \R^n \to \R$ be defined as follows:
		\begin{numcases}{F(y)=}
			f(y) &$\text{ if } y \notin U,$ \nonumber\\
			f(y) - \mu( \tilde{h}(\phi^{-1}(y)) ) &$\text{ if } y \in U,$ \label{eq:2nd_line}
		\end{numcases}
		where $\tilde{h}(x) = x_1^2 + \ldots + x_\lambda^2 + 2 \, (x_{\lambda + 1}^2 + \ldots + x_n^2)$.
		For $x \in B(0,r)$, \eqref{eq:2nd_line} also rewrites as 
		$$F(\phi(x)) = c + h(x) - \mu(\tilde{h}(x)) = c - \sum\limits_{i \leq \lambda} x_i^2 + \sum\limits_{i \geq \lambda + 1} x_i^2 - \mu\left( \sum\limits_{i \leq \lambda} x_i^2 + 2 \sum\limits_{i \geq \lambda + 1} x_i^2 \right).
		$$
		
		Because $\mu$ vanishes beyond the ellipsoid $\mathcal{E}$, which is equal to $\{x \in \R^n ~|~ \tilde{h}(x) \leq 2\,\eps\}$, the auxiliary and original functions coincide outside $\phi(\mathcal{E}^o)$: $\forall\ y \notin \phi(\mathcal{E}^o),\,  F(y)= f(y)$. 
		
		\textbf{Retractions Over Regular Points $F^{-1}[c-\eps,c+\eps]$. }
		We can verify that the constructed function $F$ is continuous and proper.  We show that $F$ only admits topological regular points in the interlevel set $F^{-1}[c-\eps,c+\eps]$.
		
		Let $y \in F^{-1}[c-\eps,c+\eps]$. 
		If $y \in F^{-1}[c-\eps,c+\eps] \setminus \phi(\mathcal{E})$, or equivalently, $y \in f^{-1}[c-\eps,c+\eps] \setminus \phi(\mathcal{E})$, we know that $F$ and $f$ coincide on a neighborhood of $y$. By assumption on $f$, there exists $\varphi : (M_1, 0) \to (M_2, y)$ such that $f \circ \varphi(x) = f(y) + x_n$, i.e., $F \circ \varphi(x) = F(y) + x_n$ by taking $M_2$ a small enough neighborhood around $y$ such that $F = f$ on it.
		
		On the other hand, if $y \in F^{-1}[c-\eps,c+\eps] \inter \phi(\mathcal{E})$, then $w = \phi^{-1}(y) \in (F \circ \phi)^{-1}[c-\eps,c+\eps] \inter \mathcal{E} \subset B(0,r)$. $F \circ \phi$ is smooth on $B(0,r)$, so based on a local integral flow defined by the non-vanishing gradient vector field around $w$,
		there exists $\theta : (M_0, 0) \to (M_1, w)$, where we can suppose $M_1 \subset B(0,r)$, such that $F \circ \phi \circ \theta(x) = F(w) + a_n$. Then by composition, $F$ rewrites in regular form up to the homeomorphism $\phi \circ \theta : (M_0, 0) \to (\phi(M_1),y)$.
		
		By \Cref{thm:topological_isotopy_lemma} applied to $F$, the sublevel sets $F^{-1}(-\infty, c-\eps]$ and $F^{-1}(-\infty, c+\eps]$ are homotopy equivalent.
		
		\textbf{Handle of Auxiliary Function Retracts into $\lambda$-cell.}
		As in \cite{milnor_morse_1963}, we have that $\{F \leq c + \eps\} = \{f \leq c + \eps\}$. We also have that $\{F \leq c - \eps\} = \{f \leq c - \eps\} \union L$, where $L = \{F \leq c - \eps\} \setminus \{f \leq c - \eps\} = \phi(L')$ is the image via $\phi$ of a handle $L'$ in the ``smooth" domain: $L' = \{h - \mu \circ \widetilde{h} \leq -\eps \} \setminus \{h \leq -\eps\})$ is homeomorphic to $(D^\lambda \times D^{n - \lambda},\, \bord D^\lambda \times D^{n - \lambda})$ ($D^s$ refers to the unit disk in dimension $s$), is contained in the ellipsoid $\mathcal{E}$, and is attached to $\{h \leq - \eps\}$.
		
		In the smooth case, \cite{milnor_morse_1963} showed that 
		the handle can be retracted into a $\lambda$-cell, $(e^{\lambda})'$ attached to $\{h \leq - \eps\}$; more precisely, that
		$\{F \circ \phi \leq c - \eps\} = \{f \circ \phi \leq c - \eps\} \union L' \cong \{f \circ \phi \leq c - \eps\} \union (e^{\lambda})'$. By composing the homeomorphism $\phi$, we obtain that there is a $\lambda$-cell $e^\lambda$ such that the following spaces are homotopy equivalent:
		$$
		\{f \leq c + \eps\} = \{F \leq c + \eps\} \simeq  \{F \leq c - \eps\} \simeq \{f \leq c - \eps\} \union e^\lambda.
		$$
		This completes the proof. \end{proof}
	
	\begin{rmk}
		\label{rmk:several_handles_topological}
		As in \cite{milnor_morse_1963} (see \Cref{thm:Morsebasic}, the proof of \Cref{thm:topological_handle_lemma} may be adapted to show that if all points in $f^{-1}[c-\eps, c+\eps]$ are regular, except for a finite number of critical points $q_1, \ldots, q_N$ sharing the same critical value $c$ and with indices $\lambda_1, \ldots, \lambda_N$, then
		\[\sublevel{f}{c+\epsilon} \simeq \sublevel{f}{c-\epsilon} \union e^{\lambda_1} \union \ldots \union e^{\lambda_N}.\]
	\end{rmk}
	
	\textbf{Topological Morse Functions and Persistence Modules.}
	Having established the Morse handle attachment and isotopy lemmas for topological Morse functions, we can make the usual inferences on the persistence module of the sublevel set filtration of $f$, and establish the same observations we had for smooth Morse functions in  \Cref{cor:PM_morse_function_tame} for topological Morse functions. 
	\begin{cor} 
		\label{cor:PM_topological_morse_function_tame} Let $f: \R^n \to \R$ be a proper topological Morse function with finitely many topological critical points. Then:
		\begin{enumerate}[label = (\roman*)]
			\item The persistence module of the sublevel set filtration $f$ is pointwise finite dimensional (p.f.d.);
			\item For all $k$, the barcode $\Bar_k(f)$ is a finite multiset of intervals $\{ [b,d) \subset \bar\R\}$;
			\item A critical point with index $\lambda$ either corresponds to the birth of an interval in $\Barc{\lambda}{f}$ with homology dimension $\lambda$, or the death of an interval in $\Barc{\lambda-1}{f}$ with homology dimension $\lambda - 1$.
		\end{enumerate} 
	\end{cor}

	\subsection{Generalization to Distance Functions}
	\label{sec:generalization_dist_functions}
	
	We can now combine the results of the previous sections on Min-type functions and topological Morse functions to eventually deduce the fundamental Morse lemmas for signed distance functions.
	
	In \Cref{sec:Euclidean_distance_functions}, we introduce the setting of general Euclidean distance functions (signed and unsigned); we discuss their properties and in particular their (non)-smoothness, and concentrate on the case of signed distance functions to smooth surfaces.
	Then in \Cref{sec:NDG_crit_pts_dist_functions}, we specify some geometric conditions for Euclidean distance functions to admit Min-type non-degenerate critical points, that in turn are also non-degenerate in the topological sense (\Cref{def:topological_critical_point}).
	In \Cref{sec:morse_istopy_handle_distance}, we extend Morse's isotopy and handle attachment lemmas to the signed distance, based on the topological Morse theorems of \Cref{sec:topological_morse_theory} and the Min-type results of \Cref{sec:min-type}.
	
	Distance functions, signed or unsigned, appear in many different fields, alongside the medial axis, which is the set of points whose distance-to-boundary is realized by at least two closest points. In topological data analysis, open sets were shown to be homotopy equivalent to their medial axis, by considering an extended gradient flow of the distance \citep{lieutier_any_2003, chazal_stability_2004}. In the finite setting, the medial axis is simply the Voronoi diagram of the point cloud \citep{lieutier_any_2003,attali_stability_2009}.
	In the field of PDEs, the distance is viewed as a viscosity solution of the Eikonal equation $\vert \grad d \vert = 1$, and the homotopy equivalence was shown in the more general Riemannian setting using a generalized gradient flow \citep{albano_singular_2013}. The medial axis and distance function were also studied in non-smooth analysis and singularity theory \citep{cheeger_critical_1991,birbrair_medial_2017}.
	The distance field (signed or not) and medial axis are involved in multiple fields of applications, ranging from computer graphics to shape analysis \citep{lee_medial_1982, lindquist_medial_1996, sigg_signed_2003, park_deepsdf_2019}. The signed version is also related to phase-field representations, that have numerous applications \citep{song_generation_2022}.
	
	\subsubsection{Distance Function to Subsets of Euclidean Space}
	\label{sec:Euclidean_distance_functions}
	We first review some fundamental concepts and facts about distance functions to subsets of Euclidean space. For $A \subset \R^n$, the distance function $\dist(\cdot, A): \R^n \to [0, \infty)$ is given by
	\begin{equation}
		\dist(q, A) = \inf_{y \in A} \| q - y \|.
	\end{equation}
	If $A$ is compact, then for any $q$ the infimum over $y \in A$ is attained by some point in $A$. If $p \in A$ realizes the distance from $q$ to $A$, i.e., if $\dist(q, A) = \| q - p \|$, then $p$ is a \textit{contact point} or closest point of $q$ in $A$. We let $\Gamma(q)$ denote the set
	of contact points:
	\[\Gamma(q) = \{ p \in A ~|~ \dist(q,p) = \dist(q,A) \}.\]
	The \emph{medial axis} of $A$ is the set of points in $\R^n$ that admit more than one contact point in $A$ \citep{federer1959curvature}
	\[\mathcal{M}_{A} = \{ q \in \R^n ~|~ |\Gamma(q)| \geq 2 \}.\]
	The closure of the medial axis in $\R^n$ is called the \textit{cut locus} $\overline{\mathcal{M}_{A}}$ of $A$. Off the medial axis, we let $\xi: \R^n \setminus \mathcal{M}_{A} \to A$ denote the projection map that sends a point in Euclidean space to its unique contact point in $A$.
	
	As the distance function is 1-Lipschitz, Rademacher's theorem implies it is differentiable almost everywhere. The following theorem by \cite{federer1959curvature} describes where the distance function is continuously differentiable and gives an expression of its derivatives.  
	
	\begin{thm} \emph{(Theorem 4.8(3-5), {\cite{federer1959curvature}})}.
		For $A$ a non-empty, closed subset of $\R^n$, the nearest neighbor map $\xi: \R^n \setminus \mathcal{M}_{A} \to A$ is continuous, and the distance function $\dist(\cdot, A)$ is continuously differentiable on $ \R^n \setminus (A \union \cutlocus{A})$. Where $\dist(\cdot, A)$ is differentiable and $q \notin A$, we have
		\begin{equation}
			\nabla \dist(q, A) = \frac{q - \xi(q)}{\| q - \xi(q) \|}.\label{eq:classic_grad_dist}
		\end{equation}
	\end{thm}
	In particular, where $\dist(\cdot, A)$ is differentiable, the distance function satisfies the Eikonal equation $\| \nabla \dist(\cdot, A) \| = 1$. In \cite{lieutier_any_2003}, a discontinuous vector field called the \emph{extended gradient} $\nabla \dist(q,A)$ (with a slight abuse of notation) is defined on all of $\R^n$, such that it vanishes on $A$, coincides with the gradient of the distance function on $\R^n \setminus (\cutlocus{A} \cup A)$. The extended gradient is given explicitly as 
	\begin{align}\label{eq:extend_grad_dist}
		\nabla \dist(q,A) &= \frac{q - c(q)}{\dist(q,A)}
	\end{align}
	for points on the cut locus $\cutlocus{A}$; $c(q)$ denotes the center of the unique smallest closed ball enclosing $\Gamma(q)$. \cite{lieutier_any_2003} shows that if $A$ is the boundary of a bounded open set, a discrete \emph{Euler scheme} integrating the extended gradient produces a continuous semi-flow in the limit of step size going to zero; along such a flow, the distance function is non-decreasing. \cite{lieutier_any_2003} also shows that the points where the extended gradient vanishes obstruct the deformation retraction of sublevel sets onto another along the aforementioned semi-flow.

	\textbf{Signed Distance Functions to Surfaces.}
	\label{sec:properties_signed_dist}
	If a subset $A$ is the boundary of a bounded open subset, then we can modify the distance function to $A$ to encompass this extra information. Let $\W \subset \R^n$ be a \textit{non-empty bounded open subset} whose boundary  $\Surf = \bord \W$ is $C^k$ with $k \geq 2$. Given $\W$, we partition the ambient Euclidean space as
	\begin{equation} \label{eq:partition}
		\R^n = \W^- \sqcup \Surf \sqcup \W^+
	\end{equation}
	where $\W^- = \W$ and $\W^+ = (\W^\mathrm{c})^\mathrm{o}$. We let $\sgn: \R^n \to \{\pm 1\}$ be the function that labels $\sgn(q) = - 1$ if $q \in \W^-$ and $\sgn(q) = +1$ if $q \in \W^+$. 
	
	The \textit{signed distance function} $d : \R^n \to \R$ associated to an open subset $\W$ with boundary $\Surf$ is given by $d(q) = \sgn(q) \cdot \dist(q, \Surf)$. Equivalently, 
	\begin{equation}
		d = \dist(\cdot, \W) - \dist(\cdot, \W^c) = \dist(\cdot, \W^-) - \dist(\cdot, \W^+). \label{eq:signed_d_def}
	\end{equation}
	Note that the \textit{pure distance function} $\dist(\cdot, \Surf) : \R^n \to \R$  (i.e.,~unsigned distance function) is related to the one-sided distances by addition:
	\[ \dist(\cdot, \Surf) =  \dist(\cdot, \W^-) + \dist(\cdot, \W^+). \]
	
	The partition of the ambient Euclidean space given by \cref{eq:partition} induces a partition of the medial axis into two disjoint subsets, the \emph{inner medial axis} $\mathcal{M}_{\Surf}^{-} = \W_- \cap \mathcal{M}_{\Surf}$, and the \emph{outer medial axis} $\mathcal{M}_{\Surf}^{+} = \W_+ \cap \mathcal{M}_{\Surf}$. The inner and outer cut loci are similarly defined. The medial axis is illustrated in \Cref{fig:petal_SDF}.
	
	\begin{figure}[h!]
		\centering
		\includegraphics[clip, width=\linewidth]{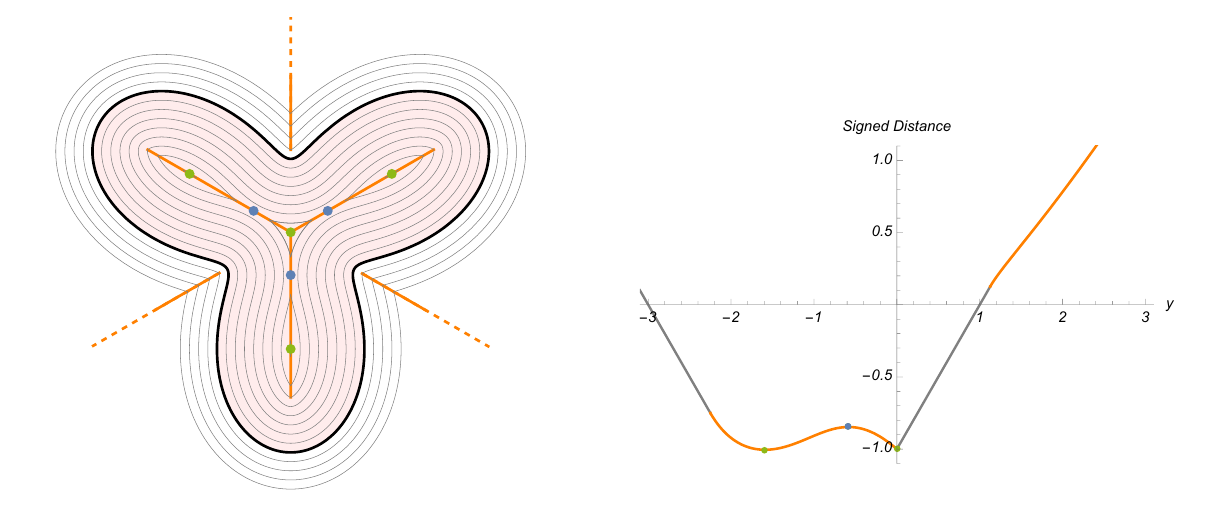}
		\caption{\textbf{Signed distance field for a petal-shaped curve.} Left: level sets of the field (gray) generated by the curve (black), its cut locus (orange), and its critical points (green for local minima, blue for saddle points). The curve is parameterized by $r(\theta) = 2 + \sin 3 \, \theta$ in polar coordinates. Right: profile of the signed distance, restricted to the vertical axis of the left subfigure. Dots and orange portions indicate the presence of the critical points and the cut locus, with blue dots for minima and green dots for saddle points. The signed distance field is not smooth in $\R^n$, in particular at points belonging to the cut locus.
		}
		\label{fig:petal_SDF}
	\end{figure}
	
	On $\R^n \setminus (\Surf \union \cutlocus{\Surf})$, where $\dist(\cdot, \Surf)$ is differentiable, the gradient of $d$ is 
	\begin{equation}\label{eq:classic_grad_signdist}
		\grad d(q) = \sgn(q) \grad \dist(q, \Surf) = \sgn(q) \, \frac{q - \xi(q)}{\|q-\xi(q)\|}.
	\end{equation}

	On a tubular neighborhood of $\Surf$, the pure and signed distance functions inherit the smoothness of the surface: if $\Surf$ is $C^k$ for $k \geq 2$, there exists a tubular neighborhood $\mathcal{N}_\mu = \{x \in \R^n ~|~ \dist(x, \Surf) \leq \mu \}$ of $\Surf$, with $\mu > 0$, on which the signed distance $d$ is $C^k$
	\citep[Theorem 8.2]{delfour_shapes_2011}. In particular, the gradient of \cref{eq:classic_grad_signdist} extends smoothly across the surface $\Surf$, and $d$ satisfies the Eikonal equation $\|\nabla d\| = 1$ on $\R^n \setminus \cutlocus{\Surf}$. In contrast, 
	the pure distance $\dist(\cdot, \Surf)$ is $C^k$ only in $\mathcal{N}_\mu \setminus \Surf$, with a discontinuous gradient across $\Surf$ \cite[Lemma 14.16]{gilbarg_elliptic_1977}, \citep{krantz_distance_1981}.
	
	We can similarly define an extended gradient for $d$, building on the extended gradient for the pure distance function. For $q \in \R^n \setminus \Surf$, the extended gradient $\grad d$ (via a slight abuse of notation) of the signed distance function is given by
	\begin{align}	
		\grad d(q) = \sgn(q) \cdot \grad \dist(q, \Surf)  = \frac{q - c(q)}{d(q)}.
		\label{eq:extend_grad_signdist}
	\end{align}
	These vector fields are illustrated in \Cref{fig:petal_gradient_fields}. Having associated an extended vector field with the signed distance function, we can now use it to define critical points of the signed distance function. 
	
	\begin{figure}[h!]
		\centering
		\includegraphics[clip, width=\linewidth]{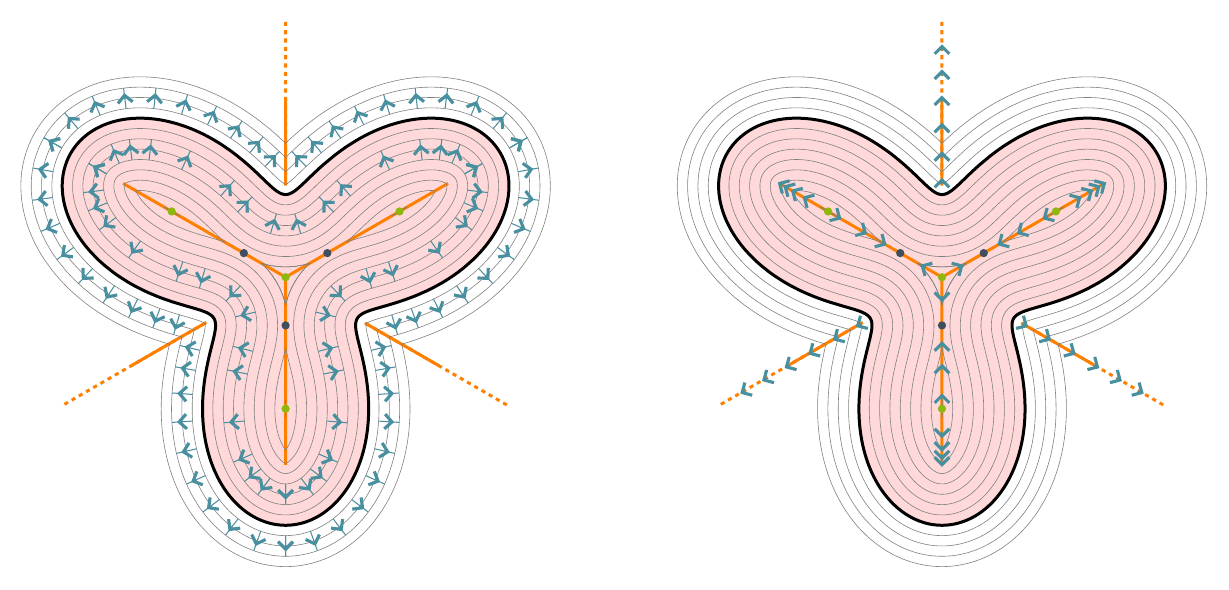}
		
		\caption{\textbf{Extended gradient field of the signed distance function.} Left: the gradient field in $\R^n \setminus \overline{\mathcal{M}_{\Surf}}$, defined in the classical sense, satisfies the Eikonal equation $|\grad d| = 1$. Right: the extended gradient field on the cut locus $\overline{\mathcal{M}_{\Surf}}$. Critical points are represented as dots, where by definition the extended gradient field vanishes, with blue for minima and green for saddle points. \label{fig:petal_gradient_fields}
		}
		
	\end{figure}

	\begin{defi}[Critical points of distance functions, signed and unsigned]
		\label{def:critical_point_dist_functions}
		A point $q \in \R^n$ is a \emph{critical point of the signed distance function} $d$ if its extended gradient vanishes at $q$:
		\[ \grad d(q) = 0.\]
		Notice that $q \in \mathcal{M}$ necessarily, since $\grad d$ is of unit norm outside the medial axis.
		
		Similarly, a point $q \in \R^n$ is a \emph{critical point of the pure distance function} $\dist(\cdot, \Surf)$ if its extended gradient $\nabla \dist(\cdot, \Surf)(q) = 0$. In particular, any point on $\Surf$ is a critical point of $\dist(\cdot, \Surf)$. On $\R^n \setminus \Surf$, $\dist(\cdot, \Surf)$ and $d$ share the same critical points.
	\end{defi}
	
	We can distinguish \textit{inner and outer critical points} that lie on $\mathcal{M}_{\Surf}^{-}$ and $\mathcal{M}_{\Surf}^{+}$, 
	so that $\mathrm{Crit}(\dist(\cdot, \Surf)) = \mathrm{Crit}^{-} \sqcup \Surf \sqcup \mathrm{Crit}^{+}$,
	$\mathrm{Crit}(\dist(\cdot, \W^-)) = \Surf \sqcup \mathrm{Crit}^{+}$, $\mathrm{Crit}(\dist(\cdot, \W^+)) = \Surf \sqcup \mathrm{Crit}^{-}$, and $\mathrm{Crit}(d) = \mathrm{Crit}^{-} \sqcup \mathrm{Crit}^{+}$.
	
	\Cref{def:critical_point_dist_functions} can in fact be reformulated into several equivalent forms for $\dist(\cdot, \W)$ in the Euclidean setting \citep{lieutier_any_2003,chazal_stability_2007}, and for $d_p$ in the Riemannian setting \citep{grove_generalized_1977, cheeger_critical_1991}.
	
	\begin{lem}
		\label{lem:equivalent_definitions_critical_point}
		For $q \notin \Surf$, the following are equivalent and characterize critical points for both $d$ and $\dist(\cdot, \Surf)$:
		\begin{enumerate}
			\item $\grad d(q) = 0$;
			\item $\grad \dist(\cdot, \Surf)(q) = 0$;
			\item $q = c(q)$ is the center of the unique smallest closed ball
			containing $\Gamma(q)$;
			\item $q$ belongs to the closed convex hull of its contact points $\Gamma(q)$;
			\item there is no open half-space containing $\Gamma(q)-q$;
			\item for any vector $v \in \R^n$, there exists $p \in \Gamma(q)$ such that 
			$v \cdot (p-q) \geq 0$.
		\end{enumerate}
	\end{lem}
	\begin{proof}
		Conditions 1.~and 2.~are equivalent, due to \cref{eq:extend_grad_signdist}. Conditions 2.~and 3.~are equivalent by definition of $c(q)$.
		Conditions 2.~and 4.~are known to be equivalent \cite[Lemma 2.2]{chazal_stability_2007}. 
		Conditions 4.~and 5.~are equivalent due to the hyperplane separation theorem.
		Finally, conditions 5.~and 6.~ are equivalent: 
		there is an open half-space containing $\Gamma(q) - q = \{p - q\}_{p \in \Gamma(q)}$ if and only if we can find a vector $v$ (perpendicular to the hyperplane and pointing in the other direction) such that $\forall\ p \in \Gamma(q)$, $v \cdot (p-q) < 0$.
	\end{proof}
	Condition 4.~was proposed as a definition of critical point by \cite{ferry_when_1976} in $\R^n$.
	Conditions 5.~and 6.~are adapted from \cite{grove_generalized_1977,cheeger_critical_1991}, where $q$ is critical if for any tangent vector $v \in T_q M$, there exists a geodesic joining $q$ to $p$ whose speed vector $\dot\gamma(0)$ forms an angle less or equal to $\pi / 2$ with $v$.
	
	\begin{rmk}
		It can also be shown that the points where the extended gradient of the (signed) distance function vanish on the medial axis are precisely the critical points of the Clarke subgradient of $d$ \citep{clarke1990optimization}. As mentioned previously in \Cref{rmk:clarke1}, the set of critical points of the extended gradient is thus a superset of the points that obstruct the deformation of sublevel sets of (signed) distance functions onto one another. 
	\end{rmk}

	\subsubsection{Min-type Non-Degenerate Critical Points of Distance Functions}
	\label{sec:NDG_crit_pts_dist_functions}
	
	We now give sufficient conditions where a distance function can be locally modeled as a $C^k$-Min-type function (\Cref{lem:dist_as_min_type}). Using those conditions, we derive in conditions where a critical point of the distance function is a Min-type non-degenerate critical point (\Cref{def:non_degen_crit}) in \Cref{def:ndg_crit_dist} and \Cref{thm:normal_form_dist}, and consequently, a topological non-degenerate critical point endowed with a Morse index (\Cref{def:topological_critical_point}). We restrict to the setting where a point in $\R^n \setminus \Surf$ can only have finitely many contact points on the surface. Conditions on the surface $\Surf$, the bounded open set $\W$ that the surface bounds, and the signed distance function $d$ are defined in \Cref{sec:properties_signed_dist}.

	We first establish conditions where the distance function $\dist(\cdot, \Surf)$ can be described as locally Min-type at a point $q$, then consider additional constraints that allow us to find an efficient representation (\Cref{def:efficient_repr}) of $\dist(\cdot, \Surf)$ at $q$. Recall an efficient representation of a Min-type function $f$ is a local expression
	\[ f = \min \{\alpha_1, \ldots,  \alpha_m\}\]
	where $m$ is the minimal number of functions for any expression of $f$ as a minimum over a collection of functions. We proceed with the following extension of Lemma 3.4 in \cite{birbrair_medial_2017}, which relates $\dist(\cdot, \Surf)$ at $q$ to local neighborhoods of contact points $\Gamma(q)$ in $\Surf$.
	
	\begin{lem}[Existence of contact pieces] \label{lem:pieces}
		Let $q \in \R^n \setminus \Surf$ and suppose that $m = |\Gamma(q)| < \infty$. Denote the contact points of $q$ by $p_1,\ldots,p_m$. 
		Then there exists $r > \dist(q, \Surf)$ such that $\overline{B}(q,r) \inter \Surf$ contains pairwise-disjoint connected closed subsets $S_i \subset \Surf$,
		and a single $p_i$ belongs to the interior of $S_i$:
		\begin{equation} \label{eq:pieces}
			\overline{B}(q,r) \inter \Surf ~ \supset ~ \bigsqcup\limits_{i=1}^m S_i \quad \text{with} \quad p_i \in \overset{\circ}{S_i}.
		\end{equation}
		Moreover, by setting $\alpha_i = \dist(\cdot, S_i)$, there exists a neighborhood $N$ of $q$ on which 
		\begin{equation} \label{eq:alpha_i}
			\forall\ x \in N, \quad \dist(x, \Surf) = \min \limits_{i = 1,\ldots,m} \alpha_i(x).
		\end{equation}
		
		The sets $\{S_i\}$ are said to form a family of \emph{contact pieces} of $q$.
	\end{lem}
	\begin{proof}
		First, we show that $r$ can be chosen such that the connected components of $\overline{B}(q,r) \inter \Surf$ contain at most one point $p_i$ each.
		
		Consider disjoint closed balls $\overline{B_i}$ centered at each $p_i$.
		There exists $r > \dist(q, \Surf)$ such that $\overline{B}(q,r) \inter \Surf \subset \sqcup_i (\overline{B_i} \inter \Surf)$. Otherwise, there is a sequence $(y_\nu)_{\nu \in \Nbb^*}$ such that $y_\nu \in \overline{B}(q, d(q,\Surf) + 1/\nu ) \inter \Surf$ and  staying in $\Surf \setminus \sqcup _i(\overline{B_i} \inter \Surf)$. Up to extracting a converging subsequence, $(y_\nu)$ converges to some $p \in \Surf$ with $\dist(q,p) = \dist(q,\Surf)$, while $p$ does not belong to $\sqcup_i (B_i \inter \Surf)$. This means that $p$ is a closest point of $q$ distinct from $p_1,\ldots,p_m$, which is a contradiction.
		
		While fixing $r$, denote by $S_i$ the connected component of $\overline{B}(x,r) \inter \Surf$ containing $p_i$, so that the inclusion \eqref{eq:pieces} is satisfied.
		Note that necessarily, $p_i \in (S_i)^\circ $ (for the induced topology of $\Surf$).
		
		Now, we show that there exists a ball $B(q,\eps)$ around $q$ such that the equality \eqref{eq:alpha_i} holds. First, notice that for any point $x$, $\dist(x, \Surf) \leq \min_i \dist(x, S_i)$. Suppose for contradiction that there is no $\eps$ such that $\dist(x, \Surf) \geq \min_i \dist(x, S_i)$ on $B(q,\eps)$. Then there would exist a sequence $(x_\nu)_{\nu \in \Nbb^*}$, with $x_\nu \in B(q,1/\nu)$, such that $\forall\ \nu, d(x_\nu,\Surf) < \min_i d(x_\nu, S_i)$. Up to extraction, we find contact points $y_\nu \in \Surf \setminus \sqcup_i S_i$ of $x_\nu$ that converge to some $y \in \overline{\Surf \setminus \sqcup_i S_i}$. Then $y$ also realizes the distance $d(q,\Surf)$, so $y$ is a contact point of $q$, for instance $p_1$. But then $y \in ( \Surf \setminus (S_1)^\circ ) \inter (S_1)^\circ = \emptyset$, a contradiction.
	\end{proof}
	
	This lemma states that, in the neighborhood of $q$, the distance field coincides with the minimum of the distance fields generated by isolated contact pieces around $\Gamma(q)$: points near $q$ only ``see" the pieces, as illustrated in \Cref{fig:petal_alpha_i}.
	
	\begin{figure}[h!]
		\centering
		\includegraphics[clip, width=.8\linewidth]{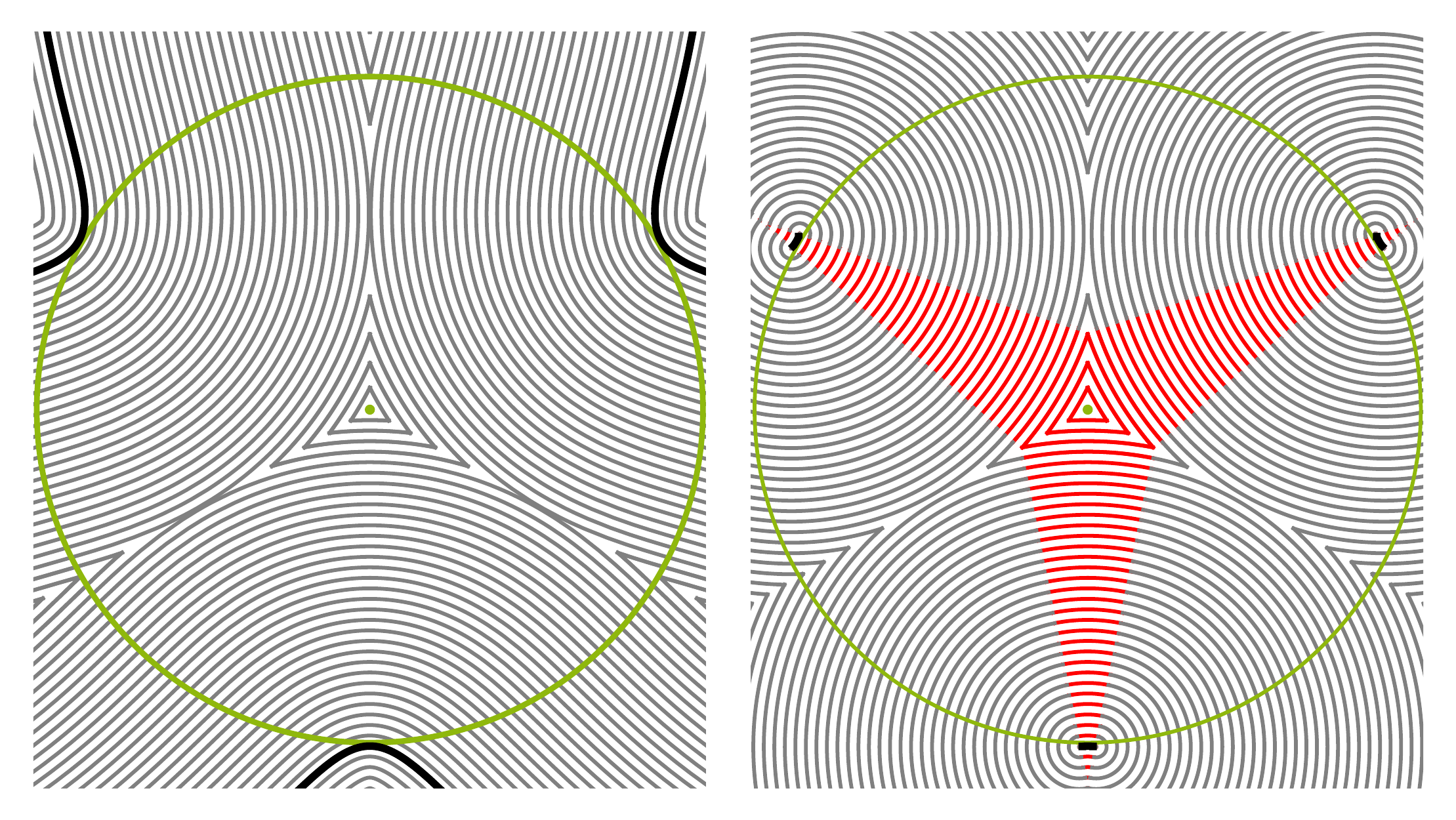}
		\caption{\textbf{Contact pieces and induced level sets.}.
			This plot compares the level sets of $\dist(\cdot, \Surf)$ (left) v.s. the level sets of $\min\{\alpha_1,\ldots,\alpha_m\}$ (right).
			Red-colored portions of the isolines of $\min\{\alpha_1,\ldots,\alpha_m\}$ (right) are those which coincide with the isolines generated by $\dist(\cdot, \Surf)$.
		}
		\label{fig:petal_alpha_i}
	\end{figure}

	While $|\Gamma(q)| < \infty$ implies $\dist(\cdot, \Surf)$ is a Min-type function at $q$, whether it is a $C^k$-smooth Min-type function depends on the regularity of $\alpha_i$. The smoothness of $\alpha_i$ will be determined by the curvature of $\Surf$ at the points $p_i$.

	The extrinsic curvature of $\Surf$ is described by the shape operator: Given the Gauss map $\nbf : \Surf \to \mathbb{S}^{n-1}$ (which points inwards into $\W^-$ by our convention),
	the \textit{shape operator} $d_p\nbf : T_p \Surf \to T_p \Surf$  
	is diagonalizable in the \textit{principal directions}:
	\[ d_p\nbf(e_i) = - \kap_i \, e_i,\] 
	with eigenvalues being the opposite of the \textit{principal curvatures} of $\Surf$ at $p$. We order them as $\kapmax = \kap_1 \geq \kap_2 \geq \ldots \geq \kap_{n-1} = \kapmin$.
	
	The factor determining whether $\alpha_i$ is smooth at $q$ depends on the curvature at $p$ and the distance between $p$ and $q$. We give the precise conditions below and an illustration in \Cref{fig:petal_contact_spheres}.
	
	\begin{defi}[Ball conditions] \label{def:Ball_conditions}
		Given a point $q \in \R^n \setminus \Surf$ and $p \in \Gamma(q)$, we say that $q$ satisfies the \textit{loose ball condition} at $p$ when
		\begin{equation} \label{eq:loose_ball_condition}
			\begin{cases}
				\kapmax(p) & \leq \frac{1}{\dist(q,\Surf)} \text{ if } q \in \W^-, \\
				\kapmin(p) & \geq - \frac{1}{\dist(q,\Surf)} \text{ if } q \in \W^+. 
			\end{cases}
		\end{equation}
		When the inequalities are strict, we say that $q$ satisfies the \textit{strict ball condition} at $p$:
		\begin{equation} \label{eq:strict_ball_condition}
			\begin{cases}
				\kapmax(p) & < \frac{1}{\dist(q,\Surf)} \text{ if } q \in \W^-, \\
				\kapmin(p) & > - \frac{1}{\dist(q,\Surf)} \text{ if } q \in \W^+. 
			\end{cases}
		\end{equation}
	\end{defi}
	
	\begin{figure}[h!]
		\centering
		\includegraphics[clip, width=\linewidth]{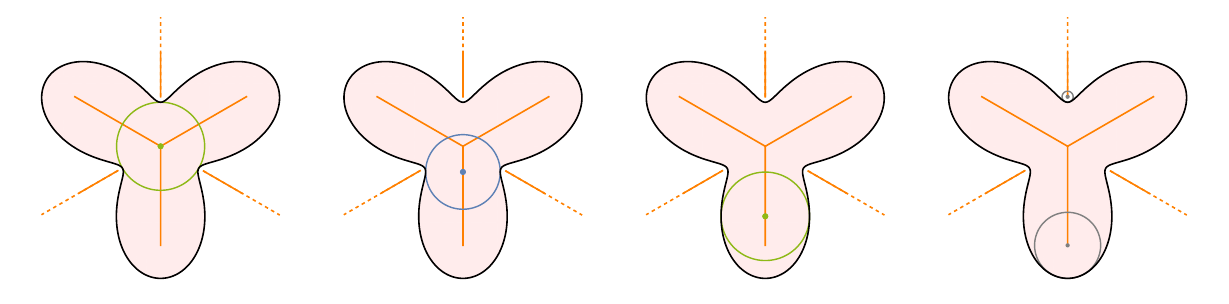}
		\caption{\textbf{Contact spheres and the strict ball condition}. From left to right: contact spheres with $3$, $2$, $3$ and $1$ contact points, respectively. The first three subfigures show critical points: a local minimum, a saddle point, a local minimum, respectively, whose contact spheres all satisfy the strict ball condition. The final subfigure show two regular points belonging to $\overline{\mathcal{M}_{\Surf}} \setminus \mathcal{M}_{\Surf}$, whose contact sphere violates the strict ball condition.
		}
		\label{fig:petal_contact_spheres}
	\end{figure}
	The lemma below says that the strict ball condition ensures that $\alpha_i$ is differentiable at $q$. 
	
	\begin{lem}[Distance as a $C^k$-Min-type function]
		\label{lem:dist_as_min_type}
		Consider $q \in \R^n \setminus \Surf$ and suppose that $m = |\Gamma(q)| < \infty$. Let $\{S_i\}$, and $\{\alpha_i\}$ be as in \Cref{lem:pieces}. 
		Then if $q$ satisfies the strict ball condition \eqref{eq:strict_ball_condition} at any contact point $p_i$, the functions $\alpha_i$ are $C^k$ in a neighborhood of $q$.
		Therefore, $\dist(\cdot, \Surf)$ is a $C^k$-Min-type function at $q$, and futhermore the representation $\dist(\cdot, \Surf) = \min \{\alpha_1,\ldots,\alpha_m\}$ is efficient at $q$.
		
	\end{lem}
	\begin{proof}
		We suppose without loss of generality that $q \in \W^-$. Note that if $m \geq 2$, then $q \in \mathcal{M}_{\Surf}$. In any case, $q \notin \mathcal{M}_{S_i}$ because $q$ admits a unique closest point in each piece, by construction of the contact pieces, but this does not say whether $q$ belongs to $\overline{\mathcal{M}_{S_i}}$ or not. In fact, the strict ball condition \eqref{eq:strict_ball_condition} implies that
		\begin{equation} \label{eq:not_in_pieces_cut_locus}
			\forall\ i = 1,\ldots,m, \quad q \notin \overline{\mathcal{M}_{S_i}}.
		\end{equation}
		Indeed, suppose for contradiction that for some $i$, $q \in \overline{\mathcal{M}_{S_i}}$. Then there is a sequence of points $(x_\nu)$ converging to $q$ that have at least two contact points on $S_i$. But it can be shown that there must be a point $y_\nu \in S_i$ whose curvature is greater or equal to the curvature of the contact ball of $x_\nu$: $\kapmax(y_\nu) \geq \frac{1}{\dist(x_\nu, S_i)} $, such that the sequence $(y_\nu)$ converges to $p_i$ (using arguments similar to \cite{niyogi_finding_2008}). 
		However, at the limit this contradicts the condition.
		
		Thus, there is a neighborhood $N$ of $q$ that does not meet the closed set $\overline{\mathcal{M}_{S_i}}$. Thus, $\alpha_i = \dist(\cdot, S_i)$ is differentiable on $N$. Any point $x \in N$ admits a unique projected point $\Pi_i(x)$ on $S_i$ and we have $\grad \alpha_i(x) = \frac{x - \Pi_i(x)}{d(x,\Surf)} = \nbf(\Pi_i(x))$ (see \Cref{sec:Euclidean_distance_functions}). 
		
		Next, to show that $\alpha_i$ is $C^k$ in a neighborhood of $q$, we adapt the proof of \cite[Lemma 14.16]{gilbarg_elliptic_1977}.
		
		Choose a principal coordinate system $(u,\varphi(u))$ of $\Surf$ around $p_i = (u_0, \varphi(u_0))$. Recall that $\Surf$ is locally the graph of the $C^k$ function $\varphi : T_{p_i}\Surf \, \inter \, U \to \R$ where $U$ is an open neighborhood of $p_i$ in $\R^n$, and $\grad \varphi(u_0) = 0$, and that the axes may be chosen such that the first $n-1$ coordinates are aligned with the principal directions. In this coordinate system the Hessian matrix $\Hess \, \varphi(u_0) = \text{diag}\left( \kap_1,\kap_2,\ldots,\kap_{n-1} \right)$ gives the principal curvatures.
		
		At a point $y(u) = (u, \varphi(u)) \in \Surf$ with $u \in T_{p_i}\Surf \, \inter \, U$, the inwards normal $\nbf(y(u))$ is a $C^{k-1}$ function w.r.t.~$u$, since
		\[ \nbf(y(u)) = \left(-\partial_1 \varphi(u),\ldots, -\partial_{n-1} \varphi(u), 1 \right) /  \sqrt{1 + |\grad \varphi (u)|^2}.\]
		With respect to the principal coordinate system, we have
		\[\partial_b( \nbf_a \circ y)(u_0) = - \kap_a \, \delta_{a,b} \quad \text{for } a,b = 1,\ldots, n-1.\]
		
		Consider the function $\Psi : (T_{p_i}\Surf \, \inter \, U) \times \R \to \R^n$ defined by
		\begin{equation}\label{eq:from_piece_to_distlevel}
			\Psi(u, t) = y(u) + t \, \nbf(y(u)), \quad y = (u, \varphi(u)).
		\end{equation}
		
		Setting $r = \dist(q,\Surf)$, we have $\Psi(u_0,r) = p_i + r \, \nbf(p_i) = q$.
		We know that $\Psi$ is $C^{k-1}$ w.r.t.~$(u,t)$ because $\nbf$ is $C^{k-1}$ w.r.t.~$u$. The Jacobian of $\Psi$ at $(u_0,r)$ is equal to 
		\[D\Psi(u_0,r) = \text{diag}\left(1-\kap_1 \, r, \ldots, 1-\kap_{n-1} \, r, 1 \right)\]
		and is invertible by the strict ball condition \eqref{eq:strict_ball_condition}:
		\[\det(D\Psi(u_0,r)) = (1-\kapmax \, r) \ldots (1-\kapmin \, r) > 0.\]
		By the inverse mapping theorem, $\Psi$ is invertible between neighborhoods of $(u_0,r)$ and $q$ and $\Psi^{-1}$ is also $C^{k-1}$. Thus, around $q$, the functions $x \mapsto u(x)$ and $x \mapsto y(u(x)) = (u(x), \varphi(u(x)))$ are $C^{k-1}$. So $\grad \alpha_i(x) = \nbf(\Pi_i(x)) = \nbf(y(u(x)))$ is also $C^{k-1}$. We conclude that there exists a neighborhood of $q$ on which $\alpha_i$ is $C^k$.
		
		We now show that $\{\alpha_1,\ldots,\alpha_m\}$ is efficient. 
		For any point $x$ on the open ray $(q,p_i)$ connecting $q$ and $p_i$,
		the only closest point on $S_1 \sqcup \ldots \sqcup S_m$ is $p_i$.
		This yields that the germ of $(q,p_i)$ 
		is a subset of $\{\forall\ j \neq i, \alpha_i < \alpha_j \} \subset A_i^\circ$. Hence, $A_i^\circ \neq \emptyset$ as a germ of set.  
		
		Let $\dist(\cdot, \Surf) = \min \{\beta_1,\ldots,\beta_l\}$ be another representation. We write $A_i = \union_j (A_i \inter B_j)$ and since $A_i^\circ \neq \emptyset$, by the Baire category theorem, there exists $j$ such that $(A_i \inter B_j)^\circ = A_i^\circ \inter B_j^\circ \neq \emptyset$, as germs of sets, up to extracting the index $j$. We can thus build a sequence $(x_\nu)$ converging to $q$ such that $\grad \alpha_i(x_\nu) = \grad \beta_j(x_\nu)$. At the limit, $\grad \alpha_i(q) = \grad \beta_j(q)$ by smoothness. Since the gradients $\{\grad \alpha_i(q) \}$ are pairwise distinct, the indices $i$ are injectively associated to the indices $j$, so $m \leq l$.
		This shows that $\{\alpha_1,\ldots,\alpha_m\}$ is an efficient representation.
	\end{proof}
	
	\begin{rmk}
		
		A point $q$ such that $|\Gamma(q)| < \infty$ already satisfies a loose ball condition \eqref{eq:loose_ball_condition}
		(see \cite[Proposition 6.1]{niyogi_finding_2008}).
	\end{rmk}
	
	\Cref{lem:dist_as_min_type} establishes the conditions where the distance function is locally a $C^k$-Min-type function, equipped with an explicit efficient representation. Thus,  combining \Cref{lem:dist_as_min_type}  and \Cref{def:non_degen_crit}, we can now give conditions where a point is a non-degenerate Min-type critical point of pure and unsigned distance functions. 
	
	\begin{prop}(Min-type non-degenerate critical point of the distance function)
		\label{def:ndg_crit_dist}
		Consider $q \in \R^n \setminus \Surf$, where
		\begin{enumerate}[label=(N\arabic*), ref=(N\arabic*), start=1]
			\item $\Gamma(q) = \{p_1,\ldots,p_m\}$ is finite; and \label{def:NDGC_Finite}
			\item The strict ball condition (\cref{eq:strict_ball_condition}) holds at any contact point $p_i$. \label{def:NDGC_StrictBall}
		\end{enumerate}
		Then $q$ is a \emph{Min-type non-degenerate critical point of the pure distance function $\dist(\cdot, \Surf)$} (and thus also of the signed distance function $d$), if and only if the following hold:
		\begin{enumerate}[label=(N3\alph*), ref=(N3\alph*)]
			\item $\Gamma(q)$ are in general position in $\R^n$ (efficient-LIG representation);\label{def:NDGC_LIG}
			\item $q$ is strictly in the interior of the convex hull of $\Gamma(q)$; and\label{def:NDGC_Convex} 
			\item The restriction $\dist(\cdot, \Surf)_{|G(q)}$ is a Morse function in a neighborhood of $q$, where $G(q)$ is the germ of set $\{\alpha_1 = \ldots = \alpha_m\}$ (\Cref{def:G(x)}), and  $\{\alpha_i\}$ is the family of distances to contact pieces of $q$. \label{def:NDGC_MorseGerm}
		\end{enumerate}
	\end{prop}
	
	Note that \ref{def:NDGC_Convex} is a stronger condition than the one presented in \Cref{def:critical_point_dist_functions} that defines the critical points of $\dist(\cdot, \Surf)$ and $d$ via extended gradient vector fields. From here on, a Min-type non-degenerate critical of either a pure or signed distance function will refer to a point satisfying the conditions listed in \Cref{def:ndg_crit_dist}. We give some examples of non-degenerate critical points in \Cref{fig:NDG_cases} and contrast them with degenerate critical points in \Cref{fig:DG_cases}.

	\begin{figure}[h!]
		\centering
		\includegraphics[clip, width=\linewidth]{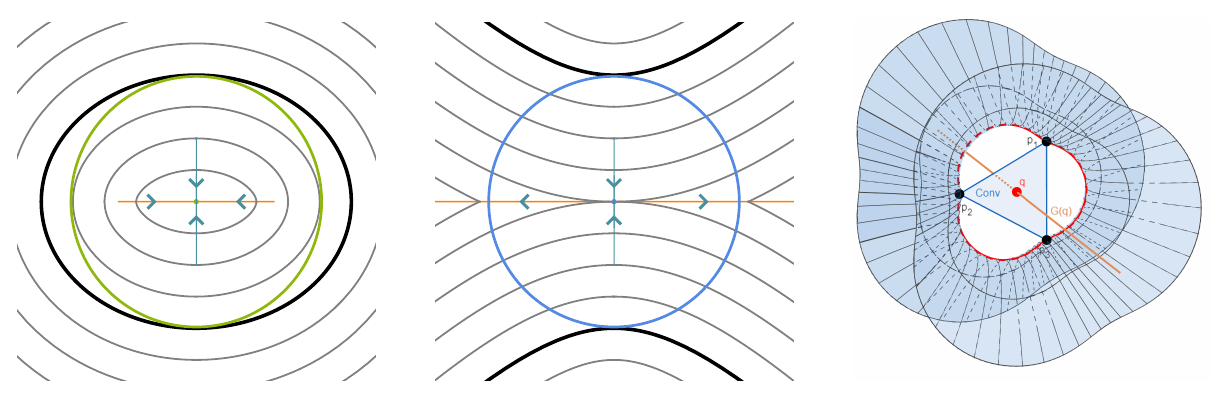}
		\caption{\textbf{Non-degenerate critical points and index.}
			The index can be interpreted as the sum of two terms $\lambda = (m-1) + \lambda'$ measuring how many directions where $q$ is a local maximum, one term for the dimension of the contact subspace spanned by $\Gamma(q)$, and one term for the restriction $\dist(\cdot,\Surf)_{\rvert G(q)}$. 
			Left: $q$ admits $2$ contact points in $\R^2$, is a local maximum of $\dist(\cdot, \Surf)$, with index $\lambda = 1 + 1 = 2$.
			Middle: $q$ admits $2$ contact points in $\R^2$, is a saddle point of $\dist(\cdot, \Surf)$, with index $\lambda = 1 + 0 = 1$.
			Right: $q$ admits $3$ contacts points in $\R^3$, is a critical point of $\dist(\cdot, \Surf)$, with index $\lambda = 2 + 0 = 2$.
			Blue arrows are aligned with the gradient field $\grad \dist(\cdot, \Surf)$.
		}
		\label{fig:NDG_cases}
	\end{figure}
	
	\begin{figure}[h!]
		\centering
		\includegraphics[clip, width=\linewidth]{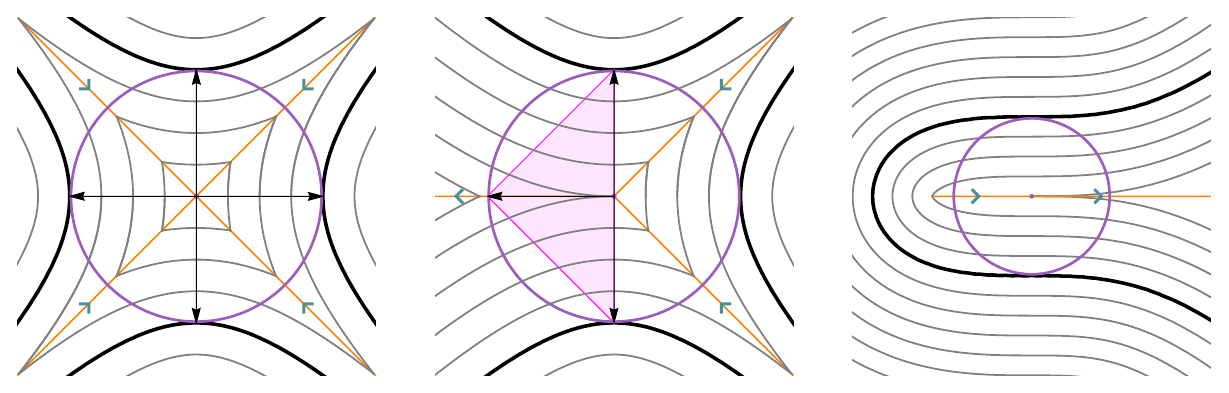}
		\caption{\textbf{Degenerate critical points in $\R^2$.} Three examples where conditions \ref{def:NDGC_LIG}, \ref{def:NDGC_Convex}, and \ref{def:NDGC_MorseGerm} of \Cref{def:ndg_crit_dist} are violated, respectively.
			Left: $q$ admits $4$ contact points that cannot form a $3$-dimensional subspace in $\R^2$.
			Middle: $q$ is located on the boundary of $\mathrm{Conv}(\Gamma(q))$, but not strictly inside.
			Right: the distance function restricted on $G(q)$ (which coincides with the medial axis here) is not Morse at $q$, as its Hessian is degenerate.
			In the three plots, $q$ is at the center of the contact sphere; the surface is represented in black, the medial axis in orange, and the level sets in gray. Black arrows are aligned with $\grad \alpha_i(q)$, blue arrows are aligned with $\grad \dist(\cdot, \Surf)$.
		}
		\label{fig:DG_cases}
	\end{figure}
	
	Applying \Cref{thm:normal_form}, we can thus obtain the Morse-like normal form for distance functions at a Min-type non-degenerate critical point $q$. In other words, the pure and signed distance functions at $q$ are locally homeomorphic to the normal form of a non-degenerate critical point of a smooth function. This idea is illustrated in \Cref{fig:NDG_levels}.
	
	\begin{thm}[Normal form for distance functions]
		\label{thm:normal_form_dist}
		Let $f = \dist(\cdot, \Surf)$ or $f = d$ denote either the pure or signed distance function. If $q \in \R^n \setminus \Surf$ satisfies the conditions listed in \Cref{def:ndg_crit_dist}, then $q$ is a topological non-degenerate critical point (\Cref{def:topological_critical_point}): i.e., there exists an almost smooth homeomorphism $\phi : U \to V$ with $\phi(0) = q$, and $U$ and $V = \phi(U)$ open neighborhoods of $0$ and $q$ in $\R^n$, such that for $x \in U$,
		\[ f \circ \phi \,(x) = f(q) - \sum_{i = 1}^{\lambda} x_i^2 + \sum_{i = \lambda + 1}^{n} x_i^2,\]
		where $\lambda = \idx(q; f)$. 
		Furthermore, the index of $q$ for the pure distance function is given by
		\begin{equation}
			\label{eq:idx_crit_pure}
			\idx{(q; \dist(\cdot, \Surf))} = (m-1) + \idx{(q; \dist(\cdot, \Surf)_{|G(q)})};
		\end{equation}
		and the index of $q$ for the signed distance function $d$ is given by
		\begin{equation}
			\idx(q;d) =	\begin{cases}
				n - \idx{(q; \dist(\cdot, \Surf))} &\text{if } d(q) < 0, \\
				\idx{(q; \dist(\cdot, \Surf))} &\text{if } d(q) > 0.
			\end{cases}
		\end{equation}
	\end{thm}
	
	As a consequence of the normal form, there can only be finitely many non-degenerate critical points for a surface $\Surf$. 
	
	\begin{cor} \label{cor:mintype_ndg_finite} The set of non-degenerate Min-type critical points of the pure or signed distance to the compact surface $\Surf$ (i.e., those that satisfy the conditions given in \Cref{def:ndg_crit_dist}) is finite. 
	\end{cor}
	
	\begin{proof}
		From \Cref{thm:normal_form}, the non-degenerate critical points described in \Cref{def:ndg_crit_dist} are isolated. Because any such critical point is contained in the convex hull of a set of points in $\Surf$, all critical points are contained in the convex hull of $\Surf$ (note that some critical points may belong to the outer set $\W^+$). As the critical points are isolated and contained in a compact subset, there can only be finitely many of them.
	\end{proof}

	\begin{rmk}
		We can compare the non-degeneracy conditions and the index from \Cref{def:ndg_crit_dist} to those of the distance function $d_\mathcal{P}$ to a point a cloud $\mathcal{P}$ \citep{bobrowski_distance_2014}, which is always $C^k$-Min-type in $\R^n \setminus \mathcal{P}$. A point $q \in \R^n \setminus \mathcal{P}$ is (non-degenerate) critical with index $m-1$ if it admits $m$ contact points $p_1,\ldots,p_m \in \mathcal{P}$ that are in general position around $q$, and $q$ is strictly in the interior of the convex hull spanned by the contact points.
		These conditions match with conditions \ref{def:NDGC_LIG} and \ref{def:NDGC_Convex} of \Cref{def:ndg_crit_dist}. Condition \ref{def:NDGC_MorseGerm} is always satisfied: $G(q)$ in that case is the cell of Voronoi diagram (medial axis) to which $q$ belongs.  Moreover, the restriction admits a local minimum at $q$ and the index reduces to $m-1$, the first term of \cref{eq:idx_crit_pure}. This observation is also true of single-point distance functions $d_p$ defined on negatively curved Riemannian manifolds $M$ \citep{gershkovich_morse_1997,itoh_cut_2007}.
	\end{rmk}

	\begin{figure}[h!]
		\centering
		\includegraphics[clip, width=.7\linewidth]{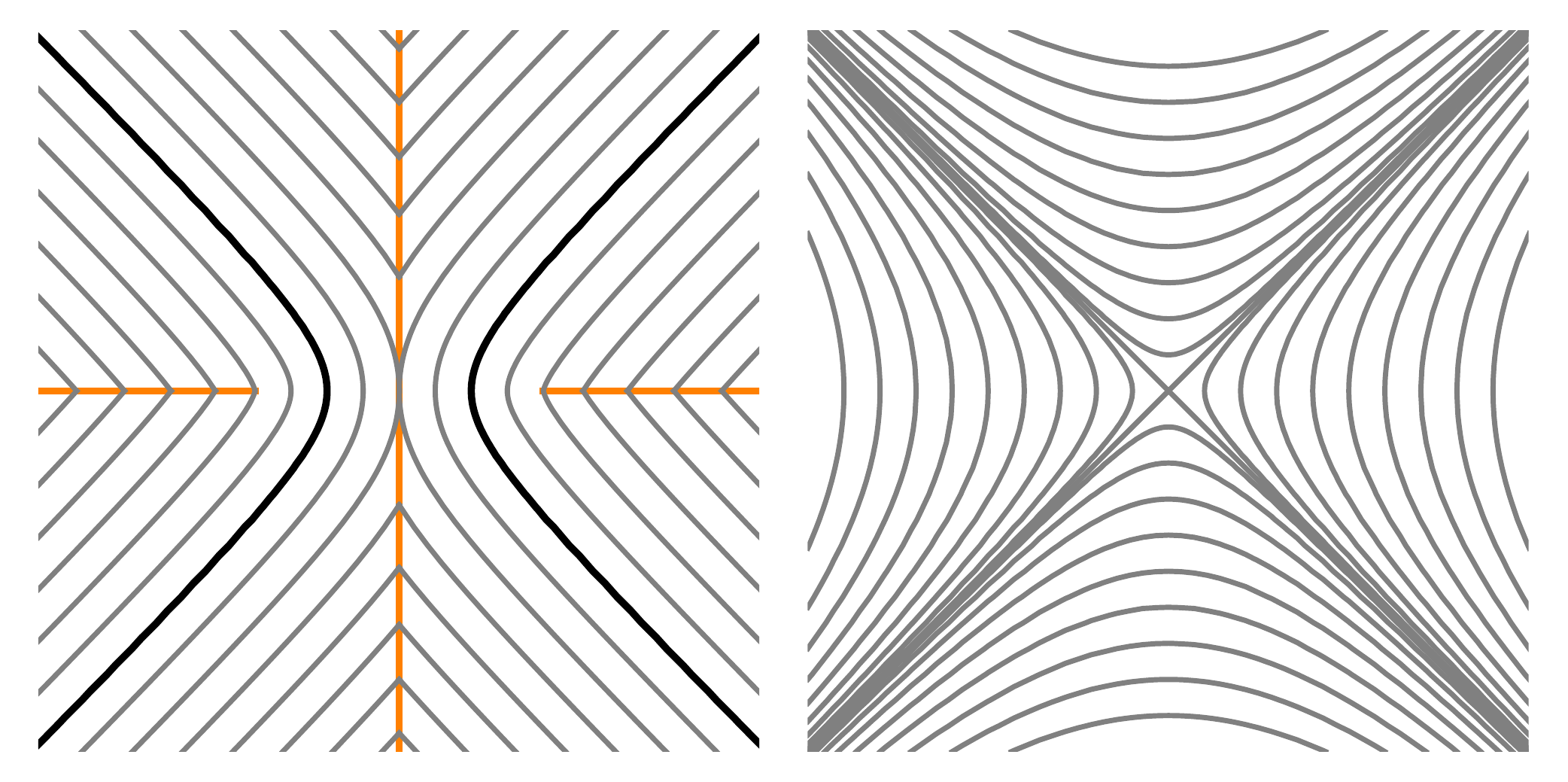}
		\caption{\textbf{Level sets around a non-degenerate critical point}, for a distance field (left) or a smooth function (right). In these examples defined in $\R^2$, the index is $\lambda = 1$ (corresponding to $m = 2$ contact points and $\lambda' = 0$ on the medial axis $G(q)$ on the left). In both cases, as soon as the level crosses the critical value at $q$, the sublevel set encounters a topological change determined by a $\lambda$-handle attachment.
		}
		\label{fig:NDG_levels}
	\end{figure}
	
	\subsubsection{Isotopy and Handle Attachment Lemmas}
	\label{sec:morse_istopy_handle_distance}
	
	We now show that Morse's isotopy and handle attachment lemmas hold for signed distance functions. 
	We combine the results of \Cref{sec:topological_morse_theory} on topological Morse functions to the existence of a normal form for $d$ around non-degenerate points given by the Min-type theory of \Cref{sec:min-type}.
	Recall that $\W$, $\Surf$, and $d$ are defined as in \Cref{sec:properties_signed_dist}.
	
	For a value $a \in \R$, we define the $a$-sublevel set
	\[V^a = d^{-1}(-\infty,a] = \{ x \in \R^n ~|~ d(x) \leq a \}. \]
	For two real values $a < b$, we define the $[a,b]$-interlevel set
	\[ V_a^b = d^{-1}[a,b] = \{ x \in \R^n ~|~ a \leq d(x) \leq b \}.  \]
	As $d$ is the signed distance function to a compact subset of Euclidean space, $d$ is proper and
	the sublevel sets are compact.
	
	\begin{thm}[Isotopy lemma for signed distance functions]
		\label{thm:dist_isotopy_lemma}
		Let $a < b$ be two real values. Suppose that $V_a^b$ contains no critical point of $d$ and that this set is compact. 
		Then $V^a$ is a deformation retract of $V^b$, giving that $V^a$ and $V^b$ are homotopy-equivalent.
	\end{thm}
	\begin{proof}
		
		$d$ is continuous and proper. Moreover, any point $y \in V_a^b$ is in fact a topological regular point (\Cref{def:topological_regular_point}), due to results given by \cite{cheeger_critical_1991} and \cite{grove_critical_1993}. There is a gradient-like vector field $W : \R^n \setminus \mathrm{Crit}(d) \to T\R^n$ such that for any point $y \in \R^n \setminus \mathrm{Crit}(d) \to T\R^n$, $W(y)$ defines an open half-space strictly containing all the contacts points : $\forall\ p \in \Gamma(y)$,
		$W(y) \cdot (p-y) < 0$. In a neighborhood $U_y$ of $y$ and for $\tilde{\eps} > 0$ small enough, this vector field induces a homeomorphism $V_{a}^b \inter U_y \simeq (U_y \inter \{d = a\}) \times [a,b]$ where $a = d(y) - \tilde{\eps}$ and $b = d(y) + \tilde{\eps}$, so that the values of $d$ are locally represented as the last coordinate (belonging to $[a,b]$) of some homeomorphic system of coordinates. We conclude by applying \Cref{thm:topological_isotopy_lemma}.
	\end{proof}
	
	\begin{rmk}
		The same result can be obtained by applying \cite[Proposition 1.8]{grove_critical_1993} multiple times; they also show that all levels $d^{-1}(t)$ are homeomorphic for $t \in [a,b]$ and $V_a^b$ is homeomorphic to $d^{-1}(a) \times [a,b]$.
	\end{rmk}
	
	\begin{thm}[Handle attachment lemma for signed distance functions]
		\label{thm:dist_handle_lemma}
		Let $q \in \R^n \setminus \Surf$ be a non-degenerate critical point of $d$ with index $\lambda$ and value $d(q) = c$ (i.e., one that satisfies the conditions of \Cref{def:ndg_crit_dist}). Suppose furthermore that, for some $\eps > 0$, the interlevel set $V_{c - \eps}^{c + \eps}$ contains no other critical point than $q$ and that it is a compact set. 
		
		Then $V^{c + \eps}$ has the homotopy type of $V^{c-\eps}$ with a $\lambda$-cell attached: 
		\[V^{c + \eps} \simeq V^{c-\eps} \union e^\lambda.\]
	\end{thm}
	\begin{proof}
		$d$ is a continuous and proper function. By \Cref{thm:normal_form_dist}, $d$ admits a normal form $d \sim \mathrm{cst} - \sum_{i = 1}^{\lambda} x_i^2 + \sum_{i = \lambda + 1}^{n} x_i^2 $ at $q$, so $q$ is also non-degenerate in the topological sense of \Cref{def:topological_critical_point}.
		On the other hand, any other point $y \in V_{c - \eps}^{c + \eps} \setminus \{q\}$ is a topological regular point in the sense of \Cref{def:topological_regular_point} (see the proof of \Cref{thm:dist_isotopy_lemma}).
		
		By applying \Cref{thm:topological_handle_lemma} we get
		\[V^{c+\eps} \simeq V^{c-\eps} \union e^\lambda.\]
	\end{proof}
	
	\begin{rmk}
		If the critical points of $d$ are non-degenerate, then there are only finitely many of them in any interlevel set $V_{c - \eps}^{c + \eps}$. This is due to non-degenerate critical points being isolated and the compactness of the interlevel set $V_{c - \eps}^{c + \eps}$, which follows from $\Surf$ being compact. As such, for the set of critical points $q_1, \ldots, q_N$ that lie in $\level{d}{c}$, there is a sufficiently small $\epsilon$ such that they are the only critical points in  $V_{c - \eps}^{c + \eps}$. We can use Remark \ref{rmk:several_handles_topological} to deduce
		\[V^{c+\eps} \simeq V^{c-\eps} \union e^{\lambda_1} \union \ldots \union e^{\lambda_N}\]
		where $\lambda_i$ denote the indices of $q_i$.
	\end{rmk}

	\subsection{Signed Distance Functions are Morse for Generic Surfaces}
	\label{sec:genericity}
	
	We have shown in \Cref{sec:morse_istopy_handle_distance} that Morse theory can be extended to the signed distance function $d$ if critical points of $d$ are $C^k$-Min-type non-degenerate; that is, if they satisfy the  geometric conditions \ref{def:NDGC_Finite} - \ref{def:NDGC_MorseGerm} outlined in \Cref{def:ndg_crit_dist}. 
	For surfaces in $\R^3$, these conditions restrict the surface from forming spherical caps and cylindrical necks around $q$. 
	We show that not only are these conditions satisfied for some surfaces in $\R^3$, but they are in fact satisfied for \textit{generic} embeddings of a surface into $\R^3$. 
	Heuristically, this means that the conditions are true for nearly every way of embedding the manifold as a surface in $\R^3$. 
	
	\begin{defi}[Generic property and residual sets]
		A property $(P)$ is said to hold for \emph{generic} elements of a topological space $E$, or \emph{generically}, if 
		$\{e \in E ~|~ P(e) \text{ is true} \}$
		is a \emph{residual set} in $E$, i.e.,
		is the countable intersection of dense open subsets of $E$.
	\end{defi}
	
	If $E$ is a complete metric space, a residual set is nonempty; in fact, by the Baire category theorem, a residual set is dense in $E$. Moreover, the intersection of countably many residual sets is also residual, so that if a countable family of properties are individually generic, then they are also simultaneously generic.
	
	In this section, we prove that the signed distance function is a topological Morse function with finitely many critical points for generic embeddings of surfaces in \Cref{thm:generic}. Our proof relies on \emph{transversality theory} \citep{guillemin_differential_2010}. Building up to the proof, we describe key concepts in transversality theory, and summarize key results in the literature which describe the structure of the cut locus for generic embeddings. This allows us to derive \Cref{thm:generic}.

	Before we begin, we make the class of surfaces that we are considering and the parameter space of manifold embeddings explicit. In this section, $M$ refers to a two-dimensional $C^k$-smooth ($k \geq 3$) manifold, that is orientable and closed (i.e., compact without boundary).
	We denote $\Emb$ to be the space of embeddings of $M$ into $\R^3$ endowed with the $C^k(M,\R^3)$ topology, defined in \cite{hirsch_differential_1976}. Note that for any embedding $\iota \in \Emb$, the embedded surface $\Surf_\iota = \iota(M)$ is closed and orientable.

	\subsubsection{Genericity of Transversality}
	Transversality describes the fact that intersections of spaces are not tangential. 
	
	\begin{defi}[Transversality]
		Two submanifolds $M$ and $N$ in a manifold $Y$ \textit{intersect transversely}, denoted by $M \pitchfork N$, if at each point of intersection $a \in M \inter N$, $T_a M + T_a N = T_a Y$. Similarly, let $f : X \to Y$ be a smooth map between smooth manifolds and denote its differential map at $x \in X$ by $D_x f : T_x X \to T_{f(x)} Y$. $f$ is \textit{transversal} to a submanifold $W$ of $Y$ and written as $f \pitchfork W$ if for any $x \in X$ such that $f(x) \in W$, 
		\[D_x f(T_x X) + T_{f(x)} W = T_{f(x)} Y .\]
	\end{defi}
	In particular, if we have an empty intersection $\Im f \inter W = \emptyset$, then we also have $f \pitchfork W$.
	Note that if the intersection is not void and contains $f(x)$, transversality implies:
	\[\rank(D_x f) + \dim W \geq \dim Y. \]

	It is important to note that \textit{transversality is a generic property}. For example, two lines in the plane generically intersect  at one point and cross each transversely. In three dimensions, two lines in $\R^3$ generically never intersect, yet they are nevertheless transversal to each other. If two lines do intersect in $\R^3$, they can always be perturbed around their intersection to avoid each other by arbitrarily small perturbations.
	
	We begin by stating the \textit{parametric transversality theorem}, which exists in various forms for finite-dimensional \citep{hirsch_differential_1976} and infinite-dimensional manifolds \citep{abraham_transversal_1967}.
	
	\begin{thm}[Parametric transversality theorem \citep{abraham_transversal_1967}]
		\label{thm:param_trans}
		Let $E$, $X$, $Y$ be $C^s$ manifolds, $W \subset Y$ a $C^s$ submanifold (not necessarily closed), and $F : E \times X \to Y$ such that the following conditions are satisfied:
		\begin{itemize}
			\item $X$ has finite dimension $\dim_X$ and $W$ has finite codimension $\mathrm{codim}_W$ in $Y$;
			\item $E$ and $X$ are second countable (i.e., their topology has a countable base);
			\item $F$ is $C^s$;
			\item $s > \max \{ 0, \dim_X - \mathrm{codim}_W\}$;
			\item $F \pitchfork W$.
		\end{itemize}
		Then the set
		\[ \{e \in E ~|~ F(e, \cdot) \pitchfork W \} \]
		is residual (hence dense) in $E$.
	\end{thm}
	In particular, $E$, $W$ and $Y$ may be infinite-dimensional.
	If $W$ and $Y$ are also finite-dimensional, then $s$ must be strictly greater than $\max \{ 0, \dim X + \dim W - \dim Y \}$.
	
	Using the parametric transversality theorem, we adopt the following proof strategy to show that the geometric conditions for the signed distance function to be topological Morse with finitely many critical points are generically satisfied (\Cref{thm:generic}): We will express each condition as a transversal intersection $F(e, \cdot) \pitchfork W$ between a map $F: E \times X \to Y$ and a submanifold $W \subset Y$ where $E$ is set to be the space of embeddings $\Emb$ of a manifold $M$ into $\R^3$. By carefully constructing the map so that the conditions of \Cref{thm:param_trans} are satisfied, we can show that the set of embeddings on which the geometric conditions are satisfied is residual. To set up the proof of \Cref{thm:generic}, we first describe the structure of the cut locus of generic surfaces. 
	
	\subsubsection{Generic Structure of the Cut Locus of Surfaces}
	
	We aim to show that, for generic embeddings, the critical points of the signed distance function $d$ satisfy the conditions listed in \Cref{def:ndg_crit_dist} to be non-degenerate Min-type critical points. These conditions describe geometric constraints between a critical point on the medial axis and its contact points on the surface. Our first step to show that these conditions are satisfied for a generic set of embeddings is to consider the possible configurations of the contact set for general points on the cut locus that need not be critical points of $d$. This subject is well studied in literature and we recall important results on the finiteness of the contact set and the generic local shape of the medial axis \citep{yomdin_local_1981, mather_distance_1983, giblin_symmetry_2000, cazals_differential_2005, damon_global_2006}. These results are summarized in \Cref{thm:generic_cut_locus}, which describes a classification of points on generic cut loci into five possible contact set configurations. 
	
	Using this description of the geometric configurations of contact sets, we observe that some of the conditions listed in \Cref{def:ndg_crit_dist} are already satisfied by some if not all types of points on the cut loci, and restrict the verification of the remaining conditions to a case-by-case analysis over a small number of contact set configurations. The following theorem by \cite{yomdin_local_1981} show that \ref{def:NDGC_Finite} and \ref{def:NDGC_LIG} are in fact generically satisfied by all points on the cut locus. 
	
	\begin{thm}[Finiteness of the contact set \citep{yomdin_local_1981}]
		\label{thm:finite_contact}
		For generic embeddings $\iota \in \Emb$ (where $k \leq 2$) , a point $q \in \R \setminus \Surf$ where $\Surf = \iota(M)$ only has a finitely many contact points with $|\Gamma(q)| = 1, 2, 3 \text{ or } 4$, and $\Gamma(q)$ are in general position.
	\end{thm}
	
	To study other constraints such as the strict ball condition \ref{def:NDGC_StrictBall}, we need a finer description of the contact set $\Gamma(q)$ involving the \emph{singularity types} of points in $\Gamma(q)$. In the following brief summary, we follow the classification of \cite{mather_distance_1983} and \cite{bruce1992curves} and refer to \cite{cazals_differential_2005} for a detailed discussion. For a point $q \in \R^3$ with finitely many contact points, and $p \in \Gamma(q)$, the squared distance function to $q$ restricted to $\Surf$ can be expressed on the local chart at $p$ as 
	\begin{equation} \label{eq:contact_pt_dist}
		g(x_1,x_2) = r^2 + x_1^2(1-r\kappa_1) + x_2^2(1-r\kappa_2) + \text{higher order terms},
	\end{equation}
	where $(x_1,x_2)$ are local coordinates along the principal directions of curvatures at $p = (0,0)$, and $r = \| p-q\|$. A contact point $p$ is an $A_1$ \emph{singularity} of $q$ if there is a diffeomorphism on the chart at $p$, such that \cref{eq:contact_pt_dist} can be written as 
	\[g = r^2 \pm x^2 \pm y^2.\]
	In other words, $p$ is a non-degenerate critical point of $g$. Note that $p$ is an $A_1$ singularity of $q$ only if the strict ball condition (\cref{eq:strict_ball_condition}) is true. $p$ is an $A_3$ \emph{singularity} of $q$ if there is a diffeomorphism on the chart at $p$, such that \cref{eq:contact_pt_dist} can be written as 
	\[g = r^2 \pm x^2 \pm y^4.\]
	Note that $p$ is then a degenerate critical point of $g$ and $r = 1/\kappa$ for one of the principal curvatures $\kappa$, as one of the leading quadratic terms of \cref{eq:contact_pt_dist} vanishes. Note that $p$ is an $A_3$ singularity of $q$ only if the strict ball condition (\cref{eq:strict_ball_condition}) is violated. 
	
	A complete list of singularities beyond types $A_1$ and $A_3$ can be found in \cite{arnol1974normal}, but as we shall see in \Cref{thm:generic_cut_locus}, points on the cut locus only possess contact points of types $A_1$ and $A_3$ for generic surfaces. We can thus classify points on the cut locus by enumerating the number of $A_1$ and $A_3$ singularities in its contact sets; $q \in \overline{\med}$ is said to be of $A_1^j A_3^k$ \emph{type} if $\Gamma(q)$ contains $j$ and $k$ many $A_1$ and $A_3$ singularities, respectively. In fact, the following lemma shows that \Cref{thm:finite_contact} implies the set of $A_1^j$ are $(4-j)$-dimensional submanifolds. The lemma also provides some intuition on the meaning of $G(q)$ (\Cref{def:G(x)}).
	
	\begin{lem} \label{lem:A1_manifolds} Consider a generic surface $\Surf$ satisfying the properties specified in \Cref{thm:finite_contact}. Let $q$ be an $A_1^j$ point for $j = 1,\ldots,4$. Then: 
		\begin{enumerate}
			\item There is an efficient-LIG representation of the distance function on a neighborhood $U$ of $q$
			$$ 
			\dist(\cdot, \Surf) = \min \{ \alpha_1, \ldots, \alpha_j\},
			$$
			where $\{\alpha_i = \dist(\cdot, S_i) \}$ are the distance functions to the contact pieces $S_i$ of $q$ as defined in \Cref{lem:pieces};
			\item There is a neighbourhood $V_1 \subseteq U$ of $q$ in $\R^3$, such that any point $q' \in V_1$ has at most $j$ many contact points, all of which are $A_1$ singularities; and 
			\item Let $\sigma = \{x \in U \ | \ \alpha_1(x) = \cdots = \alpha_j(x) = \dist(x, \Surf)\} \ni q$. Then there is a neighbourhood $V_2$ of $q$ such that $V_2 \cap \sigma $ is a $(4-j)$-dimensional manifold.
		\end{enumerate}
		Hence, there is a neighbourhood $V = V_1 \cap V_2$ of $q$, such that the set of $A_1^j$ type points in $V$ is the $(4-j)$-dimensional manifold $V \cap \sigma$. In other words, the set of $A_1^j$ points is equivalent to the $(4-j)$ dimensional manifold $G(q)$ (the germ of $\sigma$  at $q$ as defined in \Cref{def:G(x)}) as germs of sets at $q$.
	\end{lem}
	\begin{proof} \phantom{0}
		
		\begin{enumerate}
			\item Since an $A_1^j$ type point only has finitely many critical points and satisfy the strict ball condition (\cref{eq:strict_ball_condition}) by definition, \Cref{lem:dist_as_min_type} implies $\dist(\cdot, \Surf)$ is $C^k$-Min-type with an efficient representation $\dist(\cdot, \Surf) = \min \{ \alpha_1, \ldots, \alpha_j \}$ on a neighborhood $U$ of $q$, where $\alpha_i$ are distance functions to disjoint subsets $S_i$ of $\Surf$, each a neighborhood of $p_i \in \Gamma(q)$ in $\Surf$ as given by \Cref{lem:pieces}. The representation is efficient due to $\Gamma(q)$ being in general position as given in \Cref{thm:finite_contact}.
			\item  Consider $q' \in U$. Since we have on $U$ an efficient-LIG representation of the distance function, $q'$ has $j' \leq j$ contact pieces $S_1,\ldots, S_{j'}$, which are also the contact pieces of $q$ specified in the efficient-LIG representation; furthermore, since we have assumed the conditions specified in \Cref{thm:finite_contact}, the contact set is finite $\Gamma(q') = \{p'_1, \ldots, p'_j\}$ and $p_i' = \Gamma(q') \cap S_i$. We now investigate the singularity types of $p_i'$.  Recall an isolated point $p_i' \in \Gamma(q')$ is an $A_1$ singularity of $q'$ iff the principal curvatures at $p_i' \in \Surf$ are less than $r' = \| q' -p_i' \|$. We note that on $U$, the distance to the pieces $\alpha_i = \dist(\cdot, S_i)$ and the nearest neighbor projection map $\xi_i: U \to S_i$ onto the contact piece $S_i$ are continuous. We also observe that the principal curvatures $\kappa_\ell: \Surf \to \R$ for $\ell = 1,2$ are also continuous. We can thus construct a continuous map  $\eta_i: U \to \R$ given by $\eta_i(x)= 1-\alpha_i(x) \kappa_{\text{max}}(\xi_i(x))$, where $\kappa_{\text{max}} := \max\{ \kappa_1, \kappa_2 \}$. Note that $p_i' \in \Gamma(q')$ is an $A_1$ singularity of $q_i'$ iff $\eta_i(q_i') > 0$. Since $q$ is $A_1^j$ type, we have $\eta_i(q) > 0$ for all $i$. By continuity, there is a sufficiently small neighbourhood $W_i$ of $q$ for each map $\eta_i$, such that $\eta_i > 0$ on $W_i$ and the contact point $p_i'$  of $q' \in W_i$ in $S_i$ is an $A_1$ singularity. We can then take $V_1 = \cap_{i=1}^j W_i$ to be the neighbourhood of $q$ on which any $q' \in V_1$ has at most $j$ many contact points, all of which are $A_1$ singularities.
			\item  Since the representation is efficient-LIG, \Cref{lem:G(q)_mfd} implies there is a neighborhood $V_2 \subset U$ such that $V_2 \cap \sigma$ is a $(4-j)$-dimensional manifold.
		\end{enumerate}
		It follows that for $V = V_1 \cap V_2$, the submanifold $V \cap \sigma$ are $A_1^j$ type points; furthermore, since $A_1^j$ points on $U$ must belong to $\sigma$ by definition, $V \cap \sigma$ is precisely the set of $A_1^j$ points in the neighbourhood $V$ of $q$.
	\end{proof}
	\Cref{{lem:A1_manifolds}} shows that there are submanifolds in generic cut loci consisting of $A_1^j$ points. This is further developed by \cite{mather_distance_1983}, who showed that generic cut loci can be decomposed into manifolds of the same $A_1^jA_3^k$ type. Specifically, the decomposition is a \emph{Whitney stratification}: the manifolds in the decomposition, called \emph{strata}, fit together following local topological constraints that ensure the stratified space is well-behaved. The full definition of Whitney stratification can be found in  \cite{goresky_stratified_1988}; for our purposes, the theorem below gives an explicit geometric picture of how the strata fit together for generic cut loci of surfaces.
	
	\begin{thm}[Generic shape of the cut locus \citep{mather_distance_1983, giblin_symmetry_2000, cazals_differential_2005, damon_global_2006}]
		\label{thm:generic_cut_locus}
		The cut locus $\overline{\med}$ of a generic embedding $\iota$ of a smooth surface into $\R^3$ is a two-dimensional Whitney-stratified set whose strata belong to one of five classes, distinguished by the number of contact points $q \in \overline{\med}$ has on $\Surf$ and the types of singularities of $\dist(\cdot, q)^2\rvert_\Surf$ at the contact points. Specifically, a stratum of $\overline{\med}$ can either be:
		\begin{itemize}
			\item A two-dimensional stratum (or sheet) of $A_1^2$ points, each of which having two contact points on $\Surf$ of type $A_1$;
			\item A one-dimensional stratum (or curve) of $A_1^3$ points, each of which has three contact points on $\Surf$ of type $A_1$. They are curves which lie at the intersection of three $A_1^2$ sheets; 
			\item A zero-dimensional stratum consisting of an $A_1^4$ point which has four contact points of type $A_1$. The $A_1^4$ strata are isolated and lie at the intersection of six $A_1^2$ sheets and four $A_1^3$ curves along the boundaries of the $A_1^2$ sheets;
			\item A one-dimensional stratum of $A_3$ edge points, each of which has one $A_3$ contact point on $\Surf$. Thus, they do not belong to the medial axis $\mathcal{M}$ itself. They bound $A_1^2$ sheets in $\R^3$;
			\item A zero-dimensional stratum consisting of an $A_1 A_3$ `\emph{fin creation}' point, which has two contact points on $\Surf$ of types $A_1$ and $A_3$ respectively. A fin creation point is a common end point of two one-dimensional strata of types $A_1^3$ and $A_3$ respectively, and the two $A_1^3$ and $A_3$ curves bound a common $A_1^2$ sheet. Fin creation points are isolated.
		\end{itemize}
	\end{thm}

	\subsubsection{Distance Critical Points are Generically Non-degenerate}
	
	We are now in a position to prove our main result: generic signed distance functions are topological Morse functions with finitely many critical points that are all non-degenerate (\Cref{thm:generic}). Our proof proceeds by showing that each of the conditions listed in \Cref{def:ndg_crit_dist} holds on a residual set of embeddings. If this is the case, then the set of embeddings on which all conditions hold is an intersection of residual sets, which is itself residual. As we have shown in the previous section, some of the conditions are already true as generic properties of points on the cut locus. Having validated that \ref{def:NDGC_Finite} and \ref{def:NDGC_LIG} hold for all points on generic cut loci in \Cref{thm:finite_contact}, we verify in \Cref{prop:no_A_1_A_3} that the strict ball condition \ref{def:NDGC_StrictBall} is generically satisfied by showing that only $A_1^j$ points may be critical, but not $A_1A_3$ points. 
	
	Further, in \Cref{prop:A1Morse} and \Cref{prop:A1_convex hull}, we also show that \ref{def:NDGC_Convex} and \ref{def:NDGC_MorseGerm} are generic conditions on $A_1^2$, $A_1^3$, and $A_1^4$ type critical points.  
	
	Our proofs rely on the parametric transversality theorem (\Cref{thm:param_trans}), which was also used to derive the generic description of the cut locus in \Cref{thm:generic_cut_locus}.
	We make repeated use of the following transversality argument:
	We first show that, for a given embedding $\io \in \Emb$, a property listed in \Cref{def:ndg_crit_dist} is true if some explicitly constructed map $f^\io : X \to Y$ is transversal to a submanifold $W \subset Y$. 
	To prove that $f^\io \pitchfork W$ for generic $\io$, we consider the function $f^\io$ as a member of a family of maps $F : \Emb \times X \to Y$ parameterized over the space of embeddings $\Emb$, such that $f^\io = F(\io, \cdot)$. By concluding that $F$ is a submersion (and thus transversal to any arbitrary submanifold of $Y$, in particular $W$), \Cref{thm:param_trans} then implies $f^\io \pitchfork W$ for generic $\io$.
	We note that the sufficient conditions for some properties listed in \Cref{def:ndg_crit_dist} correspond to having a void intersection $\Im f^\io \inter W = \emptyset$. Showing this property for generic $\io$ can be simply achieved by finding a dimensional insufficiency of the form $\rank Df^\io + \dim W \leq \dim X + \dim W < \dim Y$, as transversality of $F$ will then imply a void intersection between $\Im f^\io$ and $W$.
	
	We can then conclude that the set of sufficient conditions listed in \Cref{def:ndg_crit_dist} for a critical point to be non-degenerate is satisfied on a residual subset of the space of embeddings $\Emb$.
	
	In the construction of the maps $F$, we choose manifolds $X$ and $Y$ such that the other assumptions of the parametric transversality theorem (\Cref{thm:param_trans}) are satisfied. We choose finite-dimensional spaces for $X$, $Y$, and $W$ so that either $\dim X + \dim W - \dim Y = 0$ or $-1$. $E = \Emb$ and $X$ will be second countable, $F$ will be $C^{s}$ with $s \geq k - 2 > \max \{ 0, \dim X + \dim W - \dim Y \} = 0$. Our constructions of $F$ will also  be sufficiently smooth submersions for the situation at hand; this is shown in \Cref{lem:submersion} (Appendix).
	
	We set out the proof in several steps. The proof mainly deals with the pure distance $\dist(\cdot, \Surf)$, which is sufficient to derive the results for the signed distance. Note that critical points of $d$ directly correspond to those of $\dist(\cdot, \Surf)$ that are not on $\Surf$ (\Cref{def:critical_point_dist_functions}), so in what follows we consider critical points $q \notin \Surf$. 
	We make a few notational remarks before we embark on the proofs.
	
	\begin{rmk}
		To keep concise notations, we may drop the index from the notation of the objects associated to an embedding $\iota$, such as from $\Surf_\iota$ and $\W_\iota$.
	\end{rmk}
	
	\begin{rmk}[Notations for the Gauss map]
		\label{rmk:gauss_map_notation}
		We set up some shorthand notation for denoting the Gauss map. By the Jordan--Brouwer separation theorem for hypersurfaces \citep{lima_jordan-brouwer_1988,mcgrath_smooth_2016}, $\Surf_\iota$ divides $\R^3 \setminus \Surf_\iota$ into an ``inside" region, denoted by $\W_\iota = \W_\iota^-$, which is a bounded open set with possibly multiple connected components, and an ``outside" region $\R^3 \setminus \overline{\W_\iota} = \W_\iota^+$. For $\io \in \Emb$, let $\nbf = \nbf_\io : \Surf \to \mathbb{S}^2$ denote the Gauss map, i.e., the globally-defined unit vector field normal to the surface that points inwards (in the direction of $\W^-$ if $\Surf = \bord \W^-$). Note that $\nbf_\io$ is orthogonal to the image of $D_{m_i} \io : T_{m_i} M \to T_{\iota(m_i)}\R^3$, which is the tangent space of the surface $T_{\iota(m_i)} \Surf = \Im D_m\io$. Since the embedding $\iota$ is implicit in the notation $\nbf$, we drop the index in $\nbf_\io$.
		We abuse notation and also use $\nbf$ to denote the Gauss map when we compose $\nbf:\Surf \to \mathbb{S}^2$ with the projection map of the tubular neighborhood $\mathrm{Tub}(\Surf)$ onto $\Surf$, or the inclusion of $\Sbb^2$ into $\R^3$. Hence $\nbf$ is a shorthand for four different maps, whose domain can either be $\Surf$ or $\mathrm{Tub}(\Surf)$, and whose codomain is either $\mathbb{S}^2$ or $\R^3$. We will specify the source and target spaces explicitly to avoid ambiguity when we reference $\nbf$ in the subsequent proofs.
	\end{rmk}

	\begin{prop} \label{prop:A1Morse} Consider a critical point $q$ of the distance function of type $A_1^j$. Then it is a Min-type critical point  (\Cref{def:q_mintype_crit}) of $\dist(\cdot, \Surf)$.  Furthermore, there is a residual subset of embeddings $\io \in \Emb$ such that $A_1^j$ critical points $q$ of the distance function are non-degenerate critical points of $\dist(\cdot, \Surf)_{\rvert G(q)}$, the smooth function that is the restriction of $\dist(\cdot, \Surf)$ to the germ $G(q)$ of the submanifold of $A_1^j$ points (as given in \Cref{lem:A1_manifolds}). Thus, $A_1^j$ type distance critical points satisfy \ref{def:NDGC_MorseGerm}.
	\end{prop}
	\begin{proof}
		Recall for $q$ an $A_1^j$ type, we have an efficient-LIG representation $\dist(\cdot, \Surf) = \min\{\alpha_1, \ldots, \alpha_j\}$, and the germ of the stratum $G(q)$ is a $(4-j)$ dimensional submanifold (\Cref{lem:A1_manifolds}). Recall \Cref{lem:equivalent_definitions_critical_point} that $q$ is a distance critical point iff $q$ is in the convex hull of its contact points: i.e., there is a set of non-negative coefficients $t_i$ that sum to one, such that $q = \sum_{i=1}^j t_i p_i$. Since $\nabla \alpha_i \rvert_{q} = (q-p_i)/r$ where $r = \|p_1 - q \| = \cdots = \|p_j - q \|$, we can rewrite this condition as $\sum_{i=1}^j t_i \nabla \alpha_i \rvert_{q} = 0$. In other words, the origin is contained in the convex hull of $\{ \nabla \alpha_i\rvert_q\}_{i = 1,\ldots,j}$. Thus, $q$ is a Min-type critical point of $\dist(\cdot, \Surf)$ (\Cref{def:q_mintype_crit}); furthermore, $q$ is a critical point of  $\dist(\cdot, \Surf)_{\lvert G_f(q)}$ (\Cref{lem:q_mintype_crit}). 
		
		We now show that the Riemannian Hessian of the pure distance restricted to the submanifold $G(q)$ is non-degenerate at $q$. 
		Without loss of generality, we assume that $q \in \W_\io^-$ is ``inside".
		Note that the explicit expression of the Riemannian Hessian of distance functions restricted on $G(q)$ is given in  \Cref{lem:shape_operators_levels} and \Cref{lem:riem_hess} of the Appendix, derived using the shape operator of offset surfaces. We proceed by a case by case analysis for $j = 2,3,4$.

		\textbf{$A_1^2$ Critical Points}. We start by establishing the proof for a point $q$ with two $A_1$ contact points. In what follows, $M^{(i)}$ refers to the set of $i$-tuples of distinct points in $M$.
		
		Let $X = M^{(2)} \times \R_{>0}$ and $Y = (\mathbb{S}^2)^2 \times (\R^3)^2$. Consider the submanifold $W = \{y = (y_1,\ldots, y_4)\in Y ~|~ y_1 + y_2 = 0 \text{ and } y_3 - y_4 = 0 \}$ of dimension $5$ in $Y$.
		Define 
		\begin{equation}
			F:
			\begin{array}{ccc}
				\Emb \times X & \to & Y \\
				(\io, m_1,m_2,r) & \mapsto & (n_1, n_2, p_1 + r \, n_1, p_2 + r \, n_2 )
			\end{array}
		\end{equation}
		where $p_i = \io(m_i)$ and $n_i = \nbf(p_i)$. In the expression $p_i + r \, n_i$, $n_i$ is viewed in $\R^3$.
		
		Let us explain the geometric meaning of $f^\io(x) \in W$, where $f^\io = F(\io,\cdot)$:
		Given some embedding $\io$, if $q$ is of type $A_1^2$ with contact points $\{p_1,p_2\}$, then setting $m_i = \io^{-1}(p_i)$, $r = \dist(q, \Surf)$ and $x = (m_1,m_2,r) = x(q)$, we have $n_1 + n_2 = 0$ and $p_1 + r \, n_1 = p_2 + r \, n_2 = q$. This implies that $f^\io(x(q)) \in W$. However, note that the relationship $f^\io(x) \in W$ also describes points $x \in X$ that do not necessarily correspond to $A_1^2$ critical points, as it does not specify, for instance, whether the $p_i$ should be the closest points on the surface. 
		
		We investigate what it means for $f^\io$ to be transversal to $W$.
		For a fixed embedding $\io$, we compute the differential of $F(\iota, \cdot)$ with respect to $x = (m_1,m_2,r)$, obtaining
		
		\begin{equation*}
			D_x F(\io, \cdot) =
			\begin{blockarray}{*{3}{c} l}
				\begin{block}{cccl}
					T_{m_1} M & T_{m_2} M & ~~\R~~ & \\
				\end{block}
				\begin{block}{[ccc] l}
					d_{p_1} \nbf \circ d_{m_1} \io & 0 & 0  \topstrut  & T_{n_1} \mathbb{S}^2 \\
					0 & d_{p_2} \nbf \circ d_{m_2} \io & 0 & T_{n_2} \mathbb{S}^2   \\
					(\Id + r \, d_{p_1} \nbf) \circ d_{m_1} \io & 0 & n_1 & \R^3 \\
					0 & (\Id + r \, d_{p_2} \nbf) \circ d_{m_2} \io & n_2 \botstrut & \R^3  \\
				\end{block}
			\end{blockarray}.
		\end{equation*}

		Then $f^\io \pitchfork W$ if and only if for $x \in X$ such that $f^\io(x) \in W$, the following map is of rank $\mathrm{codim} \, W = 5$:
		
		\begin{equation*}
			\begin{blockarray}{*{3}{c} l}
				\begin{block}{cccl}
					T_{m_1} M & T_{m_2} M & ~~\R~~ & \\
				\end{block}
				\begin{block}{[ccc] l}
					d_{p_1} \nbf \circ d_{m_1} \io & d_{p_2} \nbf \circ d_{m_2} \io & 0 \topstrut  & \R^3 \\
					(\Id + r \, d_{p_1} \nbf) \circ d_{m_1} \io & - (\Id + r \, d_{p_2} \nbf) \circ d_{m_2} \io & n_1 - n_2 \botstrut & \R^3 \\
				\end{block}
			\end{blockarray}.
		\end{equation*}
		Up to the isomorphism 
		\begin{equation*}
			\begin{blockarray}{*{3}{c} l}
				\begin{block}{cccl}
					T_{m_1} M & T_{m_2} M & ~~\R~~ & \\
				\end{block}
				\begin{block}{[ccc] l}
					d_{m_1} \io & 0 & 0 \topstrut  & T_{p_1} \Surf \\
					0 & d_{m_2} \io & 0 & T_{p_2} \Surf \\
					0 & 0 & 1 \botstrut & \R \\
				\end{block}
			\end{blockarray},
		\end{equation*} 
		we can consider this map instead
		\begin{equation*}
			\begin{blockarray}{*{3}{c} l}
				\begin{block}{cccl}
					T_{p_1} \Surf & T_{p_2} \Surf & ~~\R~~ & \\
				\end{block}
				\begin{block}{[ccc] l}
					d_{p_1} \nbf  & d_{p_2} \nbf & 0 \topstrut  & \R^3 \\
					(\Id + r \, d_{p_1} \nbf)  & - (\Id + r \, d_{p_2} \nbf)  & n_1 - n_2 \botstrut & \R^3 \\
				\end{block}
			\end{blockarray}.
		\end{equation*}
		For $x$ such that $f^\io(x) \in W$, we have in particular that $n_2 = - n_1$ so that $n_1 - n_2 = 2 n_1$, and the tangent spaces $T_{p_1} \Surf \simeq T_{p_2} \Surf$ are parallel so they may be identified as a single space denoted $T_{p_i} \Surf$. Still abusing notation, we may consider $d_{p_1} \nbf$ and $d_{p_2} \nbf$ as maps $T_{p_i} \Surf \to T_{p_i} \Surf$.
		
		The previous rank condition is then equivalent to the following ones: 
		\begin{align*}
			f^\io \pitchfork W & \Leftrightarrow &\forall\ x \in X \text{ s.t. } f^\io(x) \in W, \quad & \begin{blockarray}{*{3}{c} l}
				\begin{block}{cccl}
					T_{p_i} \Surf & T_{p_i} \Surf & ~~\R~~ & \\
				\end{block}
				\begin{block}{[ccc] l}
					d_{p_1} \nbf & d_{p_2} \nbf & 0 \topstrut  & T_{p_i} \Surf \\
					\Id + r \, d_{p_1} \nbf & - (\Id + r \, d_{p_2} \nbf) & n_1 - n_2 \botstrut & \R^3 \\
				\end{block}
			\end{blockarray}
			\text{ is of rank } 5 \\
			& \Leftrightarrow & \forall\ x \in X \text{ s.t. } f^\io(x) \in W, \quad & \begin{blockarray}{*{3}{c} l}
				\begin{block}{cccl}
					T_{p_i} \Surf & T_{p_i} \Surf & ~~\R~~ & \\
				\end{block}
				\begin{block}{[ccc] l}
					d_{p_1} \nbf & d_{p_2} \nbf & 0 \topstrut  & T_{p_i} \Surf \\
					\Id + r \, d_{p_1} \nbf & - (\Id + r \, d_{p_2} \nbf) & 0 & T_{p_i} \Surf  \\
					0 & 0 & 2 \botstrut & \R \, n_1 \\
				\end{block}
			\end{blockarray}
			\text{ is of rank } 5 \\
			& \Leftrightarrow & \forall\ x \in X \text{ s.t. } f^\io(x) \in W, \quad & \begin{blockarray}{cc l}
				\begin{block}{ccl}
					T_{p_i} \Surf & T_{p_i} \Surf & \\
				\end{block}
				\begin{block}{[cc] l}
					d_{p_1} \nbf & d_{p_2} \nbf \topstrut  & T_{p_i} \Surf \\
					\Id + r \, d_{p_1} \nbf & - (\Id + r \, d_{p_2} \nbf) & T_{p_i} \Surf \\
				\end{block}
			\end{blockarray}
			\text{ is of rank } 4.
		\end{align*}
		
		The last condition involves a block matrix whose determinant can be expressed simply, thanks to a formula due to \cite{silvester_determinants_2000}:
		\[AC = CA \quad \Rightarrow \quad \det 
		\begin{pmatrix}
			A & B \\
			C & D
		\end{pmatrix} = \det(AD - CB).\]
		Since $d_{p_1} \nbf$ commutes with $(\Id + r \, d_{p_1} \nbf)$, we have
		\[f^\io \pitchfork W \Leftrightarrow \forall\ x \in X \text{ s.t. } f^\io(x) \in W, \quad \det \big(d_{p_1} \nbf \circ (\Id + r \, d_{p_2} \nbf) + (\Id + r \, d_{p_1} \nbf) \circ d_{p_2} \nbf \big) \neq 0.\]
		By composing with the isomorphisms $(\Id + r \, d_{p_1} \nbf)^{-1}$ and $(\Id + r \, d_{p_2} \nbf)^{-1}$ on the left and the right sides, this means
		\[f^\io \pitchfork W 
		\Leftrightarrow \forall\ x \in X \text{ s.t. } f^\io(x) \in W, \quad \det \big((\Id + r \, d_{p_1} \nbf)^{-1} \circ d_{p_1} \nbf +  d_{p_2} \nbf \circ (\Id + r \, d_{p_2} \nbf)^{-1} \big) \neq 0.\]
		Finally, $d_{p_1} \nbf$ also commutes with $(\Id + r \, d_{p_1} \nbf)^{-1}$, and so
		\[f^\io \pitchfork W \Leftrightarrow \forall\ x \in X \text{ s.t. } f^\io(x) \in W, \quad \det \big( d_{p_1} \nbf \circ (\Id + r \, d_{p_1} \nbf)^{-1} + d_{p_2} \nbf \circ (\Id + r \, d_{p_2} \nbf)^{-1} \big) \neq 0.\]

		Notice that the last inequality gives non-degeneracy of the Riemannian Hessian from \Cref{lem:riem_hess}. So, for generic embeddings $\io$, we obtain
		\[f^\io \pitchfork W \Rightarrow \Hess^\mathrm{Riem} g(q) \text{ is non-degenerate for any critical point } q \text{ of type } A_1^2. \]
		
		It remains to verify that $f^\io \pitchfork W$ for generic embeddings $\io \in \Emb$. But this is true by the transversality argument stated above, because $F$ is a submersion (\Cref{lem:submersion}).

		\textbf{$A_1^3$ Critical Points}. Likewise, for a point $q$ of type $A_1^3$, we define some well-chosen spaces and maps.
		
		Let $X = M^{(3)} \times \R_{>0}$ and $Y = (\mathbb{S}^2)^3 \times (\R^3)^3$. For two vectors $u,v \in \R^3$, we denote the usual scalar product by $u \cdot v$ and the vector product by $u \times v$. Consider the submanifold 
		\[ W = \{y = (y_1,\ldots, y_6)\in Y ~|~ y_3 \cdot (y_1 \times y_2) = 0 \text{ and } y_4 = y_5 = y_6 \} \]
		which is of dimension in $8$ and codimension $7$ in $Y$.
		Note that the triple product is equal to $y_3 \cdot (y_1 \times y_2) = y_1 \cdot (y_2 \times y_3) = y_2 \cdot (y_3 \times y_1)$, and that it vanishes if and only if one of the three vectors belongs the subspace spanned by the two others (which is of dimension $1$ or $2$). There is a map $\ell$ such that $W = \{\ell = 0\} \inter Y \subset E$, with $E = (\R^3)^6$ and
		$\ell(e_1,\ldots, e_6) = (e_3 \cdot (e_1 \times e_2), e_4 - e_5, e_4 - e_6)$. The tangent space to $\{\ell = 0\}$ at a point in $E$ is given by 
		\begin{align*}
			\Ker \grad \ell & = \{ (\eps_1,\ldots, \eps_6) \in E ~|~ (e_2 \times e_3) \cdot \eps_1 + (e_3 \times e_1) \cdot \eps_2 + (e_1 \times e_2) \cdot \eps_3 = 0 \} \\ 
			& \inter \{  (\eps_1,\ldots, \eps_6) \in E ~|~ \eps_4 - \eps_5 = 0 \text{ and } \eps_4 - \eps_6 = 0\}.
		\end{align*}
		
		Now, we introduce the following map:
		\begin{equation}
			F:
			\begin{array}{ccc}
				Emb \times X & \to & Y \\
				(\io, m_1,m_2,m_3,r) & \mapsto & (n_1, n_2, n_3, p_1 + r n_1, p_2 + r n_2 , p_3 + r n_3)
			\end{array}.
		\end{equation}
		
		Again, we study what it means for $f^\io = F(\io,\cdot)$ to be transversal to $W$.
		
		Given an embedding $\io$, the differential at $x = (m_1,m_2,m_3,r)$ is equal to
		
		\begin{equation*}
			D_x F(\io, x) =
			\begin{blockarray}{cccc l}
				\begin{block}{cccc l}
					T_{m_1} M & T_{m_2} M & T_{m_3} M & ~~\R~~ & \\
				\end{block}
				\begin{block}{[cccc] l}
					d_{p_1} \nbf \circ d_{m_1} \io & 0 & 0 & 0 \topstrut  & T_{n_1} \mathbb{S}^2 \\
					0 & d_{p_2} \nbf \circ d_{m_2} \io & 0 & 0 & T_{n_2} \mathbb{S}^2   \\
					0 & 0 & d_{p_3} \nbf \circ d_{m_3} \io & 0 & T_{n_3} \mathbb{S}^2   \\
					(\Id + r \, d_{p_1} \nbf) \circ d_{m_1} \io & 0 & 0 & n_1 & \R^3 \\
					0 & (\Id + r \, d_{p_2} \nbf) \circ d_{m_2} \io & 0 & n_2 & \R^3  \\
					0 & 0 & (\Id + r \, d_{p_3} \nbf) \circ d_{m_3} \io  & n_3 \botstrut & \R^3  \\
				\end{block}
			\end{blockarray}.
		\end{equation*}
		
		Then,
		as before, we study its projection on $\Ker(\grad \ell)^\perp = \Im((\grad \ell)^T)$ in $E$. We obtain that $f^\io \pitchfork W$ if and only $\forall\ x \in X \text{ s.t. } f^\io(x) \in W$,
		\begin{equation*}
			\begin{blockarray}{ccccl}
				\begin{block}{ccccl}
					T_{p_1} \Surf & T_{p_2} \Surf & T_{p_3} \Surf & ~~\R~~ & \\
				\end{block}
				\begin{block}{[cccc] l}
					(n_2 \times n_3) \cdot d_{p_1} \nbf & (n_3 \times n_1) \cdot d_{p_2} \nbf & (n_1 \times n_2) \cdot d_{p_3} \nbf & 0 \topstrut  & \R \\
					\Id + r \, d_{p_1} \nbf & - (\Id + r \, d_{p_2} \nbf) & 0 & n_1 - n_2 & \R^3 \\
					\Id + r \, d_{p_1} \nbf & 0 & - (\Id + r \, d_{p_3} \nbf) & n_1 - n_3 \botstrut & \R^3 \\
				\end{block}
			\end{blockarray}
			\text{ is of rank } 7.
		\end{equation*}
		In particular, if $q$ is of type $A_1^3$ with contact points $\{p_1,p_2,p_3\}$, we set $m_i = \io^{-1}(p_i)$, $r = \dist(q, \Surf)$ and $x = (m_1,m_2,m_3,r) = x(q)$. By the geometric relations, we have $n_3 \cdot (n_1 \times n_2) = 0$ (because they span the same two-dimensional subspace) and $p_1 + r \, n_1 = p_2 + r \, n_2 = p_3 + r \, n_3 = q$.
		This implies that $f^\io(x(q)) \in W$. Then, because the contact points are of $A_1$ type and satisfy the strict ball condition \eqref{eq:strict_ball_condition}, the following map is an isomorphism:
		\begin{equation*}
			\begin{blockarray}{ccccl}
				\begin{block}{ccccl}
					T_{p_1} \Surf & T_{p_2} \Surf & T_{p_3} \Surf & ~~\R~~ & \\
				\end{block}
				\begin{block}{[cccc] l}
					\Id + r \, d_{p_1} \nbf & 0 & 0 & 0 \topstrut & T_{p_1} \Surf \\
					0 & \Id + r \, d_{p_2} \nbf & 0 & 0 & T_{p_2} \Surf \\
					0 & 0 & \Id + r \, d_{p_3} \nbf & 0 & T_{p_3} \Surf \\
					0 & 0 & 0 & 1 \botstrut & \R \\
				\end{block}
			\end{blockarray}.
		\end{equation*}
		By applying the inverse of this isomorphism to the previous map on the right side, we obtain that, if $f^\io \pitchfork W$, then at the point $x(q)$,
		\begin{equation*}
			\begin{blockarray}{ccccl}
				\begin{block}{ccccl}
					T_{p_1} \Surf & T_{p_2} \Surf & T_{p_3} \Surf & ~~\R~~ & \\
				\end{block}
				\begin{block}{[cccc] l}
					t_{2,3} \cdot u_1 & t_{3,1} \cdot u_2 & t_{1,2} \cdot u_3 & 0 \topstrut & \R \\
					\Id & - \Id & 0 & n_1 - n_2 & \R^3 \\
					\Id & 0 & - \Id & n_1 - n_3 \botstrut & \R^3 \\
				\end{block}
			\end{blockarray}
			\text{ is of rank } 7, \text{ i.e., is full-rank},
		\end{equation*}
		where $t_{i,j} = n_i \times n_j$ and $u_i =  d_{p_i} \nbf \circ (\Id + r \, d_{p_i} \nbf)^{-1}$.
		
		Now, we use the block matrix formula
		\[D \text{ invertible } \quad \Rightarrow \quad \det 
		\begin{pmatrix}
			A & B \\
			C & D
		\end{pmatrix} = \det(A - B D^{-1} C) \det(D)\]
		applied to 
		\[A = 
		\begin{blockarray}{ccl}
			\begin{block}{ccl}
				T_{p_1} \Surf & T_{p_2} \Surf & \\
			\end{block}
			\begin{block}{[cc] l}
				t_{2,3} \cdot u_1 & t_{3,1} \cdot u_2 \topstrut & \R \\
				\Id & - \Id \botstrut & \R^3 \\
			\end{block}
		\end{blockarray}
		\qquad
		B = 
		\begin{blockarray}{ccl}
			\begin{block}{ccl}
				T_{p_3} \Surf & \R & \\
			\end{block}
			\begin{block}{[cc] l}
				t_{1,2} \cdot u_3 & 0 \topstrut & \R \\
				0 & n_1 - n_2 \botstrut & \R^3 \\
			\end{block}
		\end{blockarray}
		\]
		
		\[
		C = 
		\begin{blockarray}{ccl}
			\begin{block}{ccl}
				T_{p_1} \Surf & T_{p_2} \Surf & \\
			\end{block}
			\begin{block}{[cc] l}
				\Id & 0 \topstrut \botstrut & \R^3 \\
			\end{block}
		\end{blockarray}
		\qquad
		D = 
		\begin{blockarray}{ccl}
			\begin{block}{ccl}
				T_{p_3} \Surf & \R & \\
			\end{block}
			\begin{block}{[cc] l}
				-\Id & n_1 - n_3 \topstrut \botstrut & \R^3 \\
			\end{block}
		\end{blockarray}.
		\]
		$D$ is indeed invertible because $n_1 - n_3$ has a non-zero component along $n_3$ because $n_1 \neq n_3$ (otherwise $p_1 = p_3$). 
		
		We consider generic embeddings in the sense of \Cref{thm:generic_cut_locus}, which allows us to work with the generic cut locus.
		Let us introduce $t$, a unit vector directing $T_q G(q)$ and known to be perpendicular to the subspace spanned by $\{n_1,n_2,n_3\}$.
		In the basis $(t,t \times n_3, n_3)$, we have then
		\[
		D = 
		\begin{blockarray}{cccl}
			\begin{block}{cccl}
				t & t \times n_3 & \R & \\
			\end{block}
			\begin{block}{[ccc] l}
				-1 & 0 & 0 \topstrut & t \\
				0 & -1 & t \cdot t_{3,1} & t \times n_3 \\
				0 & 0 & n_1 \cdot n_3 - 1 \botstrut & n_3 \\
			\end{block}
		\end{blockarray}
		\]
		so that
		\[
		D^{-1} = 
		\begin{blockarray}{cccl}
			\begin{block}{cccl}
				t & t \times n_3 & n_3 & \\
			\end{block}
			\begin{block}{[ccc] l}
				-1 & 0 & 0 \topstrut & t \\
				0 & -1 & \frac{t \cdot t_{3,1}}{n_1 \cdot n_3 - 1} & t \times n_3 \\
				0 & 0 & \frac{1}{n_1 \cdot n_3 - 1} \botstrut & \R \\
			\end{block}
		\end{blockarray}.
		\]
		After some computations with compatible choices of bases, we can explicit the expressions of $B$, $C$, and $A$ as well. For instance, the product space $T_{p_1} \Surf \times T_{p_2} \Surf$ is spanned by the basis loosely denoted by $((t, t\times n_1), (t, t \times n_2))$. We obtain that
		\[A - B D^{-1} C =
		\begin{blockarray}{ccccl}
			\begin{block}{ccccl}
				t & t \times n_1 & t & t \times n_2 & \\
			\end{block}
			\begin{block}{[cccc] l}
				(a) & (b) & (c) & (d) \topstrut & \R \\
				1 & 0 & -1 & 0 & t \\
				0 & (e) & 0 & -1 & t \times n_2 \\
				0 & (f) & 0 & 0 \botstrut & n_2 \\
			\end{block}
		\end{blockarray}
		\]
		where
		\begin{align*}
			(a) &= t_{2,3} \cdot u_1(t) - t_{2,1} \cdot u_3(t) \\
			(b) &= t_{2,3} \cdot u_1(t \times n_1) - t_{2,1} \cdot u_3(t \times n_3) \left( n_1 \cdot n_3 + \frac{(t \cdot t_{3,1})^2}{n_1 \cdot n_3 - 1} \right) \\
			(c) &= t_{3,1} \cdot u_2(t) \\
			(d) &= t_{3,1} \cdot u_2(t \times n_2) \\
			(e) &= n_1 \cdot n_2 - \frac{(t \cdot t_{2,1})(t \cdot t_{1,3})}{n_1 \cdot n_3 - 1} \\
			(f) &= t \cdot t_{1,2} - \frac{(t \cdot t_{1,3})(n_1 \cdot n_2 - 1)}{n_1 \cdot n_3 - 1}.
		\end{align*}
		Therefore, for generic embeddings, being transversal to $W$ means 
		\begin{align*}
			f^\io \pitchfork W \text{ at } x(q) \quad & \Rightarrow \quad \det(A - B D^{-1} C) \neq 0 \\
			&\Rightarrow \quad (f) ~
			\begin{vmatrix}
				(a) & (c) & (d) \\
				1 & -1 & 0 \\
				0 & 0 & -1
			\end{vmatrix} \neq 0 \\
			&\Rightarrow \quad (f) \, \left( \, (a) + (c) \, \right) \neq 0 \\
			&\Rightarrow \quad \frac{t \cdot t_{1,2}}{n_1 \cdot n_2 - 1} \neq \frac{t \cdot t_{1,3}}{n_1 \cdot n_3 - 1} \text{ and } 
			t_{2,3} \cdot u_1(t) + t_{3,1} \cdot u_2(t) + t_{1,2} \cdot u_3(t) \neq 0
		\end{align*}
		Note that, if $\widehat{(n_1, n_2)}$ denotes the algebraic angle formed by $n_1,n_2$ in the plane orthogonal to $G(q)$ and oriented by $t$, we have $t \cdot t_{1,2} = \sin \widehat{(n_1, n_2)}$ and $n_1 \cdot n_2 = \cos \widehat{(n_1, n_2)}$. Hence the first condition involves some function $\frac{\sin \theta}{\cos \theta - 1}$ and just means that the angles $\widehat{(n_1, n_2)}$ and $\widehat{(n_1,n_3)}$ should be different, i.e., that $n_2 \neq n_3$, which is true. In the second condition, we recognize the result of \Cref{lem:riem_hess}, so that, for generic embeddings,
		\[f^\io \pitchfork W \Rightarrow \Hess^\mathrm{Riem} g(q) \text{ is non-degenerate for any critical point } q \text{ of type } A_1^3. \]
		Here again, $F$ is a submersion by \Cref{lem:submersion}, thus $f^\io \pitchfork W$ for generic embeddings $\io \in \Emb$, by the transversality argument stated above.
		
		\textbf{$A_1^4$ Critical Points}. There is nothing to check here, because the stratum of medial axis $G(q)$ reduces to a single point $\{q\}$ itself.
		
	\end{proof}
	
	\begin{prop}	\label{prop:A1_convex hull}
		There is a residual subset of embeddings $\iota \in \Emb$ for which a distance critical point $q$ of type $A_1^2$, $A_1^3$, or $A_1^4$ lies strictly in the interior of the convex hull of $\Gamma(q)$ (i.e., it satisfies \ref{def:NDGC_Convex}).
	\end{prop}
	
	\begin{proof}
		A critical point $q$ of type $A_1^2$ always lies midway between the two points $p_1$ and $p_2$ in $\Gamma(q)$, so that the origin is always strictly inside of the convex hull spanned by $\{\grad \alpha_1(q), \grad \alpha_2(q)\} = \{n_1, n_2\}$.
		
		For a critical point $q$ of type $A_1^3$, we re-introduce the same spaces and maps as above, but with a slight modification to $W$.
		Let $X = M^{(3)} \times \R_{>0}$, $Y = (\mathbb{S}^2)^3 \times (\R^3)^3$, $W = \{y = (y_1,\ldots, y_6)\in Y ~|~ y_3 \cdot (y_1 \times y_2) = 0 \text{ and } y_4 = y_5 = y_6 \}$ and
		\begin{equation*}
			F:
			\begin{array}{ccc}
				Emb \times X & \to & Y \\
				(\io, m_1,m_2,m_3,r) & \mapsto & (n_1, n_2, n_3, p_1 + r n_1, p_2 + r n_2 , p_3 + r n_3)
			\end{array},
		\end{equation*} 
		as before. Consider the subspace of $W$ defined by
		\[W_{1,2} = \{ w \in W ~|~ w_1 + w_2 = 0 \},\]
		which is of dimension in $7$ and codimension $8$ in $Y$. If $f^\io(x(q))$ belongs to $W_{1,2}$, this corresponds geometrically to a right-angled triangle formed by $n_1,n_2,n_3$, where the right angle sits on $n_3$ while $n_1 + n_2 = 0$. 
		
		Actually, we already know that $F$ is a submersion (\Cref{lem:submersion}), hence for generic $\io$, $f^\io \pitchfork W_{1,2}$. But $\rank D_x f^\io \leq \dim X = 7$, so that $\rank D_x f^\io + \dim W \leq 7 + 7 < 15 = \dim Y$,
		which means that generically $\Im f^\io$ does not intersect $W_{1,2}$. Likewise, generically, $\Im f^\io$ does not intersect any of $W_{1,2}$, $W_{1,3}$ or $W_{2,3}$. In geometric terms, we get that for generic embeddings, there is no $A_1^3$ critical point lying on the boundary of the convex hull.
		
		Similarly, for the $A_1^4$ case, by introducing the relevant map $F$ and spaces $W_{i,j,k}$, we obtain a generic void intersection between images and spaces and conclude that for generic embeddings, there is no $A_1^4$ critical point lying on the boundary of the convex hull, i.e., belonging to a plane spanned by only $3$ of the $4$ contact points.
	\end{proof}

	Now, we move on to prove that for generic embeddings, critical points satisfy the strict ball condition \eqref{eq:strict_ball_condition} at any of their contact points. In particular, this excludes $A_3$ contact singularities (see \cite{cazals_differential_2005}). We must therefore show that critical points only admit $A_1$ contact points and no $A_3$ contact points, or equivalently, that there are no critical points of type $A_1 A_3$, but only of types $A_1^j$, $j = 2,3,4$.
	
	\begin{prop}
		\label{prop:no_A_1_A_3}
		There is a residual subset of embeddings $\iota \in \Emb$ that not only satisfies the properties set out in \Cref{thm:generic_cut_locus}, but also the following: the critical points of $d = d_\iota$ are of types $A_1^2$, $A_1^3$, or $A_1^4$ only, \emph{but not} $A_1 A_3$ (hence distance critical points of $\Surf = \iota(M)$ satisfy the strict ball condition \ref{def:NDGC_StrictBall}). 
	\end{prop}
	
	\begin{proof}
		We make a slight modification to the spaces and maps introduced previously in the proof of \Cref{prop:A1Morse} for $A_1^2$ critical points. Let $X = M^{(2)} \times \R_{>0}$ and $Y = E \times \mathbb{S}^2 \times \R^3 \times \R^3$, where we define a fiber bundle $(E, \mathbb{S}^2, \pi, \mathrm{Sym}((\R y_1)^\perp))$ whose total space is $E$, base space is $\mathbb{S}^2$, fiber is $\mathrm{Sym}((\R y_1)^\perp)$, and $\pi : E \to \mathbb{S}^2$. Here, for a base point $y_1 \in \mathbb{S}^2$, the notation $\mathrm{Sym}((\R y_1)^\perp)$ refers to the space of symmetrical endomorphisms of the two-dimensional subspace orthogonal to $y_1$. 
		Define the map
		\begin{equation}
			F:
			\begin{array}{ccc}
				\Emb \times X & \to & Y \\
				(\io, m_1,m_2,r) & \mapsto & (n_1, \Id + r \, d_{p_1}\nbf, n_2, p_1 + r \, n_1, p_2 + r \, n_2 )
			\end{array}
		\end{equation}
		where $p_i = \io(m_i)$ and $n_i = \nbf(p_i)$. Consider the submanifold $W = \{y = ( (y_1, \varphi), y_2, \ldots, y_4)\in Y ~|~ y_1 + y_2 = 0 \text{ and } y_3 - y_4 = 0 \text{ and } \rank \varphi = 1\}$ which is of dimension $7$ in $Y$.
		
		Geometrically, if $q$ is a critical point with one $A_3$ contact point $p_1$ and one $A_1$ contact point $p_2$, then the map $\Id + r \, d_{p_1}\nbf \in \mathrm{Sym}((\R n_1)^\perp)$ is of rank $1$ because the largest principal curvature does not satisfy the strict ball condition, i.e., we have $\kap_{\text{max}}(p_1) = \frac{1}{\dist(q,\Surf)}$, contrarily to the smallest principal curvature (we may suppose with loss of generality that $q \in \W^-$). Thus, for such $q$, we get $f^\io(x(q)) \in W$.
		
		But, with the same arguments as before, generically $f^\io \pitchfork W$, but since $\rank D_x f^\io + \dim W \leq \dim X + \dim W = 5 + 7 < 13 = \dim Y$, this means that the intersection between $\Im f^\io$ and $W$ is generically void, which shows what we need.
	\end{proof}
	
	Assembling \Cref{thm:finite_contact}, \Cref{thm:generic}, \Cref{prop:A1Morse}, \Cref{prop:A1_convex hull}, \Cref{prop:no_A_1_A_3}, and \Cref{cor:mintype_ndg_finite}, we arrive at the main result of this section. 
	
	\begin{thm}[Signed distances are generically Morse with finitely many critical points]
		\label{thm:generic}
		There is a residual subset of embeddings $\io \in \Emb$, such that the signed distance function $d$ to $\Surf = \io(M)$ is a topological Morse function with finitely many critical points. In particular, for such embeddings, the critical points of $d$ are non-degenerate Min-type critical points, satisfying the conditions \ref{def:NDGC_Finite}-\ref{def:NDGC_MorseGerm} in \Cref{def:ndg_crit_dist}.
	\end{thm}
	
	\begin{proof}
		To show that $d$ is a topological Morse function, we prove that the critical points of $d$ are topologically non-degenerate critical points; in particular, they are non-degenerate Min-type critical points. We do so by verifying conditions \ref{def:NDGC_Finite}-\ref{def:NDGC_MorseGerm} in \Cref{def:ndg_crit_dist} are individually true on a residual subset of embeddings; thus there is a residual subset of embeddings where all the conditions \ref{def:NDGC_Finite}-\ref{def:NDGC_MorseGerm} are true. The finiteness of contact points \ref{def:NDGC_Finite} and LIG \Cref{def:NDGC_LIG} conditions shown to be generic in \Cref{thm:finite_contact}. \Cref{prop:no_A_1_A_3} shows that any critical point on a generic set of embeddings are $A_1^j$ types and satisfy the strict ball condition \ref{def:NDGC_StrictBall}; and \Cref{prop:A1_convex hull} shows that the convex hull condition \ref{def:NDGC_Convex} is also generic. Finally, as the critical points are non-degenerate, there are only finitely many such critical points \Cref{cor:mintype_ndg_finite}. 
	\end{proof}
	
	\section{Signed Distance Persistent Homology (SDPH)}
	\label{sec:SDPH}
	
	Given the generalized Morse theorems for generic shapes established in the previous Section \ref{sec:morse_theory_signdist}, we may now rigorously define SDPH. Our goal is to study the shape of an object, which requires for it to be rigorously and interpretably quantified. The distance and signed distance fields associated with the object, as well as persistent homology, are tools well-suited to this task. SDPH computes persistent homology with respect to $d$ (previously introduced in \Cref{sec:PH_and_Morse_theory}).
	
	Topology, in general, can be thought of as a pure mathematical study of shape and topological approaches have been used to study porous materials and vascular networks. For instance, porous materials were characterized by Vietoris--Rips persistent homology \citep{lee_quantifying_2017, obayashi_persistent_2022}, by cubical filtrations derived from the image itself \citep{robins_theory_2011}, and by a max-flow study in network models \citep{armstrong_correspondence_2021}. Discrete Morse theory has been used in conjunction with deep learning to accurately segment thin vascular structures \citep{hu_topology-preserving_2019,hu_learning_2022}.
	Recently, \cite{stolz_multiscale_2022} showcased how persistent homology, computed on a radial filtration or an alpha-complex filtration, is able to characterize tumor vascular networks.
	
	The signed distance field is particularly useful since it naturally delineates the boundary of a shape, which happens when the value of the signed distance function is zero.  In the context of persistent homology, filtering the space by the sublevel sets of the signed distance rather than the unsigned one allows features located inside the shape to be distinguished from those outside.  The SDPH methodology appeared (under various names) in the discrete cubical setting first and was used to study porous materials with complex morphologies \citep{delgado-friedrichs_morse_2014, delgado-friedrichs_skeletonization_2015} and the interactions between fluid flow and microstructure \citep{herring_topological_2019, moon_statistical_2019}.
	
	In the works by \cite{delgado-friedrichs_morse_2014, delgado-friedrichs_skeletonization_2015,herring_topological_2019, moon_statistical_2019}, the SDPH method was described in a purely discrete setting, which is suitable for carrying numerical computations. \cite{herring_topological_2019, moon_statistical_2019} proposed an empirical interpretation of the SDPH diagrams, but without mathematical justification. 
	To mathematically study the SDPH diagrams arising in the discrete settings, several types of discrete Morse theories may be useful \citep{banchoff_critical_1967, forman_morse_1998} (see \citep{lewiner_critical_2013, bloch_polyhedral_2013, bauer_persistence_2011, saucan_discrete_2020}).
	Here, we are interested in the behavior of SDPH in the smooth case, i.e., in the setting of smooth surfaces as opposed to cubical complexes in existing literature. 
	
	We begin by recalling our setting (see \Cref{sec:properties_signed_dist}): Let $\W = \W^-$ be a bounded open set of $\R^3$ with $C^k$-smooth boundary $\Surf$. We write $\R^3 = \W^- \sqcup \Surf \sqcup \W^+$. We furthermore assume that $\Surf$ is a generic embedded surface in the sense of \Cref{thm:generic} (\Cref{sec:genericity}): in particular, we assume $k \geq 3$. Recall that the associated signed distance field $d$ is defined by $d = \dist(\cdot, \W^-) - \dist(\cdot, \W^+)$.

	Consider the \textit{sublevel set filtration} of $d$, namely, the nested family of sublevel sets $X_r = \{x \in \R^3 ~|~ d(x) \leq r\}_{r \in \R}$,
	\[\Xb : \quad X_a \subseteq X_b, \quad \forall\ a \leq b \in \R\]
	endowed with the inclusion maps.
	Applying the homology functor with coefficients in $\mathbb{Z}_2$ to the filtration $\Xb$ gives the \textit{persistence module}
	\[\PH(\Xb) : \quad \H(X_a) \to \H(X_{b}), \quad a \leq b.\]
	In $\R^3$, we are interested in the persistent homologies in dimensions $k = 0, 1, 2$.
	
	$\PH(\Xb)$ is in fact determined by a finite sequence:
	\[\PH(\Xb) : \H(X_{r_1}) \to \ldots \to \H(X_{r_N})\]
	where $r_1,\ldots,r_N$ are the critical values of $d$. Indeed, by combining \Cref{thm:generic} (signed distances are generically Morse with finitely many critical points)
	with \Cref{thm:dist_isotopy_lemma} (isotopy lemma) and \Cref{thm:dist_handle_lemma} (handle attachment lemma), we obtain that the persistence module is of finite type in the sense of \cite[Section 3.8]{chazal_stability_2007}. 
	
	\begin{cor} \label{cor:gendistPH_finite}
		For a generically embedded surface $\Surf$ (as defined in \Cref{thm:generic}), the corresponding signed distance $d$ is a proper, lower-bounded, topological Morse function (\Cref{def:topological_morse_function}) that admits finitely many critical points. By \Cref{cor:PM_topological_morse_function_tame}, the persistence module $\PH(\Xb)$ is pointwise finite-dimensional; it can be decomposed into a finite direct sum of birth--death intervals in closed--open form $[b,d) \subset \bar{\R}$;
		and a critical point with index $\lambda$ corresponds to either a birth in homology dimension $\lambda$, or a death in homology dimension $\lambda - 1$.
	\end{cor}
	Recall from \Cref{sec:PH_and_Morse_theory} that the collections of birth--death pairs in each homology dimension are called the \textit{persistence diagrams}. 
	Here, the birth and death times correspond to the critical values $r_1, \ldots, r_N$.
	
	If the critical level $d^{-1}(r_i)$ contains a single critical point with index $\lambda$, then crossing that level consists of attaching a $\lambda$-cell, which either creates a new homology generator in dimension $\lambda$ (birth), or kills a homology generator in dimension $\lambda - 1$ (death). Similarly, if the critical level contains multiple critical points, each of them contributes to a handle. Therefore, any birth--death point in the persistence diagrams corresponds to a pair of creator--destroyer critical points.  Conversely, any critical point is associated to a birth--death point of the diagram, either as a creator or as a destroyer, except one critical point that creates the homological component in $\PH_0$ that has infinite lifetime and is paired with no other critical point.
	
	This allows for a rigorous mathematical definition of the SDPH methodology.
	\begin{defi}
		Given a generic input binary shape $\W$ with boundary $\Surf$, the \emph{signed distance persistent homology} (SDPH) methodology consists of building the sublevel set filtration $\Xb$ and computing the persistence diagrams of $\PH(\Xb)$.
	\end{defi}
	
	\subsection{Six Types of Critical Points}
	\label{sec:6_types_crit_pts}
	
	Based on the results of \Cref{sec:morse_theory_signdist}, we provide now a geometric interpretation of the non-degenerate critical points of $d$ and investigate the role of their index. This will enable us to describe SDPH diagrams precisely, since they consist of pairings of critical points across consecutive homology dimensions. Recall that the index is also the maximal dimension of a submanifold around $q$ such that the restriction of the function has a maximum at $q$.
	
	Non-degenerate critical points may be classified into $6$ different types, denoted by $\lambda^\sgn$, where $\lambda = 0,1,2,3$ is the index of the critical point $q$ (\Cref{def:ndg_crit_dist}), and $\sgn = \pm 1$ is the sign of $d(q)$.
	By definition, if $\mu$ refers to the index defined relative  to the pure distance function (\Cref{def:ndg_crit_dist}), we have that $\lambda = 3 - \mu$ if $q \in \W^-$ and $\lambda = \mu$ if $q \in \W^+$. 
	Locally, the shape of the surface around a critical point $q \in \W^+$ of type $\lambda^+$ is the same as for a critical point $q \in \W^-$ of type $(3-\lambda)^-$, as can be seen by exchanging the roles of $\W^-$ and $\W^+$ locally.
	
	For each type, one or several subtypes may be distinguished further, knowing that $\mu = (m - 1) + \mu'$, where $m = 2,3,4$ is the number of contact points and $\mu'$ is the Morse index of $\dist(\cdot, \Surf)$ restricted to the stratum of medial axis $G(q)$ (\Cref{def:ndg_crit_dist}). For example, $\mu = 3$ leads to three subtypes corresponding to $3 = 1 + 2$, $2 + 1$, $3 + 0$. Because the index roughly corresponds to the number of directions in which the function decreases from the point, the subtypes encode how many of these directions belong to the subspace spanned by the contact points versus the stratum of medial axis. For instance, points of type $2^+$ (subtype $2 + 0$) have two decreasing directions in the contact subspace and none along the medial axis, which locally corresponds to a ``bottleneck" shape.
	
	Finally, there are no critical points of type $0^+$ or $3^-$. Indeed, by definition $\mu = (m - 1) + \mu' \geq 1$ since $m \geq 2$. Intuitively, this is because the signed distance admits no local minimum in $\W^+$ (and equivalently, no local maximum in $\W^-$).
	
	In summary, we have $6$ different types and a total of $12$ subtypes, reported in Table \ref{tab:critical_types} and depicted in Figures \ref{fig:index_1_and_2} and \ref{fig:index_3}. The plots were generated with Geogebra \citep{hohenwarter_geogebra_2013}.
	
	\begin{table}[h!]
		\centering
		\fontsize{10pt}{10pt}\selectfont
		\begin{tabu}{|[1.2pt]c|[1pt]c|c|c|c|[1.2pt]}
			\tabucline[1.2pt]{-}
			& $\lambda = 0$ & $\lambda = 1$ & $\lambda = 2$ & $\lambda = 3$ \\ \tabucline[1pt]{-}
			$d < 0$ & \specialcell{\textbf{type $0^-$} \\ 3 subtypes \\ $\mu = 1 + 2$, $2 + 1$, $3 + 0$} & \specialcell{\textbf{type $1^-$} \\ 2 subtypes \\ $\mu = 1 + 1$, $2 + 0$} & \specialcell{\textbf{type $2^-$} \\ 1 subtype \\ $\mu = 1 + 0$} & \diagbox[innerwidth=1cm, height=\line]{}{} \\\hline
			$d > 0$ & \diagbox[innerwidth=1cm, height=\line]{}{} & \specialcell{\textbf{type $1^+$} \\ 1 subtype \\ $\mu = 1 + 0$} & \specialcell{\textbf{type $2^+$} \\ 2 subtypes \\ $\mu = 1 + 1$, $2 + 0$} & \specialcell{\textbf{type $3^+$} \\ 3 subtypes \\ $\mu = 1 + 2$, $2 + 1$, $3 + 0$} \\
			\tabucline[1.2pt]{-}
		\end{tabu}
		\caption{\textbf{Classification of non-degenerate critical points of the signed distance function in dimension $3$}. There are $6$ types and $12$ subtypes. The index $\lambda = \idx(q; d)$ is related to the index $\mu = \idx(q;\dist(\cdot, \Surf))$ through $\lambda = 3 - \mu$ if $q \in \W^-$ and $\lambda = \mu$ if $q \in \W^+$ (see \Cref{def:ndg_crit_dist}). Geometrically, the local shape of the surface around points of type $\lambda^+$ is the same as around those of type $(3-\lambda)^-$.}
		\label{tab:critical_types}
	\end{table}

	\begin{figure}[h!]
		\centering
		\includegraphics[clip, width=\linewidth]{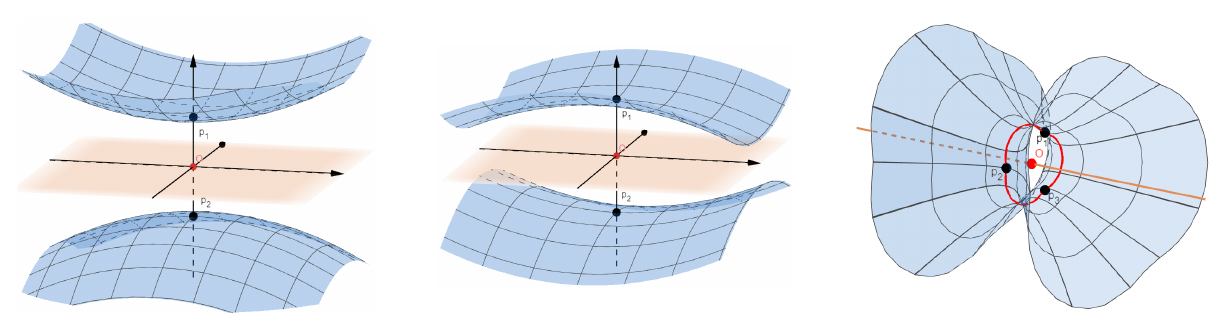}
		\caption{\textbf{Local shape of non-degenerate critical points of type $1^+$ or $2^-$ (subtype $1 + 0$), type $2^+$ or $1^-$ (subtype $1 + 1$) and type $2^+$ or $1^-$ (subtype $2 + 0$) respectively.}
			The red dot is $q$, black dots are $p_1,\ldots,p_m \in \Gamma(q)$, the orange plane or line is $G(q)$. The type of point $\lambda^+$ or $(3 - \lambda)^-$ depends on whether we consider $q \in \W^+$ or $q \in \W^-$.
		}
		\label{fig:index_1_and_2}
	\end{figure}

	\begin{figure}[h!]
		\centering
		\includegraphics[clip, width=\linewidth]{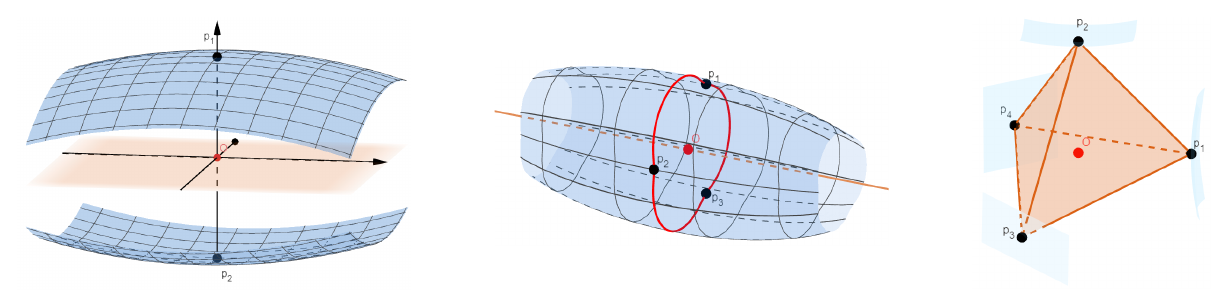}
		\caption{\textbf{Local shape of non-degenerate critical points of type $3^+$ or $0^-$, with subtypes $1 + 2$, $2 + 1$ and $3 + 0$ respectively.}
			The red dot is $q$, black dots are $p_1,\ldots,p_m \in \Gamma(q)$, the orange plane or line is $G(q)$. For subtype $3 + 0$, $G(q) = {q}$. The type of point $3^+$ or $0^-$ depends on whether we consider $q \in \W^+$ or $q \in \W^-$.
		}
		\label{fig:index_3}
	\end{figure}

	\subsection{Seven Types of Birth--Death Persistence Pairs}
	\label{sec:7_types_pairings}
	In SDPH, the critical points are paired into different types of birth--death persistence intervals, which we explain now using the classification of critical points from the previous section.
	Crossing a critical value with index $\lambda$ is topologically equivalent to attaching a $\lambda$-cell, which either creates a new homological component in dimension $\lambda$, or kills a previously-existing component in dimension $\lambda - 1$. Therefore, we obtain the following diagram (\Cref{fig:diagram_pairings}), where we can distinguish seven types of persistence pairings.
	
	\begin{figure}[h!]
		\centering
		\begin{tikzcd}
			0^- \arrow[rd, "\mathrm{II}"] \arrow[r, "\mathrm{I}"] & 1^- \arrow[rd, "\mathrm{IV}"] \arrow[r, "\mathrm{III}"] & 2^- \arrow[rd, "\mathrm{VI}"] & \\
			& 1^+\arrow[r, "\mathrm{V}"] & 2^+ \arrow[r, "\mathrm{VII}"] & 3^+
		\end{tikzcd}
		\caption{\textbf{$7$ different types of birth--death persistence pairs in the SDPH diagrams,} based on the classification of non-degenerate critical points given in Table \ref{tab:critical_types}.}
		\label{fig:diagram_pairings}
	\end{figure}
	
	Pairings of types $\mathrm{I}$ and $\mathrm{II}$ contribute to $\PH_0$, types $\mathrm{III}$, $\mathrm{IV}$, $\mathrm{V}$ contribute to $\PH_1$, and types $\mathrm{VI}$ and $\mathrm{VII}$ contribute to $\PH_2$. Recall that no critical points belong to $\Surf$, hence there is no zero critical value, and no birth or death point belonging to a main axis $b = 0$ and $d = 0$ of the diagrams. Furthermore, the points on the persistence diagrams are bounded away from the axes by the positive reach of surface, as the critical values of the distance functions must be greater than the reach \citep{niyogi_finding_2008}. Using the sign of the critical values, we can break down the types into further relationships with respect to the quadrants of the persistence diagrams.  Before doing so, we observe that unlike usual persistence diagrams (where filtrations are parameterized by some $\eps > 0$, as for Vietoris--Rips filtrations) that contain birth--death points in the NE quadrants only of each $\PH_k$, SDPH diagrams may contain points in the other quadrants as well since the signed distance has negative critical values.  Pairs of type $\mathrm{I}$ and  $\mathrm{II}$ are located in $\PH_0$ $\mathrm{SW}$ and $\mathrm{NW}$ respectively; $\mathrm{III}$, $\mathrm{IV}$, $\mathrm{V}$ pairs are in $\PH_1$ $\mathrm{SW}$, $\mathrm{NW}$, and $\mathrm{NE}$ respectively; and $\mathrm{VI}$, $\mathrm{VII}$ pairs are in $\PH_1$ $\mathrm{NW}$, $\mathrm{NE}$ respectively. 
	In particular, the absence of critical points of type $0^+$ or $3^-$ implies that no persistence pairs are located in the quadrants $\PH_0$ $\mathrm{NE}$ and  $\PH_2$ $\mathrm{NW}$. Because death times are strictly larger than birth times, all birth--death points lie strictly above the diagonal.
	Additionally, there is a point with infinite death time in $\PH_0$. Note that the number of components and the genus of $\Surf$ correspond to the number of birth--death points of type $\mathrm{II}$ and $\mathrm{IV}$, respectively. 
	These properties are summarized in \Cref{fig:pairings_quadrants} (generated with Geogebra \citep{hohenwarter_geogebra_2013}). 
	
	\begin{figure}[h!]
		\centering
		\includegraphics[clip, width=\linewidth]{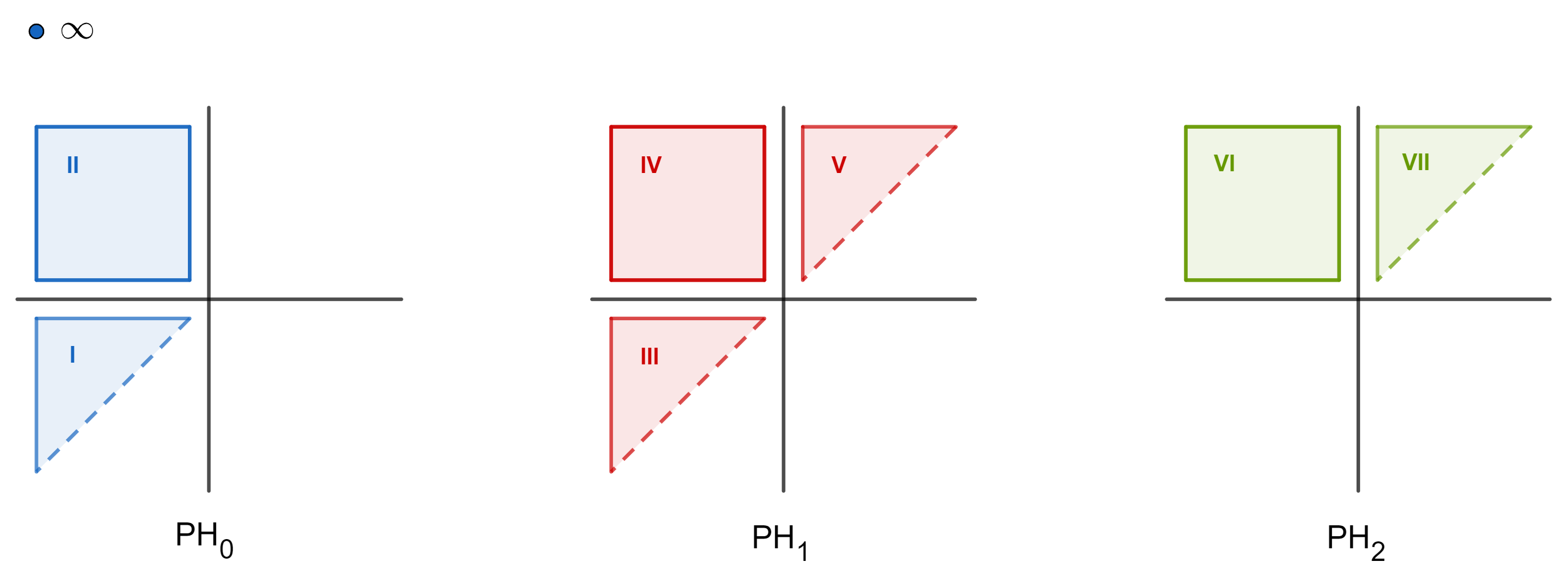}
		\caption{\textbf{Pairing types and quadrants in the SDPH persistence diagrams.} birth--death points (not represented) are located in the quadrants accordingly to their type (\Cref{fig:diagram_pairings}). 
		}
		\label{fig:pairings_quadrants}
	\end{figure}

	\subsection{Numerical Implementation}
	\label{sec:SDPH_numerical_implementation}

	To compute SDPH diagrams, we follow Algorithm \ref{algo:SDPH}.
	
	\begin{algorithm}
		\small
		\begin{algorithmic}[1]
			\Procedure{SDPH}{$\W$} \newline
			\textbf{output: persistence diagrams $\PH_0,\PH_1,\PH_2$ } \newline
			\State{$d \leftarrow \dist(\cdot, \W) - \dist(\cdot, \W^c)$}
			\Comment{build signed distance field}
			\State{$\Kb \leftarrow$ cubical complex filtered by sublevels of $d$}
			\Comment{T-constructed}
			\State{$\PH_0,\PH_1,\PH_2 \leftarrow \Dgm(\Kb)$}
			\Comment{compute persistence diagrams}
			\EndProcedure
		\end{algorithmic}
		\caption{SDPH: quantify shape texture of binary shape $\W$}
		\label{algo:SDPH}
	\end{algorithm}
	
	In practice, a 3D binary shape $\W$ is encoded as a 3D binary digital image (i.e., discrete array) with black and white values. We first used the \texttt{scikit-fmm}\footnote{Available at \url{https://github.com/scikit-fmm/scikit-fmm}.} Python package to convert $\W$ into a discrete signed distance field $d$ in the form of a 3D array. Next, we considered a finite \textit{cubical complex} filtered by the sublevel sets of $d$ \citep{kaczynski_cubical_2004, peikert_efficient_2012} :
	\[\Kb : K_1 \to \ldots \to K_N = K.\]
	Cubical complexes are the natural representation of digital images.
	We then computed the persistence diagrams (with coefficients in $\mathbb{Z}_2$)
	\[ \PH(\Kb) : \H(K_1) \to \ldots \to \H(K_n) \]
	in dimensions $0,1,2$ with the \texttt{giotto-tda} Python package \citep{tauzin_giotto-tda_2022}. 
	
	\texttt{Giotto-tda} uses the \texttt{GUDHI} package \citep{the_gudhi_project_gudhi_2015, dlotko_cubical_2021} as a back-end, and hence computes a T-constructed filtration \citep{bleile_persistent_2022}. Informally, this corresponds to increasing a level and considering each time the groups of voxels in $d$ having value less than this level. By adding groups of voxels at a time, components, loops, and voids are formed or destroyed. In vascular patterns, for instance, loops outlining cycles in the vessels start to form at negative values of $d$, before being filled out at positive values, which corresponds to pairings of type $\mathrm{IV}$ in $\PH_1$ $\mathrm{NW}$ (\Cref{fig:pairings_quadrants}).
	
	Note that the \texttt{CubicalRipser} \citep{kaji_cubical_2020} package may also be used to keep track of the critical voxels giving rise to or terminating homological components. Hence, a spatial analysis is also possible.
	
	\begin{rmk}
		As seen in the Introduction, \cite{delgado-friedrichs_skeletonization_2015}, \cite{herring_topological_2019} and \cite{moon_statistical_2019} studied SDPH directly in the setting of cubical complexes. A natural question to explore as future work is to relate the discrete Morse theories \citep{banchoff_critical_1967,forman_morse_1998, bauer_persistence_2011, lewiner_critical_2013, bloch_polyhedral_2013, saucan_discrete_2020} to our setting. In particular, it remains to be understood how critical points in the discrete settings relate to the original critical points, as the sampling resolution goes to infinity. Furthermore, how the gradient flow in the continuous case 
		relates to the gradient flow induced by a point-cloud distance field $d_\mathcal{P}$ may be studied. The latter, which was studied by \cite{giesen_medial_2006, giesen_flow_2008, giesen_parallel_2013}, is useful in applications such as surface reconstruction based on surface sampling.
	\end{rmk}
	
	\section{Quantifying Texture in Porous Shapes}
	\label{sec:quantifying_texture_applis}
	
	In this section, we showcase several applications of SDPH for quantifying texture in porous shapes. Such shapes may exhibit high morphological complexity and irregularity, including tubular, membranous and spherical elements, or mixtures of them, forming pores, cavities, channels, and so on.
	
	We computed SDPH diagrams for three case studies; two entail synthetic data and the third entails real data. 
	For the two case studies pertaining to synthetic data, the first involved simulated data from realizations of Gaussian random fields (specifically, level surfaces of GRFs), while the second involved simulated shapes generated using curvatubes \citep{song_generation_2022}; the third case study involves real samples of bone marrow vessels segmented from 3D confocal images (data courtesy of Antoniana Batsivari and Dominique Bonnet, The Francis Crick Institute, London).  The choice of these three case studies is to investigate the behavior where the randomness of the  underlying data generating process derives from a classical probability distribution, from a phase-field approximation, and from real-world noise (such as measurement, errors), respectively.  The data from each of the case studies are described in Sections \ref{sec:GRF_applis}, \ref{sec:cvtub_applis}, and \ref{sec:bio_applis} respectively. Their SDPH diagrams are compared in \Cref{sec:interpret_applis}, where a general interpretation of SDPH diagrams is given. The case studies involving synthetic data suggest in \Cref{sec:stability_SDPH} that SDPH diagrams may be stable with respect to the texture of the input shape, which in turn is related to the variability of its curvature.
	
	\textbf{Note on the Output Figures.}
	The same layout is used in Figures \ref{fig:GRF_panel_gau}, \ref{fig:cvtub_panel_1}, \ref{fig:cvtub_panel_2}, \ref{fig:BM_panel}, and \ref{fig:stab_panel}.
	The first column shows a 3D view of the surface bounding the binary shape, plotted with \textit{ParaView} \citep{ayachit_paraview_2015}. The second column shows a 2D section taken in the middle of the volume, after the shape has been smoothly closed near the boundaries of the simulation domain. A scalar field (random field, phase-field, or binary segmentation, depending on the data) is represented in there, whose zero level set gives the input surface. Finally, the three last columns display the SDPH diagrams $\PH_0$, $\PH_1$, and $\PH_2$ in the form of scatter plots, where birth--death points are colored by density (Gaussian kernel with $\sigma = 0.5$). Birth--death points $(\birth,\death)$ with persistence $\death - \birth$ less than $0.5$ were thresholded out.
	
	\subsection{Gaussian Random Fields}
	\label{sec:GRF_applis}
	We first computed SDPH diagrams on the zero level set of a \textit{Gaussian random field} (GRF) \citep{adler_geometry_2010, gaetan_spatial_2010}. GRFs may be used to model the random morphology of porous materials, such as rocks.
	
	A GRF on $\mathbb{R}^d$ is a real-valued stochastic process $f(t)$ indexed over points $t \in \mathbb{R}^d$, such that any finite dimensional distribution of $f$ indexed over any arbitrary finite set of points $T = \left\{t_1, \ldots, t_N\right\} \subset \mathbb{R}^d$ is a multivariate normal distribution:
	$$f(t_1), \ldots, f(t_N) \sim N(\mu_T, \Sigma_{T}).$$
	A classical GRF model is the zero-mean \emph{squared exponential covariance} model (a.k.a.~\textit{``Gaussian covariance model"}), where $\mu = 0$ and $\Sigma(s,t) = \mathrm{Cov}\left(f(s),\, f(t)\right) = e^{\frac{\|s-t\|^2}{2\lambda^2}}.$
	The covariance between $f(s)$ and $f(t)$ is controlled by a lengthscale parameter $\lambda > 0$: for a given spatial separation $\| s-t \|$, a larger $\lambda$ implies stronger correlation between $f(s)$ and $f(t)$, hence lower frequency features in the level sets at this scale.
	
	For GRFs with Gaussian covariance model, any realization $f$ (viewed as a function $\R^d \to \R$) admits a smooth version $\tilde{f}$ such that $\mathbb{P}(f(t) = \tilde{f}(t)) = 1$ for all $t$. Indeed, the continuity and differentiability of GRF samples are dictated by conditions on the covariance function \citep[Theorem 3.4.1]{adler_geometry_2010}, which in this case are satisfied by the Gaussian covariance model.
	
	We used \textit{GSTools} \citep{muller_gstools_2022} to generate $5$ random fields $f_1, \ldots, f_5$ in a $100 \times 100 \times 100$ domain size, using the parameter values presented in Table \ref{tab:GRF_params}. The covariance kernel here is defined by $\mathrm{Cov}\left(f(s),\, f(t)\right) = C(s - t) = \sigma^2 \, \rho_0(\|\mathrm{A} \, (s - t) \|)$, where $\rho_0(r) = e^{- \frac{\pi}{4} r^2 }$, $\mathrm{A} = \mathrm{diag}(\frac{1}{\ell_1},\,\frac{1}{\ell_2},\,\frac{1}{\ell_3}) \times \mathrm{R}$ and $\mathrm{R}$ is a rotation matrix. The lengthscales $\ell_1,\ell_2,\ell_3$ are given in terms of voxel size units. The matrix $\mathrm{A}$ encodes the isotropy of the process: differences in the lengthscales represent the anisotropy w.r.t.~the orthonormal frame rotated by $\mathrm{R}$.
	In all simulations, we took $\sigma^2 = 1$, since changing the variance does not statistically affect the behavior of the zero level set.
	
	The SDPH diagrams were computed with input shape $\W = \{f \geq 0\}$, the region of non-negative values, after smoothly closing the shape near the boundaries of the simulation domain. The results are displayed in \Cref{fig:GRF_panel_gau}.
	
	\begin{table}[h!]
		\centering
		\begin{tabu}{|[1.2pt]c|[1pt]c|[1pt]c|[1pt]c|[1pt]c|[1.2pt]}
			\tabucline[1.2pt]{-}
			GRF model & $\ell_1$ & $\ell_2$ & $\ell_3$ & other \\ \tabucline[1pt]{-}
			$F_1$ & ~~~$8$~~~ & $8$ & $8$ &   \\\hline
			$F_2$ & $8$ & $\frac{\ell_2}{\ell_1} = 0.7$ & $\frac{\ell_3}{\ell_1} = 0.85$ & \\\hline
			$F_3$ & $5$ & $5$ & $5$ & \\\hline
			$F_4$ &  &  &  & $f_4 = f_3 - 0.5$ \\\hline
			$F_5$ &  &  &  & $f_5 = f_1 + f_3$ \\\hline
			\tabucline[1.2pt]{-}
		\end{tabu}
		\caption{\textbf{Parameters of the Gaussian covariance model used to generate the random fields of \Cref{fig:GRF_panel_gau}}.}
		\label{tab:GRF_params}
	\end{table}
	
	\subsection{Curvatubes}
	\label{sec:cvtub_applis}
	
	The \textit{curvatubes} model \citep{song_generation_2022} was constructed to randomly simulate a large variety of organic-looking porous shapes.  Here, the output of curvatubes is precisely the synthetic realization of the shapes and textures that we aim to study with our SDPH framework and its process to produce random realizations is different in nature to the GRF case study. The curvatubes framework assumes that a porous shape optimizes some curvature functional involving a second-degree polynomial of the principal curvatures of the surface,
	\begin{equation*}
		\label{eq:F}
		\begin{split}
			\mathbf{F}(\Surf) &= \int_{\Surf} \big(a_{2,0} ~\kap_1^2 + a_{1,1} ~\kap_1 \kap_2 + a_{0,2} ~\kap_2^2 + a_{1,0} ~\kap_1 + a_{0,1} ~\kap_2 + a_{0,0} \big) ~dA,
		\end{split}
	\end{equation*}
	with a constraint of constant volume enclosed. Effectively, this geometric problem is approximated by a phase-field problem that is parameterized by a transition scale $\eps$ and a mass $m$ encoding the volume constraint. The optimization starts with random noise and converges towards a local optimum. The output surface is then implicitly represented as the zero level set of the phase-field.
	
	We used curvatubes to generate $10$ different shapes ($S_1, \ldots, S_{10}$) in a $100 \times 100 \times 100$ domain. The parameter values are presented in Tables \ref{tab:cvtub_params_1} and \ref{tab:cvtub_params_2}, where $\eps$ is defined relative to the side of the domain supposed to be $1$ unit. Their SDPH diagrams were then computed after smoothly closing the shapes near the boundaries of the simulation domain. The results are shown in Figures \ref{fig:cvtub_panel_1} and \ref{fig:cvtub_panel_2}.

	\begin{table}[h!]
		\centering
		\begin{tabu}{|[1.2pt]c|[1pt]c|[1pt]c|[1pt]c|[1.2pt]}
			\tabucline[1.2pt]{-}
			Shape & Coefficients & $m$ & $\eps$ \\ \tabucline[1pt]{-}
			
			$S_1$ & {\small $(1, 2, 6, -40, -40, 400)$} & $-0.6$ & $0.02$  \\\hline
			$S_2$ & {\small $(1, -1.284, 9.626, -77.283, 39.681, -633.849)$ } & $-0.421$ & $0.02$ \\\hline
			$S_3$ & {\small $(1, 2.8, 2, -10, -10, 25)$ } & $-0.25$ & $0.02$ \\\hline
			$S_4$ & {\small $(1, 3.225, 10.87, -146.309, 143.487, -2920.302)$} & $-0.499$ & $0.02$  \\\hline
			$S_5$ & {\small $(1, -0.238, 11.988, -175.909, -27.167, 2117.037)$} & $-0.648$ & $0.02$  \\\hline
			\tabucline[1.2pt]{-}
		\end{tabu}
		\caption{\textbf{Parameters of the curvatubes model used for \Cref{fig:cvtub_panel_1}}.}
		\label{tab:cvtub_params_1}
	\end{table}

	\begin{table}[h!]
		\centering
		\begin{tabu}{|[1.2pt]c|[1pt]c|[1pt]c|[1pt]c|[1.2pt]}
			\tabucline[1.2pt]{-}
			Shape & Coefficients & $m$ & $\eps$ \\ \tabucline[1pt]{-}
			
			$S_6$ & {\small $(1, 2.034, 11.166, 14.553, 28.829, -565.092 )$} & $-0.356 $ & $0.02$  \\\hline
			$S_7$ & {\small $(1, 0.63, 4.399, 132.459, 195.066, -2378.53 )$} & $-0.364 $ & $0.02$  \\\hline
			$S_8$ & {\small $( 1, 1, 1, 0, 0, 0)$} & $-0.4 $ & $0.02$  \\\hline
			$S_9$ & {\small $(1, 0.396, 1.095, -28.64, 190.906 ,2062.082 )$} & $-0.598 $ & $0.02$  \\\hline
			$S_{10}$ & {\small $(1, 4.185, 2.053, 19.375, 29.607, 120.265)$} & $-0.534 $ & $0.02$  \\\hline
			\tabucline[1.2pt]{-}
		\end{tabu}
		\caption{\textbf{Parameters of the curvatubes model used for \Cref{fig:cvtub_panel_2}}.}
		\label{tab:cvtub_params_2}
	\end{table}
	
	\subsection{Vascular Data}
	\label{sec:bio_applis}
	
	As our final case study, we considered $3$ samples $B_1, B_2, B_3$ of bone marrow vessels (``BM shapes") taken from a healthy control mouse. These samples were selected at different anatomical locations of the same femur bone, and segmented on 3D confocal images. The selected vascular crops were $100 \times 200 \times 200$ in size, where a voxel side represents $2 \, \mu m$ in physical length. Due to noise in the signal, and the complexity of the task, the segmentation may not be fully accurate.

	\subsection{Interpreting SDPH Diagrams}
	\label{sec:interpret_applis}
	
	From the classification of critical points and pairings described in Sections \ref{sec:6_types_crit_pts} and \ref{sec:7_types_pairings}, we may derive a \textit{general interpretation} of SDPH diagrams, supported by a comparison of the diagrams in Figures \ref{fig:GRF_panel_gau}, \ref{fig:cvtub_panel_1}, \ref{fig:cvtub_panel_2}, and \ref{fig:BM_panel}.
	
	SDPH may be interpreted as follows:
	\begin{enumerate}
		\item $\PH_0$ $\SW$ measures variations of thicknesses in tubular and membranous structures.
		\item $\PH_0$ $\NW$ counts the number of components and measures their characteristic size and separation. 
		\item $\PH_1$ $\SW$ detects the presence of dimples (similar to the hollowed shape of red blood cells). 
		\item $\PH_1$ $\NW$ counts the number of loops and measures their characteristic thickness and size. 
		\item $\PH_1$ $\NE$ measures variations in the proximity of the structures, creating at a distance curvy and non-convex loops. 
		\item $\PH_2$ $\NW$ indicates the presence of small voids trapped inside the shape, which do not exist in tubular shapes. 
		\item $\PH_2$ $\NE$ measures the characteristic sizes of interspaces separating plain structures and the density of tubular branching, so that larger sizes mean sparser structures. 
	\end{enumerate}
	
	Each of these properties may be observed in the following examples. For instance, 
	\begin{enumerate}
		\item Thicknesses are much more variable in the biological shapes $B_i$ than the synthetic shapes $S_i$. 
		\item Both of $S_2$ and $S_5$ have a large number of disconnected components, but in $S_2$ they are typically separated by the same distance. 
		\item Dimples are mostly found in the GRF shapes $F_i$ and the BM shapes $B_i$, but not in the tubular shapes $S_1, \ldots, S_4$.
		\item There is an increasing number of loops forming in $F_3$ than $F_1$, as the length scale parameter to generate the GRFs is decreased.
		\item When the shape is thinned, by considering different textures, as for $S_3$, $S_4$, $S_5$, or by considering different level sets, as for $F_3 = \{f_3 = 0\}$ and $F_4 = \{f_3 = 0.5\}$, homological loops tend to disappear from the initial structure, in favor of those that are formed at a distance in a non-convex curvy way.
		\item Only rarely small voids are spotted inside the shape, except for shapes $S_9$, $F_4$ and $F_5$. 
		\item Finally, while homological voids formed at a distance are mostly of the same size in the densely windowed shapes $S_9$ and $S_{10}$, they are found at much larger and variable sizes in the sparse vascular BM shapes $B_1$ and $B_3$.
	\end{enumerate}

	Finally, note that \textit{multi-scale texture} can be quantified by SDPH, as can be seen in the diagram of $S_9$, where most loops form at similar critical sizes, but a few of them appear at variable larger scales; or for $F_5$, which by construction mixes features taken from $F_1$ and $F_3$ at different scales. Generally, mixtures of different textures in space or in scale are apparent in the SDPH diagrams.
	Further observations can be made: In all quadrants, a concentration (or spread) of the scatter plot indicates \textit{(in)homogeneities}, not only in the distribution of the critical sizes, but particularly in terms of the morphological patterns described above for each respective quadrant. For instance, the shapes $S_7, S_8, S_9$ form cycles of irregular sizes, in contrast to the shapes $S_3$ and $S_4$; $B_3$ forms irregular-sized loops compared to $B_2$. Moreover, birth--death points close to the main axes arise in \textit{noisy} structures whose topological features have close to zero critical sizes, whereas \textit{smooth} structures induce a clearer gap from the axes. This effect is apparent here: curvatubes shapes are smooth and lead to clearly distinct diagrams, compared to GRF shapes which are topologically noisy and whose diagrams are more difficult to interpret, which may be expected due to the difference in the nature of their randomness.

	\begin{figure}[h!]
		\centering
		\includegraphics[clip, width=\linewidth]{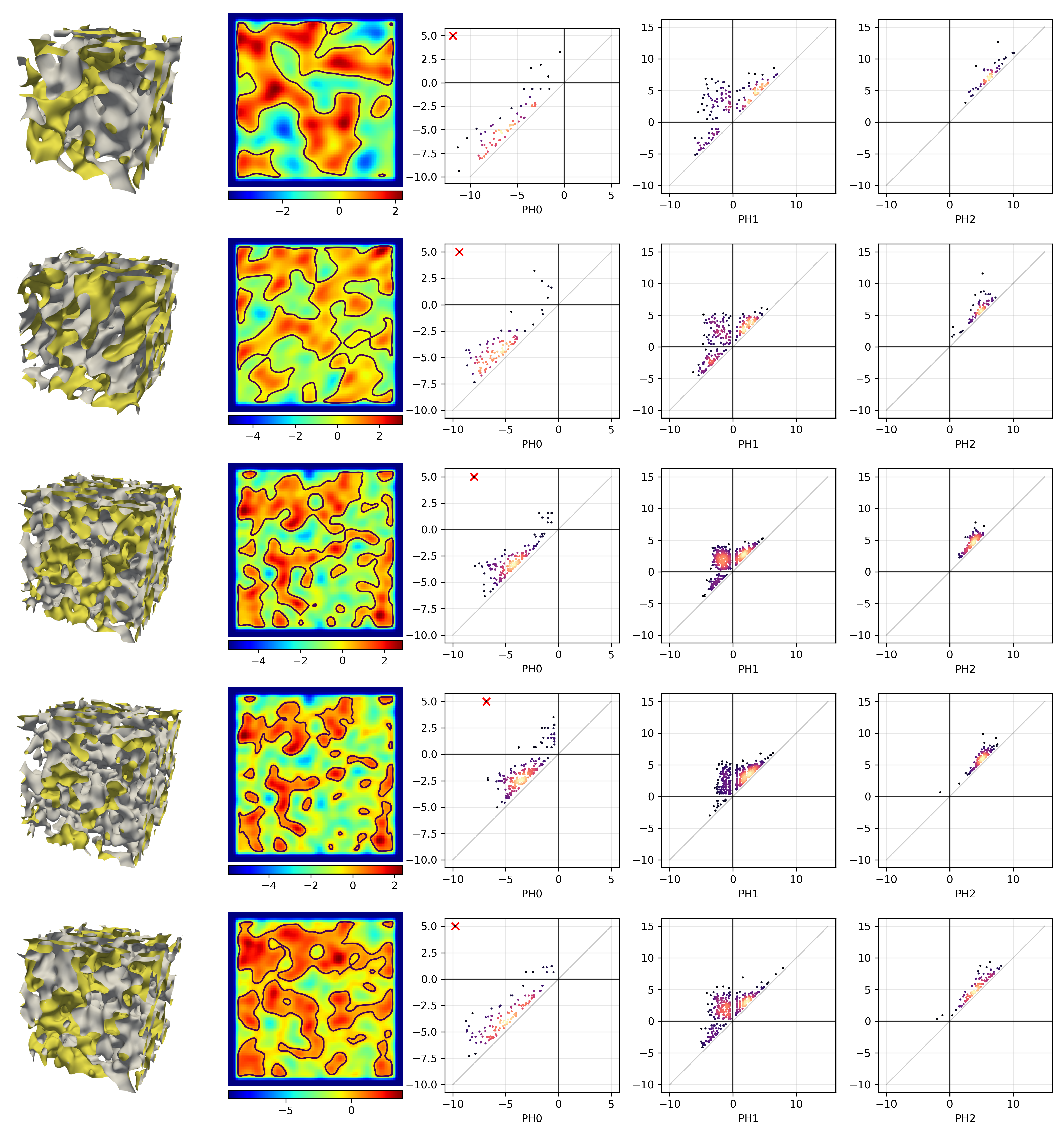}
		\caption{\textbf{SDPH diagrams of the zero level set of $5$ Gaussian random fields $F_1, \ldots, F_5$ generated with the parameters shown in Table \ref{tab:GRF_params}}.}
		\label{fig:GRF_panel_gau}
	\end{figure}

	\begin{figure}[h!]
		\centering
		\includegraphics[clip, width=\linewidth]{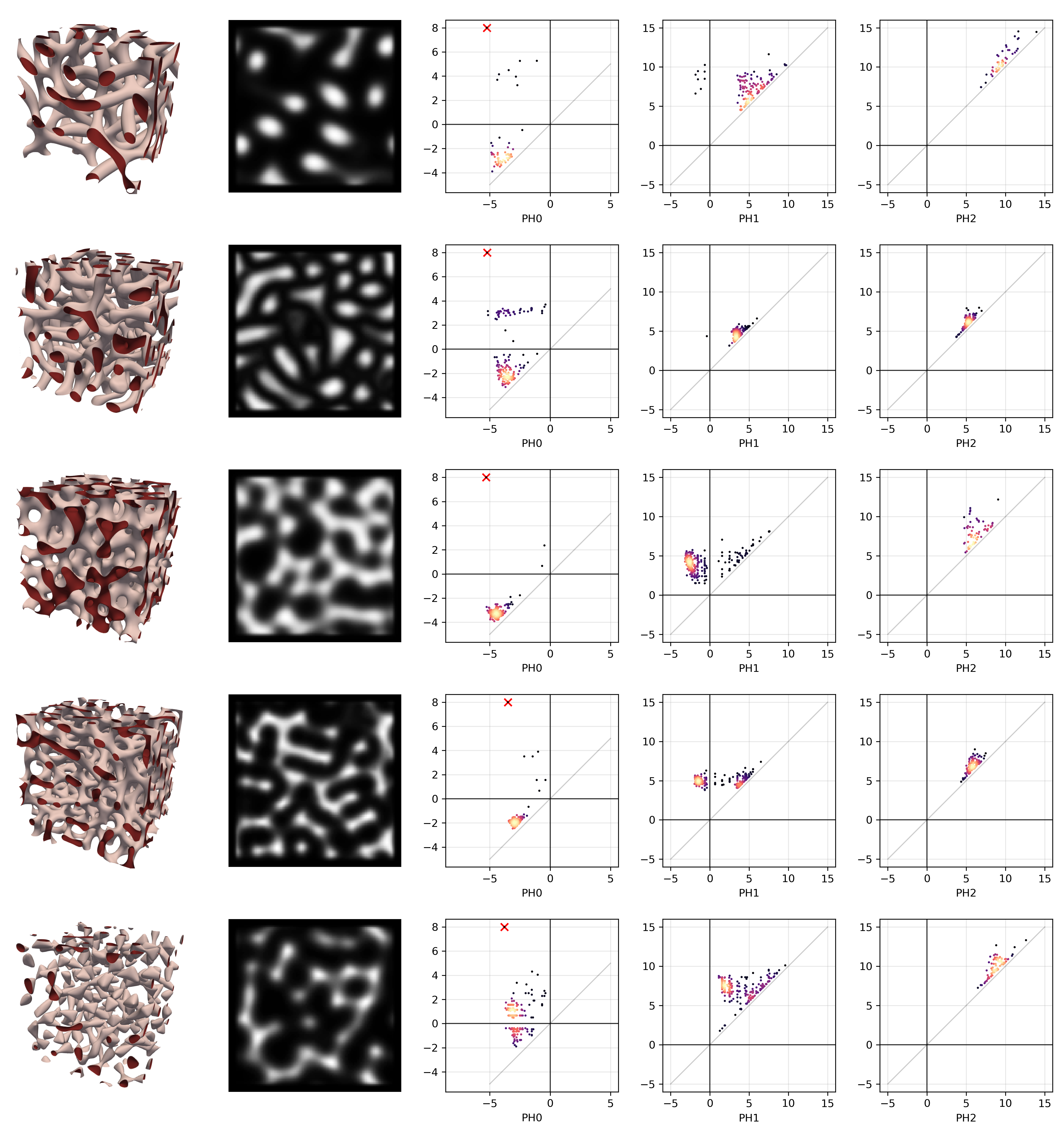}
		\caption{\textbf{SDPH diagrams of $5$ tubular shapes $S_1, \ldots, S_5$ generated with curvatubes using the parameters shown in Table \ref{tab:cvtub_params_1}}.}
		\label{fig:cvtub_panel_1}
	\end{figure}
	
	\begin{figure}[h!]
		\centering
		\includegraphics[clip, width=\linewidth]{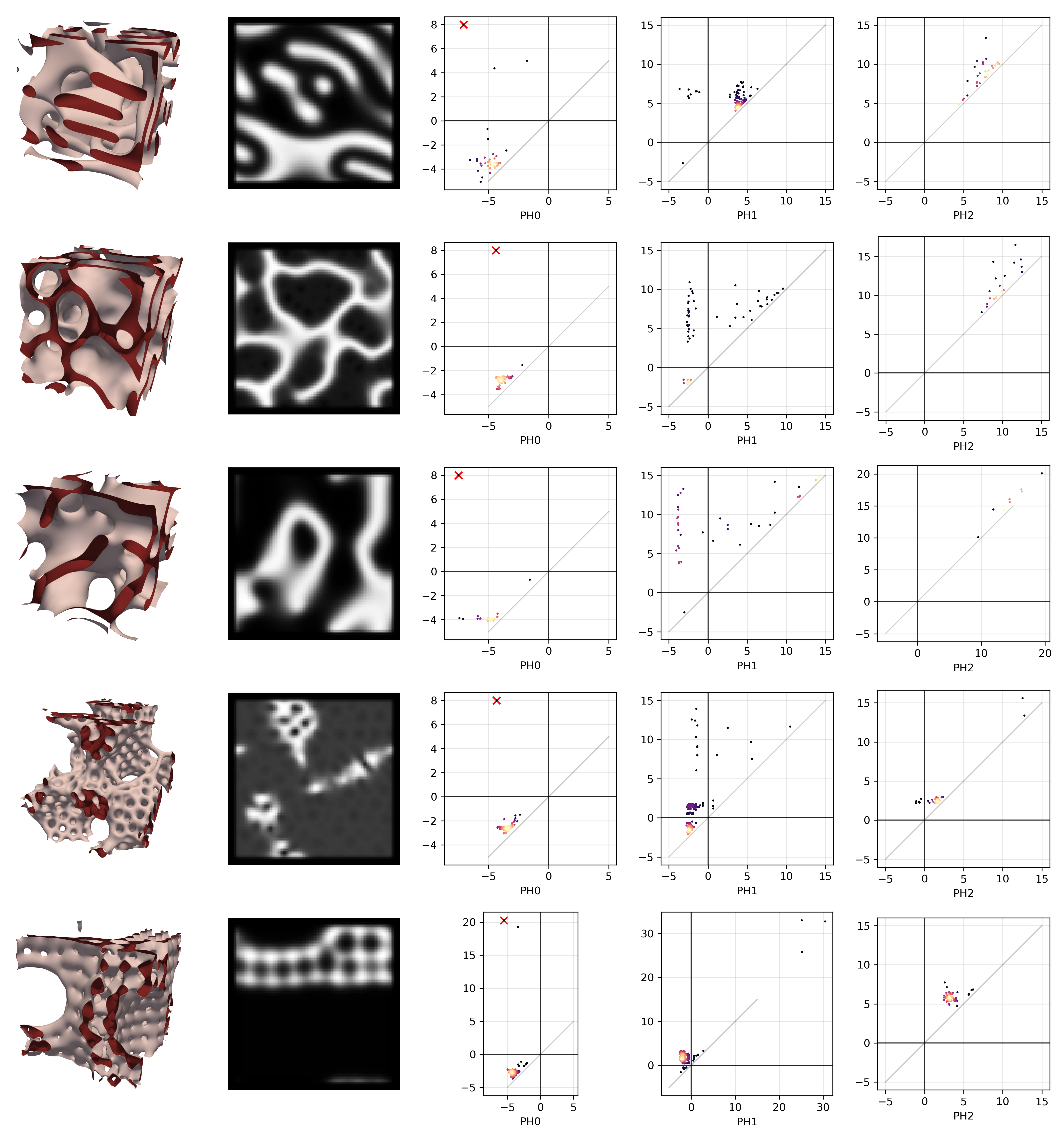}
		\caption{\textbf{SDPH diagrams of $5$ membranous and porous shapes $S_6, \ldots, S_{10}$ generated with curvatubes using the parameters shown in Table \ref{tab:cvtub_params_2}}.}
		\label{fig:cvtub_panel_2}
	\end{figure}

	\begin{figure}[h!]
		\centering
		\includegraphics[clip, width=\linewidth]{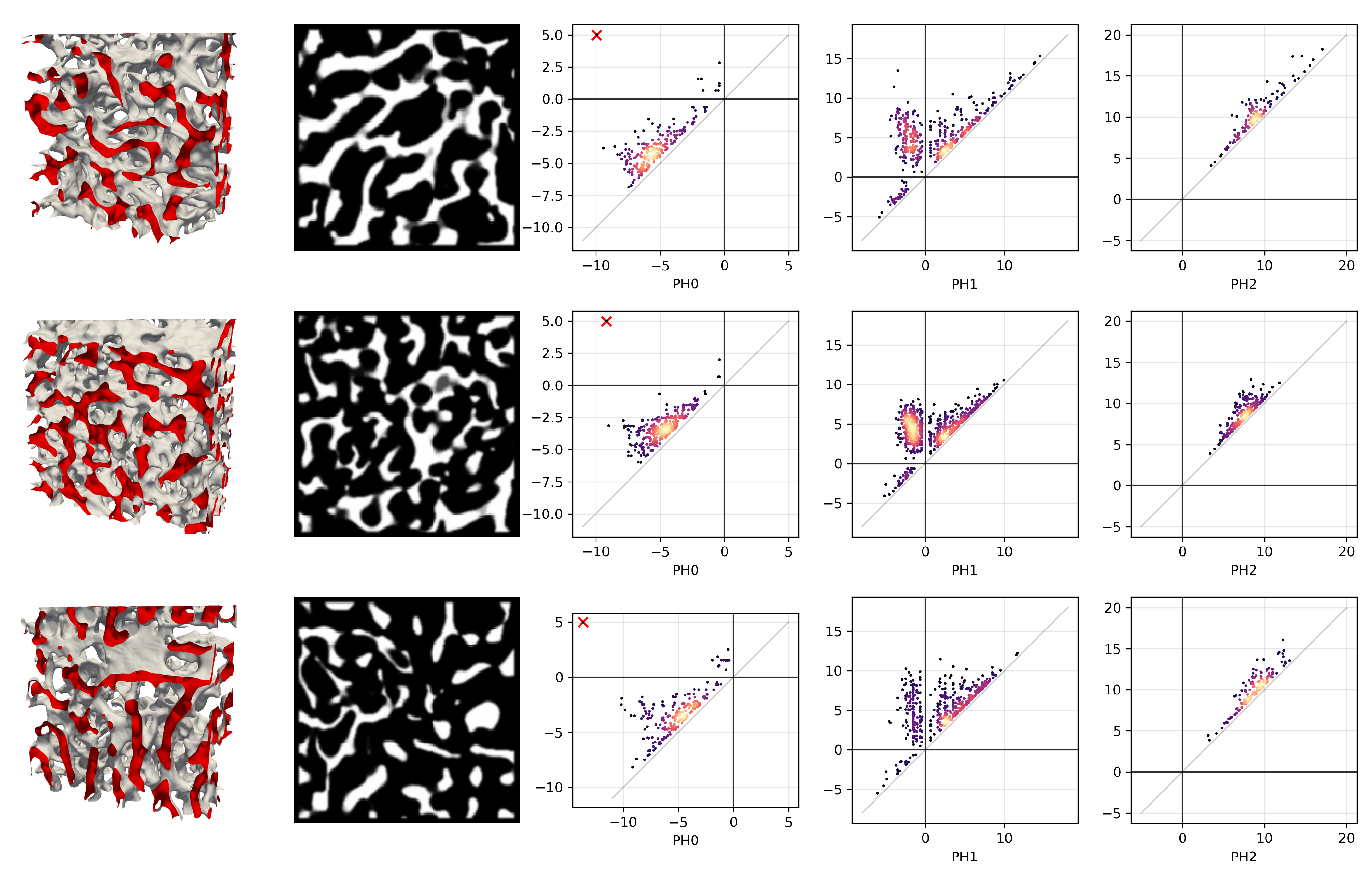}
		\caption{\textbf{SDPH diagrams of $3$ bone marrow vascular samples $B_1, B_2, B_3$ in a healthy mouse (data courtesy of Antoniana Batsivari and Dominique Bonnet, The Francis Crick Institute, London)}. }
		\label{fig:BM_panel}
	\end{figure}
	
	\subsection{Texture, Curvatures, and Stability of SDPH}
	\label{sec:stability_SDPH}
	
	Texture in shapes may be characterized as spatially repeated patterns of the surface, subject to some randomness, but not easily discriminated by visual perception \citep{julesz_visual_1962,portilla_parametric_2000,landy_visual_2004}.
	There is hence some quantity, left unchanged by some type of geometric randomness, that represents texture.
	In a previous work, we suggested that one possible texture descriptor is the 
	behavior of the curvatures on the surface, or ``curvature diagram" \citep{song_generation_2022}. Given the same choice of generation parameters $\mathbf{a}$, but with random initializations in curvatubes, we illustrated that the output curvature diagram was regularly stable with respect to $\mathbf{a}$.
	
	Here, it is natural to ask whether SDPH diagrams also quantify texture and exhibit a similar stability behavior. 
	We generated five synthetic shapes with the same choice of curvatubes parameters $\mathbf{a}$ but random initializations of the optimization flow. The results are shown in \Cref{fig:stab_panel},
	where we used a $150 \times 150 \times 150$ simulation grid, with mass $m = -0.5$, $\eps = 0.02$ (one pixel being of size $0.01$) and coefficients $\mathbf{a} = (6, 2.5, 6, -230, 20, 2350)$. These values correspond to the non-reduced curvature polynomial $p(\kap_1,\kap_2) = (H-15)^2 + .5 \, K + 5 \, (\kap_1 - 20)^2 + 5 \, (\kap_2 + 5)^2 $ where $H = \kap_1 + \kap_2$ and $K = \kap_1 \, \kap_2$.

	We found that, for a given choice of texture, the SDPH method produced very similar-looking output diagrams, especially in contrast to those produced for other generation parameters, shown in Figures \ref{fig:cvtub_panel_1} and \ref{fig:cvtub_panel_2}. Empirically, we observed the same consistent behavior for other choices of shape textures (data not shown here, but is available on our GitHub repository).
	In other words, despite the shapes being very different geometrically, the distribution of their topological patterns as measured by SDPH appears to be very similar. 
	
	The well-known Stability Theorem \citep{cohen-steiner_stability_2007} from classical persistent homology theory states that two functions which are close in infinity-norm lead to persistence diagrams which are close in bottleneck distance (see \Cref{sec:PH_and_Morse_theory}, \Cref{thm:stability_thm}); namely, that similar shapes lead to similar persistence diagrams. This guarantees stability when the input shapes are geometrically perturbed by noise (which results in a small infinity-norm between them), a property consequently satisfied by SDPH diagrams. Here, our observations suggest that \textit{``SDPH diagrams are stable with respect to texture"}, which may give rise to a new kind of stability result that extends beyond than existing stability results in classical persistent homology. The novelty of this kind of stability study lies in a more geometric and probabilistic direction, as well as the fact that there is no existing stability result entailing SDPH, to the best of our knowledge, since SDPH has been comparatively less studied than classical persistence. Finally, it will be interesting to explore as future work how curvature diagrams relate to SDPH diagrams.

	\begin{figure}[h!]
		\centering
		\includegraphics[clip, width=\linewidth]{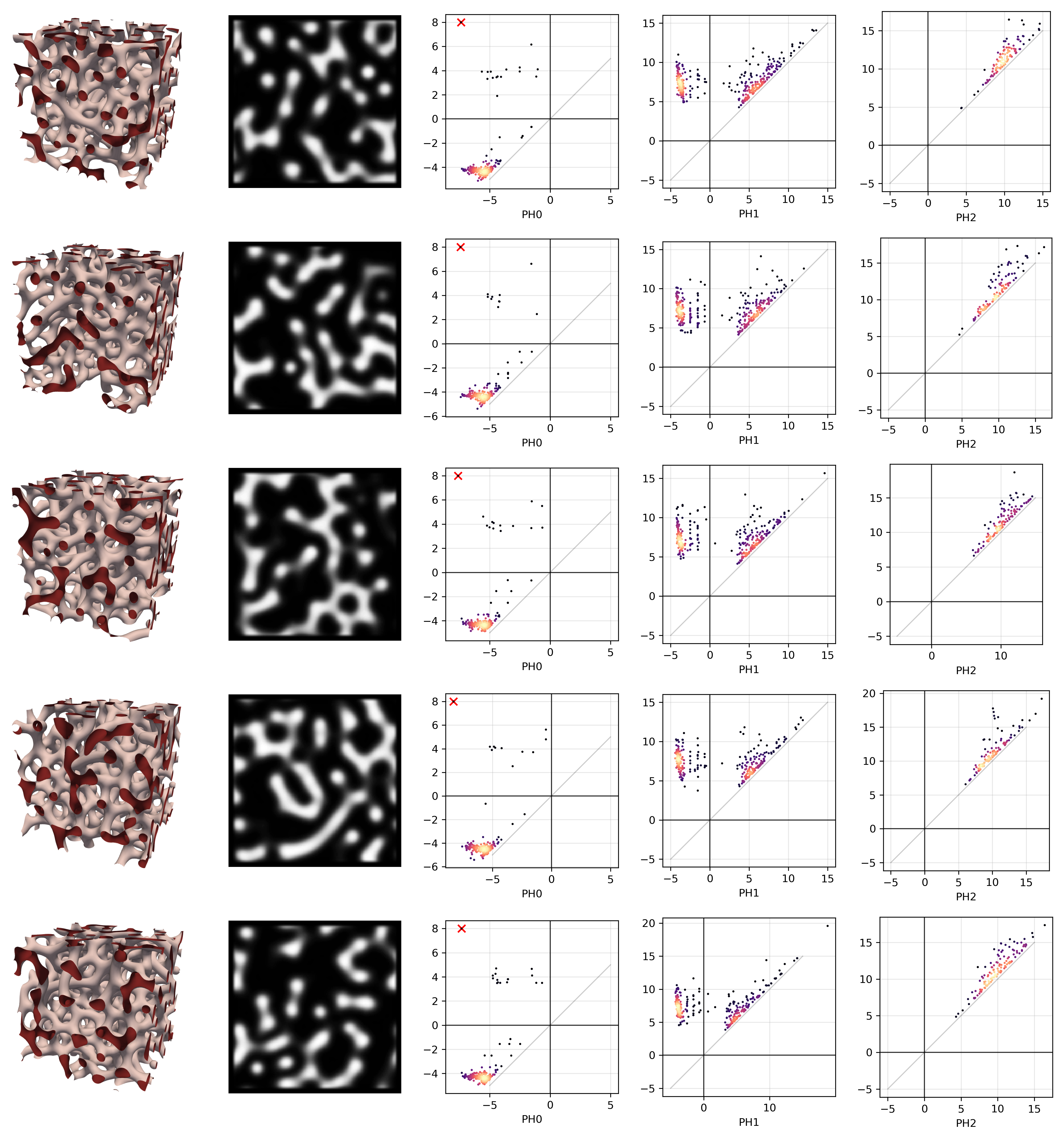}
		\caption{\textbf{Five random shapes sharing a similar texture and similar SDPH diagrams}. They were generated with curvatubes using the same parameters but different initializations.}
		\label{fig:stab_panel}
	\end{figure}
	
	\section{Discussion}
	
	Morse theory is about the study of critical points of a function, and how they characterize geometric and topological information from the underlying space. While classical Morse theory deals with smooth functions, this paper generalizes it to the broad class of Euclidean distance functions, signed or unsigned, and fills an important gap in building a complete Morse theory.
	
	We have successfully established the two basic Morse lemmas (isotopy lemma and handle attachment lemma), namely, that the topological changes encountered by the sublevel sets of Euclidean distance functions (signed or not) with non-degenerate critical points are exactly dictated by the latter ones. 
	To do this, we reframed Euclidean distances as Min-type functions (signed distances are then locally Min-type outside the shape and Max-type inside), which in turn may be recast as topological Morse functions.
	As a further contribution, we have also provided complete proofs extending the Morse lemmas to topological Morse functions.
	
	As a focal crux of our study, we have identified the correct geometric conditions corresponding to the Min-type non-degeneracy at a critical point,
	and shown that, in $\R^3$, for generic embeddings of the shape, the signed distance field only admits a finite number of critical points that are all non-degenerate. In other words, generic signed distances are topological Morse functions with finitely many critical points.
	This has important implications, since it guarantees that essentially any shape may be filtered by the sublevel sets of its signed distance field and that the persistent homology of this filtration may be computed, namely the ``signed distance persistent homology" (SDPH). 
	
	We have demonstrated how to use SDPH as a quantifier of \textit{texture} in shapes, based on the patterns of birth--death pairs in persistence diagrams.
	We have proposed a rigorous geometric interpretation of the critical points involved in the diagrams, alongside a classification of persistence pairs, which leads to a practical guide for interpreting SDPH diagrams in real applications. Furthermore, SDPH seems to attribute similar diagrams to shapes with similar texture.
	
	Case studies have been carried out on both simulated and real data,
	demonstrating that our proposed framework may be readily applied to several fields, including shape analysis, imaging, computer graphics, materials science, biology, environmental science and geospatial science to rigorously and interpretably quantify shapes and textures.
	
	\section*{Software and Data Availability}
	
	The code to compute SDPH, generate the synthetic data, and the real data are all publicly available on the GitHub repository \url{https://github.com/annasongmaths/SDPH}. Details of the numerical implementation are as above in \Cref{sec:SDPH_numerical_implementation}.
	
	\section*{Acknowledgments}
	
	The authors would like to warmly thank Dominique Bonnet and Antoniana Batsivari for providing the vascular data, which also motivated the study of the SDPH approach.
	
	A.S. is particularly grateful to Omer Bobrowski, Robert Adler, Antoine Song and Jean Feydy for helpful and interesting discussions.
	
	
	K.M.Y. thanks Vidit Nanda for insightful comments, as well as Peter Grindrod and John Harvey for their guidance and mentorship. 
	
	
	The authors wish to acknowledge The Francis Crick Institute, London
	for their computing resources which were used to implement the numerical experiments in this paper.
	
	{\footnotesize 
		\bibliographystyle{authordate3_modified}
		\bibliography{SDPH_Morse}

\begin{thebibliography}{}

\bibitem[\protect\citename{Abraham \& Robbin, }1967]{abraham_transversal_1967}
{\sc Abraham~\& Robbin}. 1967.
\newblock {\em Transversal mappings and flows}.
\newblock WA Benjamin New York.

\bibitem[\protect\citename{Absil {\em et~al.\@}\relax,
  }2013]{absil_extrinsic_2013}
{\sc Absil, Mahony~\& Trumpf}. 2013.
\newblock An {Extrinsic} {Look} at the {Riemannian} {Hessian}.
\newblock {\em Pages  361--368 of:} {\sc Nielsen~\& Barbaresco} (eds), {\em
  Geometric {Science} of {Information}}.
\newblock Lecture {Notes} in {Computer} {Science}.
\newblock Berlin, Heidelberg: Springer.

\bibitem[\protect\citename{Adler, }2010]{adler_geometry_2010}
{\sc Adler}. 2010.
\newblock {\em The {Geometry} of {Random} {Fields}}.
\newblock Classics in {Applied} {Mathematics}.
\newblock Society for Industrial and Applied Mathematics.

\bibitem[\protect\citename{Agrachev {\em et~al.\@}\relax,
  }1997]{agrachev1997morse}
{\sc Agrachev, Pallaschke~\& Scholtes}. 1997.
\newblock On Morse theory for piecewise smooth functions.
\newblock {\em Journal of Dynamical and Control Systems}, {\bf 3}(4), 449--469.

\bibitem[\protect\citename{Albano {\em et~al.\@}\relax,
  }2013]{albano_singular_2013}
{\sc Albano, Cannarsa, Nguyen~\& Sinestrari}. 2013.
\newblock Singular gradient flow of the distance function and homotopy
  equivalence.
\newblock {\em Mathematische Annalen}, {\bf 356}(1), 23--43.

\bibitem[\protect\citename{Armstrong {\em et~al.\@}\relax,
  }2021]{armstrong_correspondence_2021}
{\sc Armstrong, Lanetc, Mostaghimi, Zhuravljov, Herring~\& Robins}. 2021.
\newblock Correspondence of max-flow to the absolute permeability of porous
  systems.
\newblock {\em Physical Review Fluids}, {\bf 6}(5), 054003.
\newblock Publisher: American Physical Society.

\bibitem[\protect\citename{Arnol'd, }1974]{arnol1974normal}
{\sc Arnol'd}. 1974.
\newblock Normal forms of functions in neighbourhoods of degenerate critical
  points.
\newblock {\em Russian mathematical surveys}, {\bf 29}(2), 10.

\bibitem[\protect\citename{Attali {\em et~al.\@}\relax,
  }2009]{attali_stability_2009}
{\sc Attali, Boissonnat~\& Edelsbrunner}. 2009.
\newblock Stability and {Computation} of {Medial} {Axes} - a
  {State}-of-the-{Art} {Report}.
\newblock {\em Pages  109--125 of:} {\sc Möller, Hamann~\& Russell} (eds),
  {\em Mathematical {Foundations} of {Scientific} {Visualization}, {Computer}
  {Graphics}, and {Massive} {Data} {Exploration}}.
\newblock Mathematics and {Visualization}.
\newblock Berlin, Heidelberg: Springer.

\bibitem[\protect\citename{Ayachit, }2015]{ayachit_paraview_2015}
{\sc Ayachit}. 2015.
\newblock {\em The paraview guide: a parallel visualization application}.
\newblock Kitware, Inc.

\bibitem[\protect\citename{Banchoff, }1967]{banchoff_critical_1967}
{\sc Banchoff}. 1967.
\newblock Critical points and curvature for embedded polyhedra.
\newblock {\em Journal of Differential Geometry}, {\bf 1}(3-4).

\bibitem[\protect\citename{Bauer, }2011]{bauer_persistence_2011}
{\sc Bauer}. 2011 (July).
\newblock {\em Persistence in discrete {Morse} theory}.
\newblock Ph.D. thesis, Georg-August-Universität Göttingen.
\newblock Accepted: 2011-07-15T15:27:36Z.

\bibitem[\protect\citename{Bauer {\em et~al.\@}\relax, }2012]{bauer2012optimal}
{\sc Bauer, Lange~\& Wardetzky}. 2012.
\newblock Optimal topological simplification of discrete functions on surfaces.
\newblock {\em Discrete \& computational geometry}, {\bf 47}, 347--377.

\bibitem[\protect\citename{Birbrair \& Denkowski, }2017]{birbrair_medial_2017}
{\sc Birbrair~\& Denkowski}. 2017.
\newblock Medial {Axis} and {Singularities}.
\newblock {\em Journal of Geometric Analysis}, {\bf 27}(3), 2339--2380.

\bibitem[\protect\citename{Bleile {\em et~al.\@}\relax,
  }2022]{bleile_persistent_2022}
{\sc Bleile, Garin, Heiss, Maggs~\& Robins}. 2022.
\newblock The {Persistent} {Homology} of {Dual} {Digital} {Image}
  {Constructions}.
\newblock {\em Pages  1--26 of:} {\sc Gasparovic, Robins~\& Turner} (eds), {\em
  Research in {Computational} {Topology} 2}.
\newblock Association for {Women} in {Mathematics} {Series}.
\newblock Cham: Springer International Publishing.

\bibitem[\protect\citename{Bloch, }2013]{bloch_polyhedral_2013}
{\sc Bloch}. 2013.
\newblock Polyhedral representation of discrete {Morse} functions.
\newblock {\em Discrete Mathematics}, {\bf 313}(12), 1342--1348.

\bibitem[\protect\citename{Bobrowski \& Adler, }2014]{bobrowski_distance_2014}
{\sc Bobrowski~\& Adler}. 2014.
\newblock Distance functions, critical points, and the topology of random Čech
  complexes.
\newblock {\em Homology, Homotopy and Applications}, {\bf 16}(2), 311--344.

\bibitem[\protect\citename{Bobrowski \& Oliveira, }2019]{bobrowski2019random}
{\sc Bobrowski~\& Oliveira}. 2019.
\newblock Random {\v{C}}ech complexes on Riemannian manifolds.
\newblock {\em Random Structures \& Algorithms}, {\bf 54}(3), 373--412.

\bibitem[\protect\citename{Bobrowski \& Weinberger,
  }2017]{bobrowski2017vanishing}
{\sc Bobrowski~\& Weinberger}. 2017.
\newblock On the vanishing of homology in random {\v{C}}ech complexes.
\newblock {\em Random Structures \& Algorithms}, {\bf 51}(1), 14--51.

\bibitem[\protect\citename{Borsuk, }1948]{borsuk_imbedding_1948}
{\sc Borsuk}. 1948.
\newblock On the imbedding of systems of compacta in simplicial complexes.
\newblock {\em Fundamenta Mathematicae}, {\bf 35}(1), 217--234.

\bibitem[\protect\citename{Bott, }1988]{bott1988morse}
{\sc Bott}. 1988.
\newblock Morse theory indomitable.
\newblock {\em Publications Math{\'e}matiques de l'IH{\'E}S}, {\bf 68},
  99--114.

\bibitem[\protect\citename{Bruce {\em et~al.\@}\relax, }1992]{bruce1992curves}
{\sc Bruce, Bruce~\& Giblin}. 1992.
\newblock {\em Curves and Singularities: a geometrical introduction to
  singularity theory}.
\newblock Cambridge university press.

\bibitem[\protect\citename{Cantwell, }1968]{cantwell_topological_1968}
{\sc Cantwell}. 1968.
\newblock Topological non-degenerate functions.
\newblock {\em Tohoku Mathematical Journal}, {\bf 20}(2), 120--125.
\newblock Publisher: Tohoku University, Mathematical Institute.

\bibitem[\protect\citename{Cao \& Monod, }2022]{cao_approximating_2022}
{\sc Cao~\& Monod}. 2022 (May).
\newblock {\em Approximating {Persistent} {Homology} for {Large} {Datasets}}.
\newblock arXiv:2204.09155 [cs, math, stat].

\bibitem[\protect\citename{Carlsson, }2009]{carlsson_topology_2009}
{\sc Carlsson}. 2009.
\newblock Topology and data.
\newblock {\em Bulletin of the American Mathematical Society}, {\bf 46}(2),
  255--308.

\bibitem[\protect\citename{Cazals {\em et~al.\@}\relax,
  }2003]{cazals2003molecular}
{\sc Cazals, Chazal~\& Lewiner}. 2003.
\newblock Molecular shape analysis based upon the Morse-Smale complex and the
  Connolly function.
\newblock {\em Pages  351--360 of:} {\em Proceedings of the nineteenth annual
  symposium on Computational geometry}.

\bibitem[\protect\citename{Cazals \& Pouget, }2005]{cazals_differential_2005}
{\sc Cazals~\& Pouget}. 2005.
\newblock Differential topology and geometry of smooth embedded surfaces:
  selected topics.
\newblock {\em International Journal of Computational Geometry \&
  Applications}, {\bf 15}(05), 511--536.

\bibitem[\protect\citename{Chazal \& Soufflet, }2004]{chazal_stability_2004}
{\sc Chazal~\& Soufflet}. 2004.
\newblock Stability and {Finiteness} {Properties} of {Medial} {Axis} and
  {Skeleton}.
\newblock {\em Journal of Dynamical and Control Systems}, {\bf 10}(2),
  149--170.

\bibitem[\protect\citename{Chazal \& Lieutier, }2007]{chazal_stability_2007}
{\sc Chazal~\& Lieutier}. 2007.
\newblock Stability and {Computation} of {Topological} {Invariants} of {Solids}
  in {R}{\textasciicircum}n.
\newblock {\em Discrete \& Computational Geometry}, {\bf 37}(4), 601--617.
\newblock Number: 4.

\bibitem[\protect\citename{Chazal {\em et~al.\@}\relax,
  }2013]{chazal2013persistence}
{\sc Chazal, Guibas, Oudot~\& Skraba}. 2013.
\newblock Persistence-based clustering in Riemannian manifolds.
\newblock {\em Journal of the ACM (JACM)}, {\bf 60}(6), 1--38.

\bibitem[\protect\citename{Cheeger, }1991]{cheeger_critical_1991}
{\sc Cheeger}. 1991.
\newblock Critical points of distance functions and applications to geometry.
\newblock {\em Pages  1--38 of:} {\sc Bartolomeis~\& Tricerri} (eds), {\em
  Geometric {Topology}: {Recent} {Developments}},  vol. 1504.
\newblock Berlin, Heidelberg: Springer Berlin Heidelberg.
\newblock Series Title: Lecture Notes in Mathematics.

\bibitem[\protect\citename{Clarke, }1990]{clarke1990optimization}
{\sc Clarke}. 1990.
\newblock {\em Optimization and nonsmooth analysis}.
\newblock SIAM.

\bibitem[\protect\citename{Cohen-Steiner {\em et~al.\@}\relax,
  }2007]{cohen-steiner_stability_2007}
{\sc Cohen-Steiner, Edelsbrunner~\& Harer}. 2007.
\newblock Stability of {Persistence} {Diagrams}.
\newblock {\em Discrete \& Computational Geometry}, {\bf 37}(1), 103--120.

\bibitem[\protect\citename{Corvellec, }2001]{corvellec2001Second}
{\sc Corvellec}. 2001.
\newblock {On the Second Deformation Lemma}.
\newblock {\em Topological Methods in Nonlinear Analysis}, {\bf 17}(1), 55 --
  66.

\bibitem[\protect\citename{Corvellec {\em et~al.\@}\relax,
  }1993]{corvellec_deformation_1993}
{\sc Corvellec, Degiovanni~\& Marzocchi}. 1993.
\newblock Deformation properties for continuous functionals and critical point
  theory.
\newblock {\em Topological Methods in Nonlinear Analysis}, Mar., 151--171.

\bibitem[\protect\citename{Crawford {\em et~al.\@}\relax,
  }2020]{crawford2020predicting}
{\sc Crawford, Monod, Chen, Mukherjee~\& Rabad{\'a}n}. 2020.
\newblock Predicting clinical outcomes in glioblastoma: {A}n application of
  topological and functional data analysis.
\newblock {\em Journal of the American Statistical Association}, {\bf
  115}(531), 1139--1150.

\bibitem[\protect\citename{Crawley-Boevey, }2015]{crawley2015decomposition}
{\sc Crawley-Boevey}. 2015.
\newblock Decomposition of pointwise finite-dimensional persistence modules.
\newblock {\em Journal of Algebra and its Applications}, {\bf 14}(05), 1550066.

\bibitem[\protect\citename{Curry {\em et~al.\@}\relax,
  }2016]{curry2016discrete}
{\sc Curry, Ghrist~\& Nanda}. 2016.
\newblock Discrete Morse theory for computing cellular sheaf cohomology.
\newblock {\em Foundations of Computational Mathematics}, {\bf 16}, 875--897.

\bibitem[\protect\citename{Damon, }2006]{damon_global_2006}
{\sc Damon}. 2006.
\newblock The global medial structure of regions in {R3}.
\newblock {\em Geometry \& Topology}, {\bf 10}(4), 2385--2429.

\bibitem[\protect\citename{de~Kergorlay {\em et~al.\@}\relax,
  }2022]{de2022random}
{\sc Kergorlay, Tillmann~\& Vipond}. 2022.
\newblock Random {\v{C}}ech complexes on manifolds with boundary.
\newblock {\em Random Structures \& Algorithms}, {\bf 61}(2), 309--352.

\bibitem[\protect\citename{Delfour \& Zolésio, }2011]{delfour_shapes_2011}
{\sc Delfour~\& Zolésio}. 2011.
\newblock {\em Shapes and {Geometries}: {Metrics}, {Analysis}, {Differential}
  {Calculus}, and {Optimization}, {Second} {Edition}}. Second edn.
\newblock Society for Industrial and Applied Mathematics.

\bibitem[\protect\citename{Delgado-Friedrichs {\em et~al.\@}\relax,
  }2014]{delgado-friedrichs_morse_2014}
{\sc Delgado-Friedrichs, Robins~\& Sheppard}. 2014 (Oct.).
\newblock Morse theory and persistent homology for topological analysis of {3D}
  images of complex materials.
\newblock {\em Pages  4872--4876 of:} {\em 2014 {IEEE} {International}
  {Conference} on {Image} {Processing} ({ICIP})}.
\newblock ISSN: 2381-8549.

\bibitem[\protect\citename{Delgado-Friedrichs {\em et~al.\@}\relax,
  }2015]{delgado-friedrichs_skeletonization_2015}
{\sc Delgado-Friedrichs, Robins~\& Sheppard}. 2015.
\newblock Skeletonization and {Partitioning} of {Digital} {Images} {Using}
  {Discrete} {Morse} {Theory}.
\newblock {\em IEEE transactions on pattern analysis and machine intelligence},
  {\bf 37}(3), 654--666.

\bibitem[\protect\citename{Dlotko, }2021]{dlotko_cubical_2021}
{\sc Dlotko}. 2021.
\newblock {\em Cubical complex}.

\bibitem[\protect\citename{Edelsbrunner \& Harer,
  }2008]{edelsbrunner_persistent_2008}
{\sc Edelsbrunner~\& Harer}. 2008.
\newblock Persistent homology—a survey.
\newblock {\em Pages  257--282 of:} {\sc Goodman, Pach~\& Pollack} (eds), {\em
  Contemporary {Mathematics}},  vol. 453.
\newblock Providence, Rhode Island: American Mathematical Society.

\bibitem[\protect\citename{Edelsbrunner {\em et~al.\@}\relax,
  }2001]{edelsbrunner2001hierarchical}
{\sc Edelsbrunner, Harer~\& Zomorodian}. 2001.
\newblock Hierarchical Morse complexes for piecewise linear 2-manifolds.
\newblock {\em Pages  70--79 of:} {\em Proceedings of the seventeenth annual
  symposium on Computational geometry}.

\bibitem[\protect\citename{Federer, }1959]{federer1959curvature}
{\sc Federer}. 1959.
\newblock Curvature measures.
\newblock {\em Transactions of the American Mathematical Society}, {\bf 93}(3),
  418--491.

\bibitem[\protect\citename{Ferry, }1976]{ferry_when_1976}
{\sc Ferry}. 1976.
\newblock When eps-boundaries are manifolds.
\newblock {\em Fundamenta Mathematicae}, {\bf 90}(3), 199--210.

\bibitem[\protect\citename{Forman, }1998]{forman_morse_1998}
{\sc Forman}. 1998.
\newblock Morse {Theory} for {Cell} {Complexes}.
\newblock {\em Advances in Mathematics}, {\bf 134}(1), 90--145.

\bibitem[\protect\citename{Gabriel, }1972]{gabriel1972unzerlegbare}
{\sc Gabriel}. 1972.
\newblock Unzerlegbare darstellungen I.
\newblock {\em Manuscripta mathematica}, {\bf 6}, 71--103.

\bibitem[\protect\citename{Gaetan \& Guyon, }2010]{gaetan_spatial_2010}
{\sc Gaetan~\& Guyon}. 2010.
\newblock {\em Spatial {Statistics} and {Modeling}}.
\newblock Springer {Series} in {Statistics}.
\newblock New York, NY: Springer New York.

\bibitem[\protect\citename{Gershkovich, }1997]{gershkovich_singularity_1997}
{\sc Gershkovich}. 1997.
\newblock Singularity theory for {Riemannian} distance functions on
  non-positively curved surfaces.
\newblock {\em Page  117 of:} {\em Geometry from the {Pacific} {Rim}:
  {Proceedings} of the {Pacific} {Rim} {Geometry} {Conference} {Held} at
  {National} {University} of {Singapore}, {Republic} of {Singapore}, {December}
  12-17, 1994}.
\newblock De Gruyter.

\bibitem[\protect\citename{Gershkovich \& Rubinstein,
  }1997]{gershkovich_morse_1997}
{\sc Gershkovich~\& Rubinstein}. 1997.
\newblock Morse theory for {Min}-type functions.
\newblock {\em Asian Journal of Mathematics}, {\bf 1}(4), 696--715.

\bibitem[\protect\citename{Ghrist, }2008]{ghrist_barcodes_2008}
{\sc Ghrist}. 2008.
\newblock Barcodes: {The} persistent topology of data.
\newblock {\em Bulletin of the American Mathematical Society}, {\bf 45}(1),
  61--75.

\bibitem[\protect\citename{Ghrist, }2014]{ghrist2014elementary}
{\sc Ghrist}. 2014.
\newblock {\em Elementary Applied Topology}.
\newblock Createspace Seattle.

\bibitem[\protect\citename{Giblin, }2000]{giblin_symmetry_2000}
{\sc Giblin}. 2000.
\newblock Symmetry {Sets} and {Medial} {Axes} in {Two} and {Three}
  {Dimensions}.
\newblock {\em Pages  306--321 of:} {\sc Cipolla~\& Martin} (eds), {\em The
  {Mathematics} of {Surfaces} {IX}}.
\newblock London: Springer.

\bibitem[\protect\citename{Giesen \& John, }2008]{giesen_flow_2008}
{\sc Giesen~\& John}. 2008.
\newblock The flow complex: {A} data structure for geometric modeling.
\newblock {\em Computational Geometry}, {\bf 39}(3), 178--190.

\bibitem[\protect\citename{Giesen \& Kuehne, }2013]{giesen_parallel_2013}
{\sc Giesen~\& Kuehne}. 2013.
\newblock A parallel algorithm for computing the flow complex.
\newblock {\em Pages  57--66 of:} {\em Proceedings of the twenty-ninth annual
  symposium on {Computational} geometry}.
\newblock {SoCG} '13.
\newblock New York, NY, USA: Association for Computing Machinery.

\bibitem[\protect\citename{Giesen {\em et~al.\@}\relax,
  }2006]{giesen_medial_2006}
{\sc Giesen, Ramos~\& Sadri}. 2006.
\newblock Medial axis approximation and unstable flow complex.
\newblock {\em Pages  327--336 of:} {\em Proceedings of the twenty-second
  annual symposium on {Computational} geometry}.
\newblock {SCG} '06.
\newblock New York, NY, USA: Association for Computing Machinery.

\bibitem[\protect\citename{Gilbarg \& Trudinger, }1977]{gilbarg_elliptic_1977}
{\sc Gilbarg~\& Trudinger}. 1977.
\newblock {\em Elliptic partial differential equations of second order}.
\newblock Classics in mathematics.
\newblock Berlin Heidelberg: Springer.

\bibitem[\protect\citename{Goresky \& MacPherson,
  }1988]{goresky_stratified_1988}
{\sc Goresky~\& MacPherson}. 1988.
\newblock Stratified {Morse} {Theory}.
\newblock {\em Pages  3--22 of:} {\em Stratified {Morse} {Theory}}.
\newblock Berlin, Heidelberg: Springer Berlin Heidelberg.

\bibitem[\protect\citename{Grove, }1993]{grove_critical_1993}
{\sc Grove}. 1993.
\newblock Critical point theory for distance functions.
\newblock {\em In:} {\em Differential {Geometry}: {Riemannian} {Geometry},
  {Part} 3},  vol. 54.3.
\newblock Proceedings of Symposia in Pure Mathematics.

\bibitem[\protect\citename{Grove \& Shiohama, }1977]{grove_generalized_1977}
{\sc Grove~\& Shiohama}. 1977.
\newblock A {Generalized} {Sphere} {Theorem}.
\newblock {\em Annals of Mathematics}, {\bf 106}(1), 201--211.
\newblock Publisher: Annals of Mathematics.

\bibitem[\protect\citename{Guillemin \& Pollack,
  }2010]{guillemin_differential_2010}
{\sc Guillemin~\& Pollack}. 2010.
\newblock {\em Differential {Topology}}.
\newblock {AMS} {Chelsea} {Publishing}.
\newblock AMS Chelsea Pub.

\bibitem[\protect\citename{Harker {\em et~al.\@}\relax,
  }2014]{harker2014discrete}
{\sc Harker, Mischaikow, Mrozek~\& Nanda}. 2014.
\newblock Discrete Morse theoretic algorithms for computing homology of
  complexes and maps.
\newblock {\em Foundations of Computational Mathematics}, {\bf 14}, 151--184.

\bibitem[\protect\citename{Herring {\em et~al.\@}\relax,
  }2019]{herring_topological_2019}
{\sc Herring, Robins~\& Sheppard}. 2019.
\newblock Topological {Persistence} for {Relating} {Microstructure} and
  {Capillary} {Fluid} {Trapping} in {Sandstones}.
\newblock {\em Water Resources Research}, {\bf 55}(1), 555--573.

\bibitem[\protect\citename{Hirsch, }1976]{hirsch_differential_1976}
{\sc Hirsch}. 1976.
\newblock {\em Differential {Topology}}.
\newblock New York, NY : Springer New York.

\bibitem[\protect\citename{Hohenwarter {\em et~al.\@}\relax,
  }2013]{hohenwarter_geogebra_2013}
{\sc Hohenwarter, Borcherds, Ancsin, Bencze, Blossier, Delobelle, Denizet,
  \'Eli\'as, Fekete, G\'al, Kone\v{c}n\'y, Kov\'acs, Lizelfelner, Parisse~\&
  Sturr}. 2013 (Dec.).
\newblock {\em {GeoGebra} 4.4}.

\bibitem[\protect\citename{Hu {\em et~al.\@}\relax,
  }2019]{hu_topology-preserving_2019}
{\sc Hu, Li, Samaras~\& Chen}. 2019.
\newblock Topology-{Preserving} {Deep} {Image} {Segmentation}.
\newblock {\em In:} {\sc Wallach, Larochelle, Beygelzimer, Alché-Buc, Fox~\&
  Garnett} (eds), {\em Advances in {Neural} {Information} {Processing}
  {Systems}},  vol. 32.
\newblock Curran Associates, Inc.

\bibitem[\protect\citename{Hu {\em et~al.\@}\relax, }2022]{hu_learning_2022}
{\sc Hu, Samaras~\& Chen}. 2022 (Oct.).
\newblock {\em Learning {Probabilistic} {Topological} {Representations} {Using}
  {Discrete} {Morse} {Theory}}.
\newblock arXiv:2206.01742 [cs, eess].

\bibitem[\protect\citename{Itoh \& Sakai, }2007]{itoh_cut_2007}
{\sc Itoh~\& Sakai}. 2007.
\newblock Cut loci and distance functions.
\newblock {\em Math. J. Okayama Univ.},  65--92.

\bibitem[\protect\citename{Julesz, }1962]{julesz_visual_1962}
{\sc Julesz}. 1962.
\newblock Visual {Pattern} {Discrimination}.
\newblock {\em IRE Transactions on Information Theory}, {\bf 8}(2), 84--92.
\newblock Conference Name: IRE Transactions on Information Theory.

\bibitem[\protect\citename{Kaczynski {\em et~al.\@}\relax,
  }2004]{kaczynski_cubical_2004}
{\sc Kaczynski, Mischaikow~\& Mrozek}. 2004.
\newblock Cubical {Homology}.
\newblock {\em Pages  39--92 of:} {\sc Kaczynski, Mischaikow~\& Mrozek} (eds),
  {\em Computational {Homology}}.
\newblock Applied {Mathematical} {Sciences}.
\newblock New York, NY: Springer.

\bibitem[\protect\citename{Kaji {\em et~al.\@}\relax, }2020]{kaji_cubical_2020}
{\sc Kaji, Sudo~\& Ahara}. 2020 (June).
\newblock {\em Cubical {Ripser}: {Software} for computing persistent homology
  of image and volume data}.
\newblock Tech. rept. arXiv:2005.12692. arXiv.
\newblock arXiv:2005.12692 [cs, math] type: article.

\bibitem[\protect\citename{Kirby \& Siebenmann, }1977]{kirby_foundational_1977}
{\sc Kirby~\& Siebenmann}. 1977.
\newblock {\em Foundational {Essays} on {Topological} {Manifolds},
  {Smoothings}, and {Triangulations}. ({AM}-88)}.
\newblock Princeton University Press.

\bibitem[\protect\citename{Krantz \& Parks, }1981]{krantz_distance_1981}
{\sc Krantz~\& Parks}. 1981.
\newblock Distance to {Ck} hypersurfaces.
\newblock {\em Journal of Differential Equations}, {\bf 40}(1), 116--120.

\bibitem[\protect\citename{Landis \& Morse, }1975]{landis_tractions_1975}
{\sc Landis~\& Morse}. 1975.
\newblock Tractions in {Critical} {Point} {Theory}.
\newblock {\em The Rocky Mountain Journal of Mathematics}, {\bf 5}(3),
  379--399.
\newblock Publisher: Rocky Mountain Mathematics Consortium.

\bibitem[\protect\citename{Landy \& Graham, }2004]{landy_visual_2004}
{\sc Landy~\& Graham}. 2004.
\newblock Visual {Perception} of {Texture}.
\newblock {\em Page  1106 of:} {\sc Chalupa~\& Werner} (eds), {\em The visual
  neurosciences}.
\newblock MIT Press, Cambridge, Mass.

\bibitem[\protect\citename{Lee, }1982]{lee_medial_1982}
{\sc Lee}. 1982.
\newblock Medial {Axis} {Transformation} of a {Planar} {Shape}.
\newblock {\em IEEE Transactions on Pattern Analysis and Machine Intelligence},
  {\bf PAMI-4}(4), 363--369.
\newblock Conference Name: IEEE Transactions on Pattern Analysis and Machine
  Intelligence.

\bibitem[\protect\citename{Lee {\em et~al.\@}\relax,
  }2017]{lee_quantifying_2017}
{\sc Lee, Barthel, Dłotko, Moosavi, Hess~\& Smit}. 2017.
\newblock Quantifying similarity of pore-geometry in nanoporous materials.
\newblock {\em Nature Communications}, {\bf 8}(1), 15396.
\newblock Number: 1 Publisher: Nature Publishing Group.

\bibitem[\protect\citename{Lewiner, }2013]{lewiner_critical_2013}
{\sc Lewiner}. 2013.
\newblock Critical sets in discrete {Morse} theories: {Relating} {Forman} and
  piecewise-linear approaches.
\newblock {\em Computer Aided Geometric Design}, {\bf 30}(6), 609--621.

\bibitem[\protect\citename{Lieutier, }2003]{lieutier_any_2003}
{\sc Lieutier}. 2003.
\newblock Any open bounded subset of {Rn} has the same homotopy type than its
  medial axis.
\newblock {\em Pages  65--75 of:} {\em Proceedings of the eighth {ACM}
  symposium on {Solid} modeling and applications}.
\newblock {SM} '03.
\newblock New York, NY, USA: Association for Computing Machinery.

\bibitem[\protect\citename{Lima, }1988]{lima_jordan-brouwer_1988}
{\sc Lima}. 1988.
\newblock The {Jordan}-{Brouwer} {Separation} {Theorem} for {Smooth}
  {Hypersurfaces}.
\newblock {\em The American Mathematical Monthly}, {\bf 95}(1), 39--42.
\newblock Publisher: Mathematical Association of America.

\bibitem[\protect\citename{Lindquist {\em et~al.\@}\relax,
  }1996]{lindquist_medial_1996}
{\sc Lindquist, Lee, Coker, Jones~\& Spanne}. 1996.
\newblock Medial axis analysis of void structure in three-dimensional
  tomographic images of porous media.
\newblock {\em Journal of Geophysical Research: Solid Earth}, {\bf 101}(B4),
  8297--8310.
\newblock \_eprint: https://onlinelibrary.wiley.com/doi/pdf/10.1029/95JB03039.

\bibitem[\protect\citename{Mather, }1983]{mather_distance_1983}
{\sc Mather}. 1983.
\newblock Distance from a submanifold in {Euclidean} space.
\newblock {\em Proc. Sympos. Pure Math.}, {\bf 40}, 199--216.

\bibitem[\protect\citename{Matov, }1986]{matov_singularities_1986}
{\sc Matov}. 1986.
\newblock Singularities of the maximum function on a manifold with boundary.
\newblock {\em Journal of Soviet Mathematics}, {\bf 33}(4), 1103--1127.

\bibitem[\protect\citename{Mayost, }2014]{mayost_applications_2014}
{\sc Mayost}. 2014.
\newblock {\em Applications {Of} {The} {Signed} {Distance} {Function} {To}
  {Surface} {Geometry}}.
\newblock Thesis, University of Toronto.
\newblock Accepted: 2017-11-20T18:00:14Z.

\bibitem[\protect\citename{McGrath, }2016]{mcgrath_smooth_2016}
{\sc McGrath}. 2016.
\newblock On the {Smooth} {Jordan} {Brouwer} {Separation} {Theorem}.
\newblock {\em The American Mathematical Monthly}, {\bf 123}(3), 292--295.
\newblock Publisher: [Taylor \& Francis, Ltd., Mathematical Association of
  America].

\bibitem[\protect\citename{Milnor, }1963]{milnor_morse_1963}
{\sc Milnor}. 1963.
\newblock {\em Morse {Theory}}.
\newblock Princeton University Press.

\bibitem[\protect\citename{Moon {\em et~al.\@}\relax,
  }2019]{moon_statistical_2019}
{\sc Moon, Mitchell, Heath~\& Andrew}. 2019.
\newblock Statistical {Inference} {Over} {Persistent} {Homology} {Predicts}
  {Fluid} {Flow} in {Porous} {Media}.
\newblock {\em Water Resources Research}, {\bf 55}(11), 9592--9603.

\bibitem[\protect\citename{Morse, }1959]{morse_topologically_1959}
{\sc Morse}. 1959.
\newblock Topologically non-degenerate functions on a compact n-manifold {M}.
\newblock {\em Journal d’Analyse Mathématique}, {\bf 7}(1), 189--208.

\bibitem[\protect\citename{Morse, }1973]{morse_f_1973}
{\sc Morse}. 1973.
\newblock \textit{{F}} -{Deformations} and \textit{{F}} -{Tractions}.
\newblock {\em Proceedings of the National Academy of Sciences}, {\bf 70}(6),
  1634--1635.

\bibitem[\protect\citename{Müller {\em et~al.\@}\relax,
  }2022]{muller_gstools_2022}
{\sc Müller, Schüler, Zech~\& Heße}. 2022.
\newblock {GSTools} v1.3: a toolbox for geostatistical modelling in {Python}.
\newblock {\em Geoscientific Model Development}, {\bf 15}(7), 3161--3182.

\bibitem[\protect\citename{Niyogi {\em et~al.\@}\relax,
  }2008]{niyogi_finding_2008}
{\sc Niyogi, Smale~\& Weinberger}. 2008.
\newblock Finding the {Homology} of {Submanifolds} with {High} {Confidence}
  from {Random} {Samples}.
\newblock {\em Discrete \& Computational Geometry}, {\bf 39}(1), 419--441.

\bibitem[\protect\citename{Obayashi {\em et~al.\@}\relax,
  }2022]{obayashi_persistent_2022}
{\sc Obayashi, Nakamura~\& Hiraoka}. 2022.
\newblock Persistent {Homology} {Analysis} for {Materials} {Research} and
  {Persistent} {Homology} {Software}: {HomCloud}.
\newblock {\em Journal of the Physical Society of Japan}, {\bf 91}(9), 091013.

\bibitem[\protect\citename{Osher \& Fedkiw, }2003]{osher_signed_2003}
{\sc Osher~\& Fedkiw}. 2003.
\newblock Signed {Distance} {Functions}.
\newblock {\em Pages  17--22 of:} {\sc Marsden, Antman~\& Sirovich} (eds), {\em
  Level {Set} {Methods} and {Dynamic} {Implicit} {Surfaces}},  vol. 153.
\newblock New York, NY: Springer New York.
\newblock Series Title: Applied Mathematical Sciences.

\bibitem[\protect\citename{Park {\em et~al.\@}\relax, }2019]{park_deepsdf_2019}
{\sc Park, Florence, Straub, Newcombe~\& Lovegrove}. 2019 (June).
\newblock {DeepSDF}: {Learning} {Continuous} {Signed} {Distance} {Functions}
  for {Shape} {Representation}.
\newblock {\em Pages  165--174 of:} {\em 2019 {IEEE}/{CVF} {Conference} on
  {Computer} {Vision} and {Pattern} {Recognition} ({CVPR})}.
\newblock ISSN: 2575-7075.

\bibitem[\protect\citename{Portilla \& Simoncelli,
  }2000]{portilla_parametric_2000}
{\sc Portilla~\& Simoncelli}. 2000.
\newblock A {Parametric} {Texture} {Model} {Based} on {Joint} {Statistics} of
  {Complex} {Wavelet} {Coefficients}.
\newblock {\em International Journal of Computer Vision}, {\bf 40}(1), 49--70.

\bibitem[\protect\citename{Robins {\em et~al.\@}\relax,
  }2011]{robins_theory_2011}
{\sc Robins, Wood~\& Sheppard}. 2011.
\newblock Theory and {Algorithms} for {Constructing} {Discrete} {Morse}
  {Complexes} from {Grayscale} {Digital} {Images}.
\newblock {\em IEEE Transactions on Pattern Analysis and Machine Intelligence},
  {\bf 33}(8), 1646--1658.
\newblock Conference Name: IEEE Transactions on Pattern Analysis and Machine
  Intelligence.

\bibitem[\protect\citename{Saucan, }2020]{saucan_discrete_2020}
{\sc Saucan}. 2020 (Mar.).
\newblock {\em Discrete {Morse} {Theory}, {Persistent} {Homology} and
  {Forman}-{Ricci} {Curvature}}.
\newblock arXiv:2003.03844 [math].

\bibitem[\protect\citename{Siebenmann, }1972]{siebenmann_deformation_1972}
{\sc Siebenmann}. 1972.
\newblock Deformation of homeomorphisms on stratified sets.
\newblock {\em Commentarii Mathematici Helvetici}, {\bf 47}(1), 123--163.

\bibitem[\protect\citename{Sigg {\em et~al.\@}\relax, }2003]{sigg_signed_2003}
{\sc Sigg, Peikert~\& Gross}. 2003 (Oct.).
\newblock Signed distance transform using graphics hardware.
\newblock {\em Pages  83--90 of:} {\em {IEEE} {Visualization}, 2003. {VIS}
  2003.}

\bibitem[\protect\citename{Silvester, }2000]{silvester_determinants_2000}
{\sc Silvester}. 2000.
\newblock Determinants of {Block} {Matrices}.
\newblock {\em The Mathematical Gazette}, {\bf 84}(501), 460--467.
\newblock Publisher: Mathematical Association.

\bibitem[\protect\citename{Song, }2022]{song_generation_2022}
{\sc Song}. 2022.
\newblock Generation of {Tubular} and {Membranous} {Shape} {Textures} with
  {Curvature} {Functionals}.
\newblock {\em Journal of Mathematical Imaging and Vision}, {\bf 64}(1),
  17--40.

\bibitem[\protect\citename{Stolz {\em et~al.\@}\relax,
  }2022]{stolz_multiscale_2022}
{\sc Stolz, Kaeppler, Markelc, Braun, Lipsmeier, Muschel, Byrne~\& Harrington}.
  2022.
\newblock Multiscale topology characterizes dynamic tumor vascular networks.
\newblock {\em Science Advances}, {\bf 8}(23), eabm2456.
\newblock Publisher: American Association for the Advancement of Science.

\bibitem[\protect\citename{Tauzin {\em et~al.\@}\relax,
  }2022]{tauzin_giotto-tda_2022}
{\sc Tauzin, Lupo, Tunstall, Pérez, Caorsi, Medina-Mardones, Dassatti~\&
  Hess}. 2022.
\newblock giotto-tda: a topological data analysis toolkit for machine learning
  and data exploration.
\newblock {\em The Journal of Machine Learning Research}, {\bf 22}(1),
  39:1834--39:1839.

\bibitem[\protect\citename{{The GUDHI Project},
  }2015]{the_gudhi_project_gudhi_2015}
{\sc {The GUDHI Project}}. 2015.
\newblock {\em {GUDHI} {User} and {Reference} {Manual}}.
\newblock GUDHI Editorial Board.

\bibitem[\protect\citename{Tierny \& Pascucci, }2012]{tierny2012generalized}
{\sc Tierny~\& Pascucci}. 2012.
\newblock Generalized topological simplification of scalar fields on surfaces.
\newblock {\em IEEE transactions on visualization and computer graphics}, {\bf
  18}(12), 2005--2013.

\bibitem[\protect\citename{Wagner {\em et~al.\@}\relax,
  }2012]{peikert_efficient_2012}
{\sc Wagner, Chen~\& Vuçini}. 2012.
\newblock Efficient {Computation} of {Persistent} {Homology} for {Cubical}
  {Data}.
\newblock {\em Pages  91--106 of:} {\sc Peikert, Hauser, Carr~\& Fuchs} (eds),
  {\em Topological {Methods} in {Data} {Analysis} and {Visualization} {II}}.
\newblock Berlin, Heidelberg: Springer Berlin Heidelberg.
\newblock Series Title: Mathematics and Visualization.

\bibitem[\protect\citename{Yomdin, }1981]{yomdin_local_1981}
{\sc Yomdin}. 1981.
\newblock On the local structure of a generic central set.
\newblock {\em Compositio Mathematica}, ~15.

\bibitem[\protect\citename{Zomorodian \& Carlsson,
  }2004]{zomorodian2004computing}
{\sc Zomorodian~\& Carlsson}. 2004.
\newblock Computing persistent homology.
\newblock {\em Pages  347--356 of:} {\em Proceedings of the twentieth annual
  symposium on Computational geometry}.

\end{thebibliography}
	}

	\appendix
	\section*{Appendices}
	\addcontentsline{toc}{section}{Appendices}
	\renewcommand{\thesubsection}{\Alph{subsection}}
	\label{sec:Appendix}
	
	\subsection{Technical Lemmas for the Genericity Theorem}

	In this Appendix, we present three important components that were used to prove \Cref{thm:generic}.
	As before, we make use of simplified notations for the Gauss map (see Remark \ref{rmk:gauss_map_notation}). 
	
	\begin{lem}[The functions $F$ are submersions]
		\label{lem:submersion}
		The functions $F$ introduced in the proof of \Cref{thm:generic} are sufficiently smooth submersions. More precisely,
		\begin{itemize}
			\item
			Let $X = M^{(2)} \times \R_{>0}$ and $Y = (\mathbb{S}^2)^2 \times (\R^3)^2$. Define 
			\begin{equation}
				F:
				\begin{array}{ccc}
					\Emb \times X & \to & Y \\
					(\io, m_1,m_2,r) & \mapsto & (n_1, n_2, p_1 + r \, n_1, p_2 + r \, n_2 )
				\end{array}
			\end{equation}
			where $p_i = \io(m_i)$ and $n_i = \nbf(p_i)$. In the expression $p_i + r \, n_i$, $n_i$ is viewed in $\R^3$.
			
			Then $F$ is a $C^{k-1}$ submersion.
			
			\item Likewise, define a function $F(\io, m_1,m_2,m_3,r) = (n_1, n_2, n_3, p_1 + r \, n_1, p_2 + r \, n_2 , p_3 + r \, n_3)$.
			Then $F$ is a $C^{k-1}$ submersion.
			
			\item Let $X = M^{(2)} \times \R_{>0}$ and $Y = \Xi \times \mathbb{S}^2 \times \R^3 \times \R^3$. Here, $(\Xi, \mathbb{S}^2, \pi, \mathrm{Sym}((\R \, n)^\perp))$ is a fiber bundle of base space $\mathbb{S}^2$ and fiber space $\mathrm{Sym}((\R \, n)^\perp)$, the space of symmetric endomorphisms defined on the subspace of $\R^3$ orthogonal to $n \in \mathbb{S}^2$.
			Define the map
			\begin{equation}
				F:
				\begin{array}{ccc}
					\Emb \times X & \to & Y \\
					(\io, m_1,m_2,r) & \mapsto & (n_1, \Id + r \, d_{p_1}\nbf, n_2, p_1 + r \, n_1, p_2 + r \, n_2 )
				\end{array}
			\end{equation}
			where $p_i = \io(m_i)$ and $n_i = \nbf(p_i)$.
			
			Then $F$ is a $C^{k-2}$ submersion.
		\end{itemize}
	\end{lem}
	\begin{proof}
		We show first that the following simplified function is a submersion; in fact, $D_\io F$ alone will already be surjective. Consider the fiber bundle $\Xi$ having base space $\mathbb{S}^2$ and fiber space $\mathrm{Sym}((\R \, n)^\perp)$.
		
		\begin{equation}
			F:
			\begin{array}{ccc}
				\Emb \times M & \to & \R^3 \times \Xi \\
				(\io, m) & \mapsto & (p, n, dn),
			\end{array}
		\end{equation}
		where $p = \io(m)$, $n = \nbf(p)$, $dn =  d_{p}\nbf$.
		
		First, $F$ is a $C^{k-2}$ map.
		Let us fix $\io_0$ and $m_0$.
		Now, given some arbitrary element $(v, N, A)$ tangential to $\R^3 \times \Xi$ at $(p_0,n_0, dn_0) = F(\io_0,m_0)$, where $v \in T_{p_0} \R^3 \simeq \R^3$, $N \in T_{n_0} \mathbb{S}^2$, $A \in T_{dn_0} \mathrm{Sym}((\R \, n_0)^\perp) \simeq \mathrm{Sym}((\R \, n_0)^\perp)$, the aim is to find some tangential element $V \in T_\io \Emb \simeq C^k(M,\R^3)$ such that $D_\io F(\io_0,m_0) [V] = (v,N,A)$, by local perturbations of the embedding. We can decompose this into three independent problems, by finding preimages of $(v,0,0)$, $(0,N,0),$ and $(0,0,A)$. 
		
		Set $p(t) = p_0 + t \, v$: we have $p(0) = p_0$ and $p'(0) = v$. We want to find a path $\io(t)$ such that $\io(0) = \io_0$, $\io'(t) = V$, and satisfying $F(\io(t),m_0) = (p(t),0,0)$.
		Let $U$ be some neighborhood of $m_0$ in $M$ and $K \subset U$ a compact subset. There exists a bump function $g \in C^k(M,\R)$ such that $g \equiv 1$ in $K$ and $g \equiv 0$ outside $U$.
		Apply a small translation in the direction of $v \in \R^3$ on $\io$ in the neighborhood of $m_0$, by introducing $\io(t) = \io_0 + t \, g \, v$, which for small $t$ belongs to $\Emb$ since it is open in $C^k(M,\R^3)$. Then $F(\io(t),m_0) = (p(t),0,0)$, and the first problem is solved by taking $V = g \, v \in C^k(M,\R^3)$.
		
		There exists a path $n(t)$ in the sphere such that $n(0) = n_0$ and $n'(0) = N$, which defines a rotation in $\R^3$, denoted by $R_t$, whose axis is $(n_0,N)^\perp$ and angle $t \, \|N\|$. Applying a small such rotation in the neighborhood of $m_0$, we define $\io(t) = (1-g) \, \io_0 + g \, R_t(\io_0)$, which brings the normal $n_0$ to $n(t)$ while fixing $p_0$ and $dn_0$ ($g$ is the same as in the previous paragraph). Thus, $F(\io(t),m_0) = (0,n(t),0)$, which solves the second problem.
		
		For the third problem, we can work with principal coordinate systems. There exists a $C^k$ function $h_0 : T_{p_0} \Surf \, \inter \, \mathcal{U} \to \R$ such that $\Surf$ is locally the graph of $h_0$, where $\mathcal{U}$ is an open neighborhood of $p$, and the embedding is described by $\io_0(m) = p = (a,b,h_0(a,b)) \in \Surf$ where $a,b \in T_{p_0} \Surf \inter \mathcal{U}$. We may assume that $\partial_a h = \partial_b h = 0$ at $(0,0)$, so that the normal is $(0,0,1)$, and that the principal directions at $p_0$ are $(1,0,0)$ and $(0,1,0)$, with $\partial_{a,b} h = \partial_{b,a} h = 0$ and $\partial^2_{a,a} h = \kap_1$, $\partial^2_{b,b} h = \kap_2$. Define a bump function (still denoted by $g$) relative to the neighborhood $T_{p_0} \Surf \, \inter \, \mathcal{U}$, which is in correspondance to some neighborhood $U$ of $m_0$ in $M$.
		
		Now, the tangential element $A$ is represented as a $2 \times 2$ symmetric matrix in the chosen basis of $T_{p_0} \Surf$ and has the form $A = R_\theta \, \mathrm{diag}(\lambda, \mu) \, R_\theta^{-1}$ with $\lambda, \mu \in \R$. Define $\phi : T_{p_0} \Surf \to \R$ such that $\phi(a,b) = \frac{\lambda}{2} \, a^2 + \frac{\mu}{2} \, b^2$, whose graph is a surface with curvatures $\lambda$ and $\mu$ at $0$. Set $h_t = h_0 + t\, g \, \phi \circ R_\theta$, and finally $\io_t(m) = (a(m),\, b(m),\, h_t(a(m),\, b(m)))$ if $m \in U$, $\io_t(m) = \io_0(m)$ elsewhere on $M$. Then one can show that the perturbed surface fixes $p_0$ and $n_0$, but that its shape operator is $dn_t = dn_0 + t \, A$, which is what we wanted.
		
		To conclude, one can see that the proofs for the various forms of $F$ are similar. If $m_1$ and $m_2$ are distinct points in $M$, one can choose non-intersecting neighborhoods and respective bump functions to apply local perturbations to the embedding.
	\end{proof}

	\begin{lem}[Shape operator of offset level surfaces]
		\label{lem:shape_operators_levels}
		Let $q$ be a point in $\R^3 \setminus \Surf$ and $p$ one contact point with normal $n$ and at distance $r$ such that $q$ satisfies the strict ball condition  \eqref{eq:strict_ball_condition} at $p$. We have $q = p + r\,n$. Let $\alpha$ be the distance function to an associated contact piece.
		Then the shape operator at $q$ of the offset surface $\{\alpha = r\}$ is
		\begin{equation}
			\label{eq:shape_operators_levels}
			d_q \nbf_{\alpha = r} = d_{p} \nbf \, (\Id + r \, d_{p} \nbf )^{-1}.
		\end{equation}
	\end{lem}
	\begin{proof}
		Suppose first that $p$ is not an umbilical point, i.e., $\Surf$ admits two distinct principal curvatures there. A remarkable property of the principal directions $e_1(\rho), e_2(\rho)$ of the offset surfaces $\{\alpha = \rho\}$ for $\rho \in [0,r]$ is that the orthonormal frame $(n, e_1, e_2)$ is unchanged when travelling in the normal direction from $p$ to $q$ \citep{mayost_applications_2014}. On the other hand, the principal curvatures obey the following differential equation in the normal direction:
		\begin{align*}
			D_n(\kap_1) &= \kap_1^2, \\
			D_n(\kap_2) &= \kap_2^2.
		\end{align*}
		If $\kap_i(0) = 0$ at the contact point $p$, then $\kap_i(\rho) \cong 0$ for any offset, and \cref{eq:shape_operators_levels} holds. If $\kap_i(0) \neq 0$, by assumption $\kap_i(0) \neq r$ at $p$, then $\kap_i(\rho) = \frac{1}{\frac{1}{\kap_i(0)} - \rho}$ for $\rho \in [0,r]$. In particular,
		in the basis $(n, e_1, e_2)$,
		\[d_q \nbf_{\alpha = r} = \Hess \, \alpha (q) = 
		\begin{pmatrix}
			0 & 0 & 0 \\
			0 & -\kap_1(r) & 0 \\
			0 & 0 & -\kap_2(r) 
		\end{pmatrix} = d_{p} \nbf \, (\Id + r \, d_{p} \nbf )^{-1}.\]
		Even if $p$ is umbilical, all offset surfaces $\{\alpha = \rho\}$ are umbilical at $p + \rho \, n$ and we can choose consistent frames so that \cref{eq:shape_operators_levels} still holds.
	\end{proof}
	
	\begin{lem}[Riemannian Hessian of the restricted distance function]
		\label{lem:riem_hess}
		Let $\io$ be a generic embedding in the sense of \Cref{thm:generic_cut_locus}. Let $q$ be a critical point of the pure distance $\dist(\cdot, \Surf)$, with value $r$, contact points $\Gamma(q) = \{p_1,\ldots,p_m\}$ and normals $n_1,\ldots,n_m$ at these contacts. Denote by $\{\alpha_1,\ldots,\alpha_m\}$ the distance functions to some associated contact pieces, and by $g = \dist(\cdot, \Surf)_{|G(q)}$ the restriction of this function to $G(q)$ (\Cref{def:G(x)}).
		\begin{itemize}
			\item If $q$ is of type $A_1^2$, then on the $A_1^2$ sheet stratum $G(q)$ orthogonal to $n_1 = - n_2$,
			\begin{align}
				\Hess^{\mathrm{Riem}} g(q) &= \frac{1}{2} (d_q \nbf_{\alpha_1 = r} + d_q \nbf_{\alpha_2 = r}) \\
				&= \frac{1}{2} \left( d_{p_1} \nbf \, (\Id + r \, d_{p_1} \nbf )^{-1} + d_{p_2} \nbf \, (\Id + r \, d_{p_2} \nbf )^{-1} \right),
			\end{align}
			where $d_q \nbf_{\alpha_1 = r}$ is the shape operator of the submanifolds $\{\alpha_i = r\}$.
			\item If $q$ is of type $A_1^3$, then the $A_1^3$ curve stratum $G(q)$ is directed by some unit vector $t$ orthogonal to the \textit{plane} spanned by $n_1,n_2,n_3$. Then,
			\begin{equation*}
				\Hess^{\mathrm{Riem}} g(q) \text{ is non-degenerate } \Leftrightarrow t_{2,3} \cdot u_1(t) + t_{3,1} \cdot u_2(t) + t_{1,2} \cdot u_3(t) \neq 0
			\end{equation*}
			where $t_{i,j} = n_i \times n_j$ and $u_i = d_{p_i} \nbf \circ (\Id + r \, d_{p_i} \nbf)^{-1} = d_q \nbf_{\alpha_i = r}$.
		\end{itemize}
	\end{lem}
	\begin{proof}
		
		We apply an extrinsic formula for the Riemannian Hessian of a smooth function defined on a submanifold $\mathcal{M}$ of $\R^3$, given by \cite{absil_extrinsic_2013}. If $f : \mathcal{M} \to \R$ is defined on the submanifold and $\bar{f} : \R^3 \to \R$ extends $f$ to a neighborhood of $\mathcal{M}$, then for $p \in \mathcal{M}$ and $v \in T_p \mathcal{M}$,
		\[\Hess^{\mathrm{Riem}} f(p)[v] = \Pi_p \, \Hess \bar{f}(p)[v] + W_p(v, \Pi_p^\perp \grad \bar{f}(p) ),\]
		where the Weingarten map $W_p$ of the submanifold $\mathcal{M}$ at $p$ is the operator that associates to $v \in T_p \mathcal{M}$ and $ \mathbf{N}_0 \in T_p^\perp \mathcal{M}$ the value $W_p(v, \mathbf{N}_0) = - \Pi_p D_v \mathbf{N}$ for any local extension $\mathbf{N}$ of $\mathbf{N}_0$ to a normal vector field on $\mathcal{M}$. Here, $D_v \mathbf{N}$ designates the directional derivative of $\mathbf{N}$ along the direction $v$, $\Pi_p$ is the projection onto the tangent space $T_p \mathcal{M}$, and $\Pi_p^\perp$ the projection onto the normal space $T_p^\perp \mathcal{M}$.
		$\Hess \, \bar{f}(p)$ and $\grad \bar{f} (p)$ are the usual Hessian and gradient of a scalar function defined in $\R^3$.
		
		Reformulating the previous equality, we obtain
		\[\Hess^{\mathrm{Riem}} f(p)[v] = \Pi_p \, \Hess \bar{f}(p)[v] - 
		\Pi_p D_v (\Pi_{(\cdot)}^\perp \grad \bar{f}(\cdot)).\]
		We can now express the Riemannian Hessian at $q$ of the function $\dist(\cdot, \Surf)_{|G(q)}$ defined on the (germ of) submanifold $G(q) = \{\alpha_1 = \ldots = \alpha_m = \dist(\cdot, \Surf)\}$. 
		
		Here, the extension to consider is any function $\alpha_i$, for which we take advantage of well-known properties of distance functions.
		In particular, $\Hess \, \alpha_i(q) [n_i] = 0$ and $\grad \alpha_i = n_i \in T_q^\perp G(q)$.  Observe that $\Hess \, \alpha_i(q)$ is just the usual shape operator $d_q \nbf_{\alpha_i = r} : T_q \{\alpha_i = r\} \to T_q \{\alpha_i = r\}$ of the surface defined by $\{\alpha_i = r\}$, with orientation given by $n_i$ at $q$.
		Also, let $\nbf^i_G$ denote a local extension of $\grad \alpha_i$ to a normal vector field on the stratum $G(q)$. 
		
		In particular, if $G(q)$ is a $A_1^2$ sheet, $\nbf^i_G$ is the Gauss map with the same orientation as $n_i$ at $q$, so that $d_q \nbf^i_G : T_q G(q) \to T_q G(q)$ will just be the usual shape operator, but this time of the submanifold $G(q)$. The operators  $d_q \nbf_{\alpha_i = r}$ and $d_q \nbf^i_G$ are defined on the same tangent space: $T_q \{\alpha_i = r\} = T_q G(q)$.
		If $G(q)$ is a $A_1^3$ curve, then its tangent direction $t$ is contained in any $T_q \{\alpha_i = r\} \supsetneq T_q G(q) = \R\, t$.
		
		\textbf{$q$ is of Type $A_1^2$.}
		In this case, we get
		\begin{align*}
			\Hess^{\mathrm{Riem}} g(q)[v] &= \Hess \, \alpha_1(q)[v] - D_v(\grad \alpha_{1|G(q)})) \\
			&= d_q \nbf_{\alpha_i = r}[v] - d_q \nbf^i_G[v].
		\end{align*}
		Since $n_1 = -n_2$, we can choose $\nbf^1_G = -\nbf^2_G$, to obtain
		\begin{align*}
			\Hess^{\mathrm{Riem}} g(q) &= d_q \nbf_{\alpha_1 = r} - d_q \nbf^1_G \\
			&= d_q \nbf_{\alpha_2 = r} - d_q \nbf^2_G \\
			&= \frac{d_q \nbf_{\alpha_1 = r} + d_q \nbf_{\alpha_2 = r}}{2}.
		\end{align*}
		Using that $d_q \nbf_{\alpha_1 = r} = d_{p_i} \nbf \, (\Id + r \, d_{p_i} \nbf )^{-1}$ thanks to \Cref{lem:shape_operators_levels}, we establish the first case.
		
		\textbf{$q$ is of Type $A_1^3$.}
		Now, the expression involves a projection on $\R \, t$. We use the concise notation $u_i(t) = d_q \nbf_{\alpha_i = r}[t]$. Applying the formula above, we get
		\begin{align*}
			h = \Hess^{\mathrm{Riem}} g(q)[t] \cdot t &= t \cdot (u_1(t) - d_q \nbf^1_G[t]) \\
			&= t \cdot(u_2(t) - d_q \nbf^2_G[t]) \\
			&= t \cdot(u_3(t) - d_q \nbf^3_G[t]) \in \R.
		\end{align*} 
		Two of the vectors $n_i$, say $n_1$ and $n_2$, form a basis of the plane spanned by $n_1,n_2,n_3$. In that basis, $n_3 = \lambda \, n_1 + \mu \, n_2$. Importantly, we can choose the local extensions in such a way that $\nbf^3_G = \lambda \, \nbf^1_G + \mu \, \nbf^2_G$, hence $d_q \nbf^3_G[t] = \lambda \, d_q \nbf^1_G[t] + \mu \, d_q \nbf^2_G[t]$.
		
		After summing the three lines with the weights $-\lambda, -\mu, 1$ respectively, the second terms of each line cancel out and we obtain
		\[(1 - \lambda - \mu) \, h = t \cdot \left(u_3(t) - \lambda \, u_1(t) - \mu \, u_2(t) \right).\]
		
		By some trigonometric computations, one can see that
		\[\lambda = \frac{\sin \theta_{2,3}}{\sin \theta_{2,1}} \quad \mu = \frac{\sin \theta_{1,3}}{\sin \theta_{1,2}}\]
		where $\theta_{i,j} = (\widehat{n_i,n_j})$ is the algebraic angle in the plane orthogonal to $t$, with sign convention given by $t$. This way, $t_{i,j} = n_i \times n_j = \sin \theta_{i,j} \, t$. Also, $1 \neq \lambda + \mu$ because $n_3 \neq n_1$ and $n_3 \neq n_2$. 
		Therefore, we are led to the desired inequality:
		\begin{align*}
			h \neq 0 & \Leftrightarrow t \cdot \left(u_3(t) - \frac{\sin \theta_{2,3}}{\sin \theta_{2,1}}  \, u_1(t) - \frac{\sin \theta_{1,3}}{\sin \theta_{1,2}} \, u_2(t) \right) \neq 0 \\
			& \Leftrightarrow t \cdot \left( \sin \theta_{1,2} \, u_3(t) + \sin \theta_{2,3} \, u_1(t) + \sin \theta_{3,1} \, u_2(t) \right) \neq 0 \\
			& \Leftrightarrow t_{1,2} \cdot u_1(t) + t_{2,3} \cdot u_2(t) + t_{3,1} \cdot u_3(t) \neq 0.
		\end{align*}
		
	\end{proof}

\end{document}